\documentclass[draft,12pt]{article}

\usepackage[latin1]{inputenc}
\usepackage{amsmath,amssymb}
\usepackage{latexsym}

\newtheorem{theorem}{Theorem}[section]
\newtheorem{corollary}[theorem]{Corollary}
\newtheorem{lemma}[theorem]{Lemma}
\newtheorem{proposition}[theorem]{Proposition}
\newtheorem{definition}[theorem]{Definition}

\newtheorem{remark}[theorem]{Remark}

\numberwithin{equation}{section}

\def\square{{\vcenter{\vbox{\hrule height.3pt
        \hbox{\vrule width.3pt height5pt \kern5pt
           \vrule width.3pt}
        \hrule height.3pt}}}}

  \def\sF {{\cal F}}
  
  \def\sL {{\cal L}}
\def\sM {{\cal M}}  
  
\def\sS {{\cal S}} \def\sT {{\cal T}}

 \def\bN {{\mathbb N}} 
\def\bP {{\mathbb P}}  \def\bR {{\mathbb R}}

\def\wt{\widetilde}
\def\wh{\widehat}
\def\ol{\overline}
\def\E{{\mathbb E}}
\def\P{{\mathbb P}}
\def\norm#1{{\Vert #1 \Vert}}
\def\del{{\partial}}
\def\lam{{\lambda}}
\def\angel#1{{\langle #1 \rangle}}
\def\bee{\begin{equation}}
\def\bet{\begin{theorem}}
\def\bep{\begin{proposition}}
\def\bel{\begin{lemma}}
\def\bec{\begin{corollary}}
\def\bed{\begin{definition}}
\def\ber{\begin{remark}}
\def\eee{\end{equation}}
\def\eet{\end{theorem}}
\def\eep{\end{proposition}}
\def\eel{\end{lemma}}
\def\eec{\end{corollary}}
\def\eed{\end{definition}}
\def\eer{\end{remark}}

\def\R{{\mathbb R}}
\def\E{{{\mathbb E}\,}}
\def\P{{\mathbb P}}

\def\Z{{\mathbb Z}}

\def\lam{{\lambda}}

\def\al{{\alpha}}

\def\eps{\varepsilon}
\def\vp{\varphi}

\def\angel#1{{\langle#1\rangle}}
\def\norm#1{\Vert #1 \Vert}

\def\q {\quad} \def\qq {\qquad}
\def\del{{\partial}}
\def\wt{\widetilde}
\def\ol{\overline}

\def\wh{\widehat}

\def\ni{\noindent }
\def\ms{\medskip}
\def\bs{\bigskip}

\def\vep{\varepsilon}

\def\square{{\vcenter{\vbox{\hrule height.3pt
        \hbox{\vrule width.3pt height5pt \kern5pt
           \vrule width.3pt}
        \hrule height.3pt}}}}

\def\tfrac#1#2{{\textstyle {\frac{#1}{#2}}}}

\def\tlint{{- \kern-0.85em \int \kern-0.2em}}  
\def\dlint{{- \kern-1.05em \int \kern-0.4em}}  

  \def\sF {{\cal F}}
  
  \def\sL {{\cal L}}
\def\sM {{\cal M}}  
  
\def\sS {{\cal S}} \def\sT {{\cal T}}

 \def\bN {{\mathbb N}} 
\def\bP {{\mathbb P}}  \def\bR {{\mathbb R}}

\def\nn{{\nonumber}}

\begin{document}

\title{On Uniqueness in Law for Parabolic SPDEs and Infinite-dimensional SDEs}
\author{Richard F. Bass\footnote{Research partially supported by NSF grant
DMS-0901505.}  \; and Edwin Perkins\footnote{Research partially supported
by an NSERC Discovery Grant.}}

\date{\today}

\maketitle

\ms

\begin{center}
\emph{Dedicated to the memory of our friend and colleague\\
Michael ``Miki'' Neumann.}
\end{center}

\ms

\begin{abstract}  
\noindent {\it Abstract:} We give a sufficient conditions for uniqueness in
law for the stochastic partial differential equation
$$\frac{\del u}{\del t}(x,t)
=\tfrac12 \frac{\del^2 u}{\del x^2}(x,t)+A(u(\cdot,t)) \dot W_{x,t},$$
where $A$ is an operator mapping $C[0,1]$ into itself and $\dot W$ is
a space-time white noise. The approach is to first prove uniqueness
for the martingale problem for the operator
$$\sL f(x)=\sum_{i,j=1}^\infty a_{ij}(x) \frac{\del^2 f}{\del x^2}(x)
-\sum_{i=1}^\infty \lam_i x_i \frac{\del f}{\del x_i}(x),$$
where $\lam_i=ci^2$ and the $a_{ij}$ is a positive definite bounded
operator in Toeplitz form.

\end{abstract}

\section{Introduction}

Our goal is to obtain  a uniqueness in law result for parabolic stochastic partial differential equations (SPDEs) of the form
\bee\label{SPDEeq0}
\frac{\partial u}{\partial t}=\tfrac{1}{2} \frac{\del^2 u}{\del x^2}(x,t)+A(u(\cdot,t))(x)\dot W_{x,t},\qq  x\in[0,1],\ t\ge 0,
\eee
where $\dot W$ is a space-time white noise on $[0,1]\times[0,\infty)$, suitable boundary conditions are imposed at $0$ and $1$, and $A$ is an appropriate operator from $C[0,1]$ to $C[0,1]$ which is bounded above and away from zero.  A common approach to \eqref{SPDEeq0} (see, e.g., Chapter 3 of Walsh \cite{walsh}) is to convert it to a Hilbert space-valued stochastic differential equation (SDE)
 by setting
 \[X^j(t)=\langle u_t,e_j\rangle,\]
where $\{e_j\}$ is a complete orthonormal sequence of eigenfunctions for 
the Laplacian  (with the above boundary conditions) on $L^2[0,1]$ with eigenvalues $\{-\lambda_j\}$, $u_t(\cdot)=u(\cdot,t)$, and $\langle \cdot,\cdot\rangle$ is the usual inner product on $L^2[0,1]$.  This will convert the SPDE \eqref{SPDEeq0} to the $\ell^2$-valued SDE 
\bee\label{SDEeq0}
dX^j(t)=-\lambda_jX^j(t)dt+\sum_k\sigma_{jk}(X_t)d W^k_t,
\eee
where $\{ W^j\}$ are i.i.d.\ one-dimensional Brownian motions, $\sigma(x)=\sqrt{a(x)}$, $\sL_+(\ell^2, \ell^2)$
is the space of positive definite bounded self-adjoint operators on $\ell^2$,
and $a:\ell^2\to\sL_+(\ell^2,\ell^2)$ 
 is easily defined in terms of $A$ (see \eqref{adef0} below).  \eqref{SDEeq0} has been studied extensively (see, for example, Chapters 4 and 5 of Kallianpur and Xiong \cite{KX} or Chapters I and II of Da Prato and Zabczyk \cite{DZ}) but, as discussed in the introduction of Zambotti \cite{Zamb}, we are still far away from any uniqueness theory that would allow us to characterize solutions to \eqref{SPDEeq0}, except of course in the classical Lipschitz setting.

There has been some interesting work on Stroock-Varadhan type uniqueness results for 
equations such as \eqref{SDEeq0}.  These focus on Schauder estimates, that is,
 smoothing
properties of the resolvent, for the constant coefficient case which correspond to infinite-dimensional Ornstein-Uhlenbeck processes, and produce uniqueness under appropriate H\"older continuity conditions on $a$.  For example Zambotti \cite{Zamb} and Athreya, Bass, Gordina and Perkins \cite{ABGP} consider the above equation and Cannarsa and Da Prato \cite{CD} considers the slightly different setting where there is no restorative drift but (necessarily) a trace class condition on the driving noise.  Cannarsa and Da Prato \cite{CD} and Zambotti \cite{Zamb} use clever interpolation arguments to derive their Schauder estimates.  However, none of the above results appear to allow one to establish uniqueness in equations arising from the SPDE \eqref{SPDEeq0}.  In \cite{Zamb} $a$ is assumed to be a small trace class perturbation of a constant operator (see (9) and (10) of that reference) and in \cite{CD} the coefficient of the noise is essentially a H\"older continuous trace class perturbation of the identity.  If we take $e_j(y)=\exp(2\pi ijy)$, $j\in\Z$ (periodic boundary conditions) and $\lambda_j=2\pi^2j^2$, then it is not hard to see that in terms of these coordinates the corresponding operator $a=(a_{jk})$ associated with the SPDE \eqref{SPDEeq0} is 
\bee\label{adef0} a_{jk}(x)=\int_0^1A(u(x))(y)^2 e^{2\pi i(j-k)y}\,dy,\qq j,k\in\Z,
\eee
where $u=\sum_j x_je_j$.
In practice we will in fact work with cosine series  and Neumann boundary conditions and avoid complex values -- see \eqref{sdeder} in Section~\ref{sec:spde} for a more careful derivation. Note that $a$ is a Toeplitz matrix, that is, $a_{jk}$ depends only on $j-k$. In particular $a_{jj}(x)=\int_0^1 A(u(x))(y)^2\,dy$ and $a(x)$ will not be a trace class perturbation of a constant operator 
unless $A$ itself  is constant.  In \cite{ABGP} this restriction manifests itself in a condition (5.3) which in particular forces the $\alpha$-H\"older norms $|a_{ii}|_{C^\alpha}$ to approach zero at a certain rate as $i\to\infty$; a condition which evidently fails unless $A$ is constant. 

Our main results for infinite-dimensional SDEs (Theorems~\ref{mainSDE} and  \ref{SDEth2} below) in fact will use the Toeplitz form of $a$ (or more precisely its near Toeplitz form for our cosine series) to obtain a uniqueness result under an appropriate H\"older continuity condition on $a$. See the discussion prior to \eqref{diagonsum} in Section~\ref{sec:overview} to see how the Toeplitz condition is used.  As a result these results can be used to prove a uniqueness in law result for the SPDE \eqref{SPDEeq0} under a certain H\"older continuity condition on $A(\cdot)$ (see Theorem~\ref{mainSPDE} and Theorem~\ref{mainSPDE2}).

There is a price to be paid for this advance.  First, the H\"older continuity of $a$ in the $e_k$ direction must improve as $k$ gets large, that is,
 for appropriate $\beta>0$
\bee\label{ahold}
|a_{ij}(y+he_k)-a_{ij}(y)|\le\kappa_\beta k^{-\beta}|h|^\alpha.
\eee
Secondly, we require $\al>1/2$.
Finally, to handle the off-diagonal terms of $a$, we assume that for appropriate $\gamma>0$,
\bee\label{offdiag}
|a_{ij}(x)|\le \frac{\kappa_\gamma}{1+|i-j|^\gamma}.
\eee

To handle the SPDE, 
these conditions on the $a_{ij}$ translate to assumptions on $A$.
The operator $A$ will have two types of smoothness. 
The  more interesting type of smoothness is the H\"older continuity of the
map $u\mapsto A(u)$. In order that \eqref{ahold} be satisfied, we require
H\"older continuity of the map $u\mapsto A(u)$ of order $\al>1/2$ and
 with respect to a weak Wasserstein norm involving sufficiently smooth test functions (see \eqref{Fholder} in Theorem~\ref{mainSPDE} and \eqref{Wasserstein} in Theorem~\ref{mainSPDE2}).  
The other type of smoothness is that of $A(u)(x)$ as a function of $x$.
In order that the $a_{ij}$ satisfy \eqref{offdiag}, we require that
$A$ map $C[0,1]$ into a bounded subset of $C^\gamma$ for sufficiently
large $\gamma$.

A consequence of the fact that $A$ must be H\"older continuous with respect
to a weak Wasserstein norm is  
 that $A(u)(x)$ cannot be a H\"older continuous function of point values $u(x+x_i,t)$, $i=1,\dots, n$ but can be a H\"older continuous function of $\langle u,\phi_i\rangle$, $i=1,\dots,n$, for sufficiently smooth test functions as in Corollary~\ref{Examples}. One can of course argue that all measurements are averages of $u$ and so on physical grounds this restriction could be reasonable in a number of settings.  
Although dependence on point values is not a strong feature of our results, it is perhaps of interest
to see what can be done in this direction.  Let $\{\psi_\eps:\eps>0\}$ be a $C^\infty$ compactly supported  even approximate identity so that $\psi_\eps*h(x)\to h(x)$ as $\eps\to0$ for any bounded continuous $h$.  Here $*$ is convolution on the line as usual.  Let $f:\R^n\to [a,b]$ ($0<a<b<\infty$) be H\"older continuous of index $\alpha>\frac{1}{2}$ and $x_1,\dots,x_n\in [0,1]$.
  Then a special case of Corollary~\ref{convoleg} implies uniqueness in law for \eqref{SPDEeq0} with Neumann boundary conditions if 
\bee\label{approxpt}
A(u)(y)=\psi_\delta*(f(\psi_\eps*\ol u(x_1+\cdot),\dots,\psi_\eps*\ol u(x_n+\cdot)))(y),
\eee
where $\ol u(y)$ is the even $2$-periodic extension of $u$ to $\R$.  As $\delta,\eps\downarrow 0$ the above approaches 
\bee\label{classholder}\tilde A(u)(y)=f(\ol u(x_1+y),\dots,\ol u(x_n+y)).
\eee 
Proving uniqueness in 
\eqref{SPDEeq0}  for $A=\tilde A$ remains unresolved for any $\alpha<1$ unless $n=1$ and $x_1=0$. In this case and for the equation \eqref{SPDEeq0} on the line, Mytnik and Perkins \cite{MP} established pathwise uniqueness, and hence uniqueness in law for $A(u)(y)=f(u(y))$ when $f$ is H\"older continuous of index $\alpha>3/4$, while Mueller, Mytnik and Perkins \cite{MMP} showed uniqueness in law may fail in general for $\alpha<3/4$. These latter results are infinite-dimensional extensions of the classical pathwise uniqueness results of Yamada and Watanabe \cite{YW} and a classical example of Girsanov (see e.g. Section V.26 of \cite{RW}), respectively. 
These equations are motivated by branching models with interactions ($f(u)=\sqrt{\sigma(u)u},\ u\ge 0$),
the stepping stone models in population genetics ($f(u)=\sqrt{u(1-u)},\ u\in[0,1]$) and two type branching models with annihilation $f(u)=\sqrt{|u|}, \ u\in\R$. Note these examples have 
degenerate diffusion coefficients and, 
as in the finite-dimensional case, \cite{MP} does not require any non-degeneracy condition on $f$ but is very much confined to the diagonal case in which $A(u)(y)$ depends on $u(y)$.  
In particular their result certainly cannot deal with $A$ as in \eqref{approxpt} (and conversely).  

Due to the failure of standard perturbation methods to produce a uniqueness result for \eqref{SDEeq0} which is applicable to \eqref{SPDEeq0}, we follow a different and more recent approach used to prove well-posedness of martingale problems, first for jump processes in Bass\cite{Ba}, for uniformly elliptic finite dimensional diffusions in Bass and Perkins \cite{BP-Bismut},  and recently for a class of degenerate diffusions in Menozzi \cite{men}.  
Instead of perturbing off a constant coefficient Ornstein-Uhlenbeck operator, the method perturbs off of a mixture of such operators.  
Further details are provided in Section~\ref{sec:overview}. 

We have not spent too much effort on trying to minimize the coefficients $\beta$ and $\gamma$ appearing in \eqref{ahold} and \eqref{offdiag}, and it would be nice to either get rid of $\gamma$ altogether or  produce examples showing some condition here is needed.  Our current hypothesis in Theorems~\ref{mainSDE} and \ref{mainSPDE} require $\gamma\to\infty$ as $\alpha\downarrow 1/2$.  
Do the results here remain valid if the strengthened H\"older conditions \eqref{ahold}, or (for the SPDE), \eqref {Fholder} or \eqref{Wasserstein1}, are replaced with standard H\"older continuity conditions?  Are there examples showing that $\alpha>1/2$ is needed (with or without these additional regularity conditions on $A$) for uniqueness to hold in \eqref{SPDEeq0}?
Most of the motivating examples for \cite{MP} from population genetics and measure-valued diffusions had a H\"older coefficient of $\alpha=1/2$. 
(The counter-examples in \cite{MMP} are for $A(u)(x)=|u(x)|^{(3/4)-\epsilon}$ and so do not satisfy our non-degeneracy condition on $A$.)

The main existence and uniqueness results for \eqref{SDEeq0} and \eqref{SPDEeq0} are stated in Section~\ref{S:RO}. Section~\ref{sec:overview} contains a more detailed description of our basic method using mixtures of Ornstein-Uhlenbeck densities.  Section~\ref{sec:linalg} collects some linear algebra results and elementary inequalities for Gaussian densities. In addition
this section presents Jaffard's theorem and some useful applications of it.  The heavy lifting is done in Sections~\ref{S:GL} and \ref{S:secder} which give bounds on the mixtures of Ornstein-Uhlenbeck process and their moments, and the second order derivatives of these quantities, respectively.  Section~\ref{S:secmainest} then proves the main estimate on smoothing properties of our mixed semigroup.  
The main uniqueness result for Hilbert space-valued SDEs (Theorem~\ref{mainSDE}) is proved in Section~\ref{S:Uniq}.  Finally Section~\ref{sec:spde} proves the slightly more general uniqueness result for SDEs, Theorem~\ref{SDEth2}, and uses it to establish the existence and uniqueness results for the SPDE \eqref{SPDEeq0} (Theorem~\ref{mainSPDE} and Theorem~\ref{mainSPDE2}) and then some specific applications (Corollaries~\ref{Examples} and \ref{convoleg}).  

The proofs of some of the linear algebra results and of the existence
of a solution to \eqref{SPDEeq} are given in Appendices \ref{sec:linalgpfs}
and \ref{A-existence}.

We often use $c_1$ for constants appearing in statements of results and use $c_2$, $c'_2$, $c_3$, $c'_3$ etc. for constants appearing in the proofs.

\bs

\ni{\bf Acknowledgment.} 
M. Neumann acquainted us with
the theorem of Jaffard and related work and  also provided additional
help with some of the linear algebra.
We would also like to thank K. Gr\"ochenig 
and V. Olshevsky for information concerning Jaffard's theorem. 
Finally, we want to thank an anonymous referee, who did a fine job of reading the paper carefully
and making useful suggestions.

\section{Main results}\label{S:RO}

We use $D_if$ for the partial derivative of $f$ in the $i^{th}$
coordinate
direction and $D_{ij}f$ for the corresponding second derivatives.
We denote the inner product in $\R^d$ and the usual inner product
in $L^2[0,1]$ by $\angel{\cdot,\cdot}$; no confusion should result.

Let $C^2_b(\R^k)$ be the set of twice continuous differentiable functions
on $\R^k$ such that the function together with all of its first
and second partial derivatives are bounded, and define $C^2_b(\ell^2)$
analogously.
Let us say $f\in\sT^2_k$ if there exists an $f_k\in C_b^2(\bR^k)$ such 
that $f(x)\allowbreak =f_k(x_1,\dots,x_k)$ and we let $\sT_k^{2,C}$ be 
the set of such
$f$ where $f_k$ is 
 compactly supported.
Let $\sT^2=\cup_k\sT_k^2$ be the class of functions in $C_b^2(\ell^2)$ which depend 
only on finitely many coordinates.
 We let $X_t(\omega)=\omega(t)$ denote the coordinate maps on $C(\bR_+,\ell^2)$.

 We are interested in the Ornstein-Uhlenbeck type operator
 \bee\label{RO-E1}
\sL f(x)=\sum_{i=1}^\infty\sum_{j=1}^\infty a_{ij}(x)D_{ij}f(x)-\sum_
{i=1}^\infty \lam_ix_iD_if(x), \qq x\in \ell^2,
\eee
 for $f\in\sT^2$. Here $\{\lam_i\}$ is a sequence of positive numbers satisfying
\bee\label{s7-E-1} \kappa_{\lam}i^2\leq \lam_i\le \kappa^{-1}_{\lam} i^2
\eee
for all $i=1,2,\ldots$, where $\kappa_\lam$ is a fixed positive finite
constant. We assume throughout that $a$ is a map from $\ell^2$ to $\sL_+(\ell^2,\ell^2)$ so that there exist $0<\Lambda_0\le\Lambda_1<\infty$ satisfying
\bee\label{posdef}
\Lambda_0|w|^2\le \langle a(x)w,w\rangle\le \Lambda_1|w|^2\q \hbox{for all }x,w\in \ell^2.
\eee

Later on we will  suppose
 there exist $\gamma>1$  and a constant $\kappa_\gamma$ such that
\bee\label{gammacond}
|a_{ij}(x)|\leq \frac{\kappa_\gamma}{1+|i-j|^\gamma}
\eee
for all $x\in \ell^2$ and all $i,j$. We will also suppose there exist $\alpha\in(\frac{1}{2},1]$, $\beta>0$ and  a constant $\kappa_\beta$ such that for all $i,j,k\ge 1$ and $y\in \ell^2$,
\bee\label{Ehyp}
|a_{ij}(y+he_k)-a_{ij}(y)|\leq \kappa_\beta|h|^\al k^{-\beta} \hbox{ for all }h\in\R,
\eee
where $e_k$ is the unit vector in the
$x_k$ direction.

Recall that $a_{ij}$ is of Toeplitz form if $a_{ij}$ depends only on 
$i-j$.

 We consider $C(\R_+,\ell^2)$ together with the right continuous filtration generated by
the cylindrical sets. A probability $\bP$ on $C(\bR_+,\ell^2)$ satisfies the martingale 
problem for $\sL$ starting at $v\in \ell^2$ if $\bP(X_0=v)=1$ and 
\[M^f(t)=f(X_t)-f(X_0)-\int_0^t\sL f(X_s)\,ds\]
 is a martingale under $\bP$ for each $f\in \sT^2$.

Our main theorem on countable systems of SDEs, and the theorem whose
proof takes up the bulk of this paper, is the following.

\bet\label{mainSDE} Suppose $\al\in (\tfrac12,1]$, $\beta>\tfrac92 -\al$,
and $\gamma> 2\al/(2\al-1)$. Suppose the $a_{ij}$ satisfy
\eqref{posdef}, \eqref{gammacond}, and \eqref{Ehyp} and that the 
$a_{ij}$ are of Toeplitz form. Let $v\in \ell^2$. Then there exists a solution to the
martingale problem for $\sL$ starting at $v$ and the solution is unique.
\eet

It is routine to derive the following
corollary from Theorem \ref{mainSDE}.

\bec\label{weakuniqueness} Let $\{W^i\}$, $i=1,2,\ldots$ be a sequence
of independent Brownian motions. Let $\sigma_{ij}$ be maps from $\ell^2$
into $\R$ such that if 
$$a_{ij}(x)=\tfrac12 \sum_{k=1}^\infty \sigma_{ik}(x)
\sigma_{kj}(x),$$ then the $a_{ij}$ satisfy the assumptions of 
Theorem \ref{mainSDE}.
Then the $\ell^2$-valued continuous solution to the system of SDEs
\bee\label{system}
dX^i_t=\sum_{j=1}^\infty \sigma_{ij}(X_t)\, dW^j_t-\lam_i X_t^i\, dt, \qq i=1, 2, \ldots,
\eee
is unique in law.
\eec

Uniqueness in law has the usual meaning here. If there exists another process
$\ol X$ with the same initial condition and satisfying
$$d\ol X^i_t=\sum_{j=1}^\infty \sigma_{ij}(\ol X_t)\, d\ol W^j_t-\lam_i 
\ol X_t^i \, dt,$$
where $\{\ol W\}$ is a sequence of independent Brownian motions,
then the joint laws of $(X,W)$ and $(\ol X, \ol W)$ are the same.

We now turn to the stochastic partial differential equation (SPDE) that we are considering:
\bee\label{SPDEeq}
\frac{\del u}{\del t}(x,t)=\tfrac12 \frac{\del^2 u}{\del x^2}(x,t)
+A(u_t)(x)\, \dot W_{x,t}, \qq x\in [0,1],
\eee
where $u_t(x)=u(x,t)$ and $\dot W_{x,t}$ is an adapted space-time Brownian
motion on $[0,1]\times \R_+$ defined on some filtered probability space $(\Omega,\sF,\sF_t,P)$. Here $A$ maps continuous functions on $[0,1]$ to continuous
functions on $[0,1]$. We impose Neumann boundary conditions at the endpoints. Following
Chapter 3 of \cite{walsh}, this means that a continuous $C[0,1]$-valued adapted process $t\to u(t,\cdot)$
 is a solution to \eqref{SPDEeq} if and only if
\bee\label{s10-E11}
\angel{u_t, \vp}=\angel{u_0, \vp}+\int_0^t \angel{u_s, \vp''/2}\, ds
+\int_0^t \int \vp(x) A(u_s)(x)\, dW_{x,s}
\eee
for all $t\ge 0$.
whenever $\vp\in C^2[0,1]$ satisfies $\vp'(0)=\vp'(1)=0$.  Solutions to \eqref{SPDEeq} are unique in law if and only if for a given $u_0\in C[0,1]$ the laws of any two solutions to \eqref{SPDEeq} on $C(\R_+,C[0,1])$ coincide.

Recall that $\{e_k\}$ is a complete orthonormal system for $L^2[0,1]$ of eigenfunctions of the Laplacian satisfying appropriate boundary conditions.
We specialize our earlier notation and let $e_k(x)=\sqrt 2\cos(k\pi x)$ if $k\ge 1$, and $e_0(x)\equiv 1$. 
Here is our theorem for SPDEs. It is proved in Section~\ref{sec:spde} along with the remaining results in
this section. 

\bet\label{mainSPDE} Assume 
\bee\label{L2cont} u_n\to u\hbox{ in } C[0,1]\hbox{ implies }\norm{A(u_n)-A(u)}_2\to 0.
\eee
Suppose there exist 
$$\al\in \Big(\frac12,1\Big], \qq
\gamma>\frac{2\al}{2\al-1},\qq
 \beta>\Big(\Big(\frac92\Big)-\alpha\Big)\lor \Big(\frac{\gamma}{2-\gamma}\Big),$$
 and also positive 
constants $\kappa_1$, $\kappa_2$ and $\kappa_3$ such that for all $u\in C[0,1]$,
\begin{equation}\label{Fholder}
\norm{A(u+he_k)-A(u)}_2\le \kappa_1|h|^\alpha (k+1)^{-\beta}\qq\hbox{ for all }k\ge 0,\ h\in\R,
\end{equation}
\begin{equation}\label{Abnd}
0<\kappa_2\leq A(u)(x)\leq \kappa_2^{-1}, \qq \hbox{ for all }x\in [0,1], 
\end{equation}
and
\begin{equation}\label{Fdecay}
|\langle A(u)^2,e_k\rangle|\le \frac{\kappa_3}{1+(k+1)^\gamma}\qq\hbox{ for all }k\ge 0.
\end{equation}
Then for any $u_0\in C([0,1])$ there is a solution of \eqref{SPDEeq} 
in the sense of \eqref{s10-E11} and the solution is unique in law.
\eet

To give a better idea of what the above conditions \eqref{Fholder} and \eqref{Fdecay} entail we formulate some regularity conditions on $A(u)$ which will imply them.

For $\delta\in [0,1)$ and $k\in\Z_+$, 
$\norm{u}_{C^{k+\delta}}$ has the usual definition:
$$\norm{u}_{C^{k+\delta}}=\sum_{i=0}^k \norm{u^{(i)}}_\infty
+1_{(\delta>0)}\sup_{x\ne y;x,y\in[0,1]} \frac{|u^{(k)}(y)-u^{(k)}(x)|}{|y-x|^\delta},
$$ where $u^{(i)}$ is the $i^{th}$ derivative of $u$ and we consider
the $0^{th}$ derivative of $u$ to just be $u$ itself. $C^k$ is the usual
space of $k$ times continuously differentiable functions equipped with $\norm{\cdot}_{C^k}$
and $C^{k+\delta}=\{u\in C^k:\norm{u}_{C^{k+\delta}}<\infty\}$ with the norm $\norm{u}_{C^{k+\delta}}$.

If $f\in C([0,1])$ let $\ol f$ be the extension of $f$ to $\R$ obtained by first reflecting to define an even function on $[-1,1]$, and then extending to $\R$ as a $2$-periodic continuous function. That is, $\ol f(-x)=f(x)$ for
$0<x\leq 1$ and $\ol f(x+2)=\ol f(x)$ for all $x$.
In order to be able to work with real valued processes and functions,
we introduce the space 
\[C^\zeta_{per}=\{f\in C^\zeta([0,1]):\ol f\in C^\zeta(\R)\},\]
that is, the set of $f$ whose even extension to the circle of circumference 2  is in $C^\zeta$.
A bit of calculus shows that 
$f\in C^\zeta_{per}$ if and only if $f\in C^\zeta([0,1])$  and $f^{(k)}(0)=f^{(k)}(1)=0$ for all odd $k\le \zeta$.
Such $f$ will be even functions, and consequently their Fourier
coefficients (considered on the interval $[-1,1]$) will be real.

The following theorem  is a corollary to Theorem~\ref{mainSPDE}.

\bet\label{mainSPDE2} Suppose there exist 
$$\al\in \Big(\frac12,1\Big],\qq 
\gamma>\frac{2\al}{(2\al-1)},\qq
\ol \beta>\Big(\Big(\frac9{2\al}\Big)-1\Big)\lor \Big(\frac{\gamma}{\al(2-\gamma)}\Big),$$  
 and also positive 
constants $\kappa_1$, $\kappa_2$ and $\kappa_3$ such that for all $u,v$ continuous on $[0,1]$,
\bee\label{Wasserstein1}
\norm{A(u)-A(v)}_2\leq \kappa_1 \sup_{\vp\in C^{\ol\beta}_{per}, \norm{\vp}_{C^{\ol\beta}}\leq 1}
|\angel{u-v,\vp}|^\al,
\eee 
\bee\label{Auplow}
0<\kappa_2\leq A(u)(x)\leq \kappa_2^{-1}, \qq x\in [0,1],
\eee
and
\bee\label{Asmooth1}A(u)\in C^\gamma_{per}\hbox{ and }\norm{A(u)}_{C^\gamma}\leq \kappa_3.
\eee
Then for any $u_0\in C([0,1])$ there is a solution of \eqref{SPDEeq} and the solution is unique in law.
\eet

Note that \eqref{Wasserstein1} is imposing H\"older continuity in a certain Wasserstein metric. 

\begin{remark}{\rm  The above conditions on $\alpha$, $\beta$ and $\gamma$ hold if $\gamma>\frac{2\alpha}{2\alpha-1}\vee \frac{14}{5}$, and $\ol\beta>\frac{9}{2\alpha}-1$.
}
\end{remark}

As a consequence of Theorem~\ref{mainSPDE2}, we give a class of examples. Let $\al\in (\tfrac12, 1]$.
Suppose $n\geq 1$ and
$\vp_1, \ldots, \vp_n$ are functions in $C^{\ol\beta}_{per}$ for $\ol\beta>\frac{9}{2\alpha}-1$.
Suppose $f:[0,1]\times \R^n\to [0,\infty)$ is bounded
above and below by positive constants, and $f$ as a function of the
first variable is in $C^\gamma_{per}$ for $\gamma>\frac{2\al}{2\al-1}\vee \frac{14}{5}$ and 
satisfies $\sup_{y_1,\dots,y_n}\norm{f(\cdot,y_1,\dots,y_n)}_\gamma\le \kappa$.  Assume also that
$f$ is H\"older continuous of order $\al$ with respect to its
second through $(n+1)^{st}$ variables:
\begin{align*}
|f(x, y_1, \ldots, y_{i-1}, y_i+h, y_{i+1}, & \ldots, y_n)
-f(x, y_1, \ldots, y_{i-1}, y_i, y_{i+1}, \ldots, y_n)|\\
&\le c|h|^\al,\qq\hbox{for }1\le i\le n,
\end{align*}
where $c$ does not depend on $x, y_1, \ldots, y_n$.

\bec\label{Examples}
With $f$ and $\vp_1, \ldots, \vp_n$  as above, 
 let
$$A(u)(x)=f(x, \angel{u,\vp_1}, \ldots, \angel{u, \vp_n}).$$
Then a solution to \eqref{SPDEeq} exists and is unique in law.
\eec

A second class of examples can be built from convolution operators. If $f$, $g$ are 
real-valued functions on the line, $f*g$ is the usual convolution of $f$ and $g$.  

\bec\label{convoleg} Assume $\psi:\R\to \R_+$ and $\phi_1,\phi_2,\dots\phi_n:\R\to\R$ are even $C^\infty$ functions with compact support and $\psi$ is not identically $0$.  Suppose also that for some $0<a\le b<\infty$ and some $\alpha\in(1/2,1]$, $f:\R^n\to [a,b]$ satisfies
\bee\label{fHold}
|f(x)-f(x')|\le c_f\norm{x-x'}^\alpha_\infty\qq\hbox{for all }x,x'\in\R^n.
\eee
If 
\bee\label{convA}
A(u)(x)=\psi*(f(\phi_1*\ol u(\cdot),\dots,\phi_n*\ol u(\cdot)))(x),
\eee
then there is a solution to \eqref{SPDEeq} and the solution is unique in 
law.
\eec

One can construct a
physical model corresponding to Corollaries
\ref{Examples} and \ref{convoleg}. Consider a thin metal rod of unit length
with insulated ends and wrapped with a non-homogeneous partially insulated
material. Subject the rod to random heat along the length of the rod; this
represents $\dot W_{t,x}$. The heat flows along the rod according to 
\eqref{SPDEeq0}. The partially insulated wrapping corresponds to $A(u)$. 
If $n=1$ and $A$ is a function of a weighted average of the temperatures
along the rod, we are in the context of Corollary \ref{Examples}. If $n=1$
and one can only measure temperatures as an average of a neighborhood
of any given point, then Corollary \ref{convoleg} might apply.

\section{Overview of proof}\label{sec:overview}

In this section we give an overview of our argument. For most of this overview, we focus on the stochastic differential equation \eqref{SDEeq0} where $a$ is of Toeplitz form, that is, $a_{ij}$ depends only on $i-j$.  This is where the difficulties lie and
puts us in the context of Theorem \ref{mainSDE}.

Assume we have a $K\times K$ matrix $a$
that is of Toeplitz form, and we will require all of our estimates
to be independent of $K$.  Define
$$\sM^z f(x)=\sum_{i,j=1}^K a_{ij}(z) D_{ij}f(x)
-\sum_{i=1}^K \lam_i x_i D_i f(x),$$
where $\lambda_i$ satisfies \eqref{s7-E-1}.
Let $p^z(t,x,y)$ be the corresponding transition probability
densities and let $r^z_\theta(x,y)$ be the resolvent densities. Thus $\sL f(x)
=\sM^x f(x)$. 

We were unable to get the standard perturbation method to work and instead
we used the method described in \cite{BP-Bismut}. The idea is to suppose
there are two solutions $\P_1$ and $\P_2$ to the martingale problem
and to let $S_i f=\E_i \int_0^\infty e^{-\theta t} f(X_t)\, dt$.
Some routine calculations show that 
$S_i(\theta-\sL)f=f,$
and so $S_\Delta (\theta-\sL)f=0$, where $S_\Delta$ is the linear
functional $S_1-S_2$.
If $$f(x)=\int r_\theta^y(x,y) g(y)\, dy$$
were in the domain of $\sL$ when $g$ is $C^\infty$ with compact support, we would have
\begin{align*}
(\theta-\sL)f(x)&=\int (\theta-\sM^y)r^y_\theta(x,y) g(y)\, dy
+\int (\sM^y-\sM^x)r^y_\theta(x,y) g(y)\, dy\\
&=g(x)+\int (\sM^y-\sM^x)r^y_\theta(x,y) g(y)\, dy.
\end{align*}
Such $f$ need not be in the domain of $\sL$, but we can do an approximation
to get around that problem.

If we can show that
\bee\label{s3-E1}
\Big|\int (\sM^y-\sM^x)r^y_\theta(x,y) g(y)\, dy\Big|\leq \tfrac12 \norm{g}_\infty,
\eee 
for $\theta$ large enough, we would then
get $$|S_\Delta g|\leq \tfrac12 \norm{S_\Delta}\, \norm{g}_\infty ,$$
which implies that the norm of the linear functional $S_\Delta$ is zero.
It is then standard to obtain the uniqueness of the martingale problem from 
this.

We derive \eqref{s3-E1} from a suitable bound on 
\bee\label{s3-E2}
\int \Big| (\sM^y-\sM^x)p^y(t,x,y)\Big|\, dy.
\eee
Our bound needs to be independent of $K$, and it turns out the difficulties 
are all when $t$ is small.

When calculating $D_{ij}p^y(t,x,y)$, where the derivatives are with
respect to the $x$ variable, we obtain a factor $e^{-(\lam_i+\lam_j) t}$
(see \eqref{s7-E-1A}),
and thus by \eqref{s7-E-1}, when summing over $i$ and $j$, we need only sum from 1 to
$J\approx t^{-1/2}$ instead of from 1 to $K$. When we estimate \eqref{s3-E2},
we get a factor $t^{-1}$ from $D_{ij} p^y(t,x,y)$ and we get a factor
$|y-x|^{\al}\approx t^{\al/2}$ from the terms $a_{ij}(y)-a_{ij}(x)$. 
If we consider only the main diagonal, we have $J$ terms, but they behave
somewhat like sums of independent mean zero random variables, so we
get a factor $\sqrt J\approx t^{-1/4}$ from summing over the main diagonal
where $i=j$ ranges from 1 to $J$. Therefore when $\al> 1/2$, we get
a total contribution of order $t^{-1+\eta}$ for some $\eta>0$, 
which is integrable near
0. The Toeplitz form of $a$ allows us to factor out $a_{ii}(y)-a_{ii}(x)$ from the sum
since it is independent of $i$ and so we are indeed left with the integral in $y$ of 
\bee\label{diagonsum}\left|\sum_{i=1}^JD_{ii}p^y(t,x,y)\right|.\eee

Let us point out a number of difficulties. All of our estimates
need to be independent of $K$, and it is not at all clear that
$$\int_{\R^K} p^y(t,x,y)\, dy$$
can be bounded independently of $K$. That it can is Theorem \ref{GLM-main}.
We replace the $a_{ij}(y)$ by a matrix that does
not depend on $y_K$. This introduces an error, but not too bad a one.
We can then integrate over $y_{K}$ and reduce the
situation from the case where $a$ is a $K\times K$ matrix  to where it is
$(K-1)\times (K-1)$ 
and we are now in the $(K-1)\times (K-1)$ situation. We do an induction
and keep track of the errors. 

From \eqref{diagonsum} we need to handle 
$$\int\Big|\sum_{i=1}^J D_{ii} p^y(t,x,y)\Big|\, dy,$$
and here we use Cauchy-Schwarz, and get an estimate on
$$\int \sum_{i,j=1}^J D_{ii}p^y(t,x,y) D_{jj}p^y(t,x,y)\, dy.$$
This is done in a manner similar to bounding $\int p^y(t,x,y)\, dy$, although
the calculations are of course more complicated. 

 We are assuming that
$a_{ij}(x)$ decays at a rate at least
$(1+|i-j|)^\gamma$ as $|i-j|$ gets
large. Thus the other diagonals besides the main one can be handled in a similar
manner and $\gamma>1$ allows us to then sum over the diagonals. 

A major complication that arises is that $D_{ij} p^y(t,x,y)$
involves $a^{-1}$ and we need a good off-diagonal decay on $a^{-1}$
as well as on $a$. An elegant linear algebra theorem of Jaffard
gives us the necessary decay, independently of the dimension.
 
 To apply the above, or more precisely its cousin Theorem~\ref{SDEth2}, 
 to the SPDE \eqref{SPDEeq0} with Neumann boundary conditions, we write a solution $u(\cdot,t)$ in terms of a Fourier cosine  series with random coefficients.  Let $e_n(x)=\sqrt 2 \cos(\pi n x)$ if $n\ge 1$, and $e_0(x)\equiv 1$, $\lambda_n=n^2\pi^2/2$ and define $X^n(t)=\langle u(\cdot,t),e_n\rangle$.  Then it is easy to see that $X=(X^n)$ satisfies \eqref{SDEeq0} with 
 \[a_{jk}(x)=\int_0^1A(u(x))^2(y)e_j(y)e_k(y)\,dy,\qq x\in\ell^2(\Z_+),\]
 where $u(x)=\sum_0^\infty x_ne_n$.  We are suppressing some  issues in this
overview, such as  extending the domain of $A$ to $L^2$.  Although $(a_{jk})$ is not of Toeplitz form it is easy to see it is a small perturbation of a Toeplitz matrix and satisfies the hypotheses of Theorem~\ref{SDEth2}.  This result then gives the uniqueness in law of $X$ and hence of $u$.

\section{Some linear algebra}\label{sec:linalg}

Suppose $m\ge 1$ is given.
Define $g_r=rI$, where $I$ is the $m\times m$ identity matrix 
and let $E(s)$ be the diagonal matrix whose $(i,i)$ entry 
is $e^{-\lam_i s}$ for a given sequence of positive reals $\lambda_1\le\dots\le \lambda_m$.
Given an $m\times m$ matrix $a$,
let 
\bee \label{LAE1} a(t)=\int_0^t E(s)\, a\, E(s)\, ds
\eee
 be the matrix whose $(i,j)$ entry is
$$a_{ij}(t)=a_{ij} \frac{1-e^{-(\lam_i+\lam_j)t}}{\lam_i+\lam_j}.$$
Note $\lim_{t\to 0} a_{ij}(t)/t=a_{ij}$, and we may view $a$ as $a'(0)$.

Given a nonsingular matrix $a$, we use $A$ for $a^{-1}$. When we write
$A(t)$, this will refer to the inverse of $a(t)$. Given a matrix
$b$ or $g_r$, we define $B, G_r, b(t), g_r(t), B(t)$, and $G_r(t)$ 
analogously. If $r=1$ we will write $G$ for $G_1$ and $g$ for $g_1$.

Let $\Vert a\Vert$ be the usual operator norm, that is,
$\norm{a}=\sup\{\norm{aw}: \norm{w}\leq 1\}$.  
If $C$ is a $m\times m$ matrix, recall that the determinant of $C$
is the product of the eigenvalues and the spectral radius is bounded
by $\norm{C}$. Hence
\bee\label{detbd}
|\det C|\leq \norm{C}^m.
\eee

If $a$  and $b$ are non-negative definite matrices, we write $a\geq b$ if $a-b$
is also non-negative definite.
Recall that if $a\geq b$, then $\det a\geq \det b$ and $B\geq A$. 
This can be found, for example, in \cite[Corollary 7.7.4]{horn-johnson}.

\bel\label{Q1} Suppose $a$ is a matrix with $a\geq g_r$. Then 
$a(t)\geq g_r(t)$
and 
\bee\label{LQ2}
\det a(t)\ge \det g_r(t).
\eee
\eel

\begin{proof}
Using \eqref{LAE1},
$$a(t)=\int_0^t E(s)\, a\, E(s)\, ds\geq
\int_0^t E(s)\, g_r\, E(s)\, ds=g_r(t).$$
\eqref{LQ2} follows.
\end{proof}

For arbitrary square matrices  $a$ we let $$\Vert a\Vert_s=\max\{\sup_i\sum_j|a_{ij}|,\sup_j\sum_i|a_{ij}|\}.$$ 
Schur's Lemma (see e.g., Lemma 1 of \cite{jaffard}) states that
\bee\label{Schur}\Vert a\Vert\le \Vert a\Vert_s.\eee 
As an immediate consequence we have:
\bel\label{LAL3} If $a$ is a $m\times m$ matrix, then
$$|\angel{x,ay}|\leq  \norm{x}\, \norm{ay}\le \norm{x}\, \norm{y} \Vert a\Vert_s.$$
\eel

\bel\label{jaff2L1} For all $\lam_i, \lam_j$,
\bee\label{jaff2E1}
\Big(\frac{2\lam_i}{1-e^{-2\lam_it}}\Big)^{1/2} 
\Big(\frac{1-e^{-(\lam_i+\lam_j)t}}{\lam_i+\lam_j}\Big)
\Big(\frac{2\lam_j}{1-e^{-2\lam_jt}}\Big)^{1/2}\leq 1
\eee
for all $t>0$.
\eel

\begin{proof} This is equivalent to 
$$\int_0^t e^{-(\lambda_i+\lambda_j)s}\,ds\le\Big(\int_0^te^{-2\lambda_is}\,ds\Big)^{1/2}\Big(\int_0^te^{-2\lambda_js}\,ds\Big)^{1/2}$$
and so is immediate from Cauchy-Schwarz.
\end{proof}

Define 
$$\wt a(t)=G(t)^{1/2}a(t) G(t)^{1/2},$$
so that
\bee\label{JAE3}
\wt a_{ij}(t)= G_{ii}(t)^{1/2} a_{ij}(t) G_{jj}(t)^{1/2}.
\eee
Let $\wt A(t)$ be the inverse of $\wt a(t)$, that is,
\bee \label{tildeAdef}
\wt A(t)=g(t)^{1/2}A(t)g(t)^{1/2}.\eee
A calculus exercise will show that for all positive $\lam,t$,
\bee\label{ETE1}
\frac{1+\lam t}{2t}\le\frac{2\lam}{1-e^{-2\lam t}}\leq \frac{2(1+\lam t)}{t}.
\eee

\bel\label{tildeabnd} Suppose $0<\Lambda_0\le \Lambda_1<\infty$. If $a$ be a positive definite matrix with
$g_{\Lambda_1}\geq a\geq g_{\Lambda_0}$, and 
\[\ol\Lambda_0(t)=\Lambda_0\Bigl(\frac{1-e^{-2\lam_m t}}{2\lam_m}\Bigr),\qq \ol\Lambda_1(t)=\Lambda_1\Bigl(\frac{1-e^{-2\lam_1 t}}{2\lam_1}\Bigr),\]
then for all $t>0$,
\[g_{\Lambda_1}\ge \wt a(t)\ge g_{\Lambda_0}\quad g_{\ol\Lambda_1(t)}\ge a(t)\ge g_{\ol\Lambda_0(t)}.\]
\eel

For the proof see Appendix \ref{sec:linalgpfs}.

\bel\label{Q4} Let $a$ and $b$ be positive definite matrices with
$g_{\Lambda_1}\geq a,b\geq g_{\Lambda_0}$. Then 
\bee\label{tildeab}\norm{\wt a(t)-\wt b(t)}\le \norm{\wt a(t)-\wt b(t)}_s\le \norm{a-b}_s,
\eee
\bee\label{normbnd}\norm{\wt A(t)-\wt B(t)}\le \Lambda_0^{-2}\norm{a-b}_s,
\eee
and
for all $w,w'$,
\bee\label{LAE3}
|\langle w,(\wt A(t)-\wt B(t))w'\rangle|\leq \Lambda_0^{-2}\norm{w}\norm{w'} \Vert a-b\Vert_s.
\eee
\eel

For the proof see Appendix \ref{sec:linalgpfs}.

\bel\label{Q5} Let $a$ and $b$ be positive definite matrices with
$g_{\Lambda_1}\geq a,b\geq g_{\Lambda_0}$, and set
$\theta=\Lambda_0^{-1}m\norm{a-b}_s$.
Then
$$\Big|\frac{\det \wt b(t)}{\det \wt a(t)}-1\Big|\leq \theta e^\theta.$$
\eel

For the proof see Appendix \ref{sec:linalgpfs}.

Let us introduce the notation
\bee\label{key-31}
Q_m(w,C)=(2\pi)^{-m/2} (\det C)^{1/2} e^{-\langle w,Cw\rangle/2},
\eee
where $C$ is a positive definite $m\times m$ matrix, and $w\in \R^m$.

\bep\label{Q5A} Assume $a,b$ are as in Lemma~\ref{Q5}. Set 
\[\theta=\Lambda_0^{-1}m\norm{a-b}_s\hbox{ and }\phi=\Lambda_0^{-2}\norm{w}^2\norm{a-b}_s.\]
For any $M>0$ there is a constant $c_1=c_1(M)$ so that if $\theta, \phi<M$, then
$$\Big|\frac{Q_m(w,\wt A(t))}{Q_m(w,\wt B(t))}-1\Big|\leq c_1 (\phi+\theta).$$
\eep

For the proof see Appendix \ref{sec:linalgpfs}.

We note that if   $a$, $b$ are $m\times m$ matrices satisfying $\sup_{i,j}|a_{ij}-b_{ij}|\le \delta$, we
 have the trivial bound
\bee\label{easyschur}
\norm{a-b}_s\le m\delta.
\eee

\bel\label{LAL100} 
Suppose $a$ is a $(m+1)\times (m+1)$ positive definite matrix, $A$
is the inverse of $a$, $B$ is the $m\times m$ matrix defined by
\bee\label{LAE101}
B_{ij}=A_{ij}-\frac{A_{i,m+1}A_{j,m+1}}{A_{m+1,m+1}}, \qq i,j\leq m.
\eee
Let $b$ be the $m\times m$ matrix defined by $b_{ij}=a_{ij}$, $i,j\leq m$.
Then $b=B^{-1}$.
\eel

For the proof see Appendix \ref{sec:linalgpfs}.

We will use the following result of Jaffard (Proposition 3 in \cite{jaffard}). Throughout the remainder of this  section $\gamma>1$ is fixed.  
Suppose $0<\Lambda_0\le \Lambda_1<\infty$ and let $K\ge 1$.

\bep\label{jaffthm} Assume $b$ is an invertible $K\times K$ matrix satisfying \hfil\break $\norm{b}\le \Lambda_1$, $\norm{B}\le \Lambda_0^{-1}$, and 
\[ |b_{ij}|\le \frac{c_1}{1+|i-j|^\gamma}\quad\hbox{ for all }i,j,\]
where $B=b^{-1}$.
There is a constant $c_{2}$, depending only on $c_1$, $\gamma$,  $\Lambda_0$ and $\Lambda_1$, but not $K$,  such that 
\[|B_{ij}|\le \frac{c_{2}}{1+|i-j|^\gamma}\quad\hbox{ for all }i,j.\]
\eep

The dependence of $c_{2}$ on the given parameters is implicit in the proof in \cite{jaffard}.

We now suppose that $a$ is  a positive definite $K\times K$ matrix such that for some positive $\Lambda_0,\Lambda_1$, 
\bee\label{posdefJA}
g_{\Lambda_1}\ge a\ge g_{\Lambda_0}.
\eee
We suppose also that \eqref{gammacond} holds.
Our estimates and constants in this section may depend on $\Lambda_i$ and $\kappa_\gamma$,
but will be independent of $K$, as is the case in Proposition~\ref{jaffthm}.

Recall $a(t)$ and $\wt a(t)$ are defined in \eqref{LAE1} and \eqref{JAE3}, respectively, and $A(t)$ and $\wt A(t)$, respectively,  are their inverses.
\bel\label{JAP4}
For all $t>0$, 
$$|\wt a_{ij}(t)|\leq \frac{\kappa_\gamma}{1+|i-j|^\gamma}\quad\hbox{ for all }i,j.$$
\eel

\begin{proof}
Since $$G_{ii}(t)=\frac{2\lam_i}{1-e^{-2\lam_it}},$$
then
$$\wt a_{ij}(t)=\Big(\frac{2\lam_i}{1-e^{-2\lam_it}}\Big)^{1/2}
a_{ij}\Bigl( \frac{1-e^{-(\lam_i+\lam_j)t}}{\lam_i+\lam_j}\Bigr)
\Big(\frac{2\lam_j}{1-e^{-2\lam_jt}}\Big)^{1/2}.$$
Using \eqref{gammacond} and Lemma \ref{jaff2L1},
we have our result.
\end{proof}

\bel\label{JAP5}
There exists a constant $c_{1}$,  depending only on $\kappa_\gamma$, $\Lambda_0$ and $\Lambda_1$,
 so that
$$|\wt A_{ij}(t)|\leq \frac{c_{1}}{1+|i-j|^\gamma}.$$
\eel

\begin{proof} This follows immediately from Lemma~\ref{tildeabnd}, Lemma~\ref{JAP4}, and Jaffard's theorem (Proposition~\ref{jaffthm}).
\end{proof}

We set 
\[L(i,j,t)=\Big(\frac{1+ \lam_i t}{t}\Big)^{1/2}
\Big(\frac{1+ \lam_j t}{t}\Big)^{1/2}.\]

The proposition we will use in the later parts of the paper is the following.

\bep\label{JAP6}
There exists a constant $c_{1}$, depending only on 
$\kappa_\gamma$, $\Lambda_0$ and $\Lambda_1$,  such that
$$(2\Lambda_{1})^{-1}L(i,i,t)1_{(i=j)}\le|A_{ij}(t)|\leq L(i,j,t) \Bigl(\frac{c_{1}}{1+|i-j|^\gamma}\Bigr).$$
\eep

\begin{proof}
Since $\wt a(t)=G(t)^{1/2}a(t)G(t)^{1/2}$, then
$$a(t)=g(t)^{1/2}\wt a(t) g(t)^{1/2},$$ and hence
$$A(t)=G(t)^{1/2} \wt A(t) G(t)^{1/2}.$$
Therefore
\bee\label{Aeqn}A_{ij}(t)=\Big(\frac{2\lam_i}{1-e^{-2\lam_it}}\Big)^{1/2} \wt A_{ij}(t)
\Big(\frac{2\lam_{j}}{1-e^{-2\lam_{j}t}}\Big)^{1/2}.\eee
The upper bound now follows from Lemma~\ref{JAP5} and \eqref{ETE1}.

For the left hand inequality, by
\eqref{Aeqn} and the lower bound in \eqref{ETE1} 
it suffices to show
\bee\label{tildebnd}
\wt A_{ii}(t)\ge \Lambda_1^{-1},
\eee
and this is immediate from the uniform upper bound on $\wt a(t)$ in Lemma~\ref{tildeabnd}.

\end{proof}

\section{A Gaussian-like measure}\label{S:GL}

Let us suppose $K$ is a fixed positive integer, $0<\Lambda_0\leq \Lambda_1<\infty$, and  that we have a $K\times K$ symmetric matrix-valued function $a:\R^K\to \R^{K\times K}$ with
$$\Lambda_0 \sum_{i=1}^K |y_i|^2 \leq \sum_{i,j=1}^K a_{ij}(x) y_iy_j
\leq \Lambda_1 \sum_{i=1}^K |y_i|^2, \quad x\in \R^K, y=(y_1, \ldots, y_K)\in \R^K.$$
It will be important that all our bounds and estimates in this section will not depend on $K$.
We will assume $0< \lam_1\le \lam_2\le\dots\le \lam_K$ satisfy \eqref{s7-E-1}.
As usual, $A(x)$ denotes the inverse to $a(x)$, and we define
$$a_{ij}(x,t)=a_{ij}(x) \int_0^te^{-(\lam_i+\lam_j)s}\,ds,$$
and then 
$A(x,t)$ to be the inverse of $a(x,t)$. 
Let $\wt a(x,t)$ and $\wt A(x,t)$ be defined as in \eqref{JAE3} and \eqref{tildeAdef}, respectively. 
When $x=(x_1, \ldots, x_K)$, define $x'=(x'_1, \ldots, x'_K)$ by
$$x'_i=e^{-\lam_it}x_i,$$
and set $w=y-x'$.
For $j\leq K$, define $\pi_{j,x}:\R^K\to\R^K$ by
$$\pi_{j,x}(y)=(y_1, y_2, \ldots, y_j, x'_{j+1}, \ldots, x'_K),$$
and write $\pi_j$ for $\pi_{j,x}$ if there is no ambiguity. 
From \eqref{key-31} we see that
\bee\label{defN}
Q_K(w,A(y,t))=(2\pi)^{-K/2} (\det A(y,t))^{1/2}
\exp\Big( -\langle w,A(y,t)w\rangle/2\Big).
\eee
The dependence of $A$ on $y$ but not $x$ is not a misprint; $y\to Q_K(y-x',A(y,t))$ will not be a probability density.  It is however readily seen to be integrable; we show more below.

The choice of $K$ in the next result is designed to implement a key induction argument later in this section.

\bel\label{Q3} Assume $K=m+1$ and $a(y)=a(\pi_m(y))$ for all 
$y\in \R^K$, that is, $a(y)$ does not depend on $y_{m+1}$. Let $b(y)$ be the $m\times m$ matrix with
$b_{ij}(y)=a_{ij}(y)$ for $i,j\leq m$, and let
$B(y)$ be the inverse of $b(y)$.
Then for all $x$, \\
\noindent(a) we have $$\int Q_{m+1}(w,A(y))\,dy_{m+1}=Q_m(w,B(y)).$$

\noindent(b) If $y_1, \ldots, y_m$ are held fixed, 
$Q_{m+1}(w,A(y))/Q_{m}(w,B(y))$ equals the density of a 
normal random variable with 
 mean 
$$\mu(y_1,\dots,y_m)=-\,\frac{\sum_{i=1}^mw_iA_{i,m+1}(y)}{A_{m+1,m+1}(y)}$$
and variance $\sigma^2(y_1,\dots,y_m)=(A_{m+1,m+1}(y))^{-1}$.
\eel

\begin{proof} Lemma~\ref{LAL100} and some algebra show that
\begin{align}
\sum_{i,j=1}^{m+1}(y_i-x'_i)(y_j-x'_j) A_{ij}(y)
&=\sum_{i,j=1}^m(y_i-x'_i)(y_j-x'_j) B_{ij}(y) \label{LAE103}\\
&\qq +A_{m+1,m+1}(y)|y_{m+1}-x'_{m+1}-\mu|^2.\nn
\end{align}

Let $C(y)$ be the $(m+1)\times (m+1)$ matrix such that
$$\begin{cases}
C_{ij}(y)=B_{ij}(y), & i,j\leq m;\\
C_{i,m+1}(y)=0, & i\leq m;\\
C_{m+1,j}(y)=A_{m+1,j}(y), & j\leq m+1.
\end{cases}$$
If $\mbox{row}_i(D)$ denotes the $i^{th}$ row of a matrix $D$, note that
$$\mbox{row}_i(C(y))=\mbox{row}_i(A(y))-1_{(i\le m)} \frac{A_{i,m+1}}{A_{m+1,m+1}}
\mbox{row}_{m+1}(A(y)).$$
Therefore $\det C(y)=\det A(y)>0$, and it follows that
\bee\label{LAE104}\det A(y)=\det C(y)=A_{m+1,m+1}(y)\det B(y).
\eee

Part (a) now follows from \eqref{LAE103}, \eqref{LAE104}, and
evaluating the standard Gaussian integral. Part (b) is then immediate from \eqref{LAE103} and \eqref{LAE104}.
\end{proof}

Let $B_0=8\log(\Lambda_1/\Lambda_0)+4\log 2$ and for $B>0$ let
\bee\label{GLM-E101}
S_{B,K}=S_B=\{z\in\R^K:\norm{z}^2<B\Lambda_1K\}.
\eee
Recalling that $w=y-x'$, we will often use the further change of variables
\bee\label{w'defn}
w'=G(t)^{1/2}w=G(t)^{1/2}(y-x').
\eee
Note that when integrating $Q_K(w',A(y,t))$ with respect to $w'$, $y$ is an implicit function of $w'$.

\bel\label{U1.5} (a) For any $p\ge 0$ there is a $c_p=c_p(\Lambda_1)$ such that if $B\ge B_0$ and  $F$ is a $K\times K$ symmetric matrix-valued function of
$w$ with
$G_{\Lambda_0}\geq F\geq G_{\Lambda_1}$, then
$$\int_{S_B^c} \norm{w}^{2p}Q_K(w,F)\, dw\leq c_pK^{p}e^{-BK/16}.$$

(b) Let $x\in \R^k$. There exists $c_p$ not depending on $x$ such that
$$\int_{S_B^c} \norm{w'}^{2p}Q_K(w',\wt A(y,t))\, dw'\leq c_pK^{p}e^{-BK/16}.$$
\eel

\begin{proof} 
(a) 
We have
$G_{\Lambda_0}=(\Lambda_1/\Lambda_0)G_{\Lambda_1}$, and so
\begin{align}
\label{w'int}Q_K(w, F)&\leq (2\pi)^{-K/2} (\det G_{\Lambda_0})^{1/2}
e^{-\langle w,G_{\Lambda_1}w\rangle/2}\\
\nonumber&=\Big(\frac{\Lambda_1}{\Lambda_0}\Big)^{K/2} Q_K(w, G_{\Lambda_1}).
\end{align}

Let $Z_i$ be i.i.d. mean zero normal random variables with variance 1 and let
$$Y_i=\sqrt{\Lambda_1}Z_i.$$
From \eqref{w'int} we have
$$\int_{S_B^c} \norm{w}^{2p}Q_K(w, F)\, dw
\leq \Bigl(\frac{\Lambda_1}{\Lambda_0}\Bigr)^{K/2}
\int_{S_B^c} \norm{w}^{2p}Q_K(w,G_{\Lambda_1}) \, dw.$$
The right hand side is the same as
\begin{align*}
\Bigl(\frac{\Lambda_1}{\Lambda_0}\Bigr)^{K/2}&\E\Big[\Bigl(\sum_{i=1}^K\Lambda_1|Z_i|^2\Bigr)^p; \sum_{i=1}^K \Lambda_1|Z_i|^2\geq B\Lambda_1 K\Big]\\
&\leq \Bigl(\frac{\Lambda_1}{\Lambda_0}\Bigr)^{K/2} (\Lambda_1)^p\E\Bigl[\Bigl(
\sum_{i=1}^K |Z_i|^2\Bigr)^p;\sum_{i=1}^K |Z_i|^2 \geq BK\Bigr]\\
&\leq \Bigl(\frac{\Lambda_1}{\Lambda_0}\Bigr)^{K/2} (\Lambda_1)^p\Bigl[\E(\Bigl(
\sum_{i=1}^K|Z_i|^2\Bigr)^{2p}\Bigr]^{1/2}\\
&\qq\qq \times \Bigl[\E \exp\Bigl(\sum_{i=1}^K |Z_i|^2/4 \Bigr)\Bigr]^{1/2}
e^{-BK/8}\\
&\le c_pK^p\Bigl[\Bigl(\frac{\Lambda_1}{\Lambda_0}\Bigr)^{1/2} \E(\exp(|Z_1|^2/4))^{1/2}e^{-B/8}\Bigr]^K.
\end{align*}
Since $\E e^{|Z_1|^2/4}=\sqrt 2$,  our choice of $B$ shows that the above is at most 
$$c_pK^p\exp(-BK/16).$$

(b)
By Lemma \ref{tildeabnd}, $g_{\Lambda_0}\leq \wt a(y,t)\leq g_{\Lambda_1}$,
so $G_{\Lambda_0}\geq \wt A(y,t)\geq G_{\Lambda_1}$.  Hence (b) follows
from (a).
\end{proof}

For $m\le K$ we let $a^m(y,t)$, respectively $\wt a^m(y,t)$, be the $m\times m$ matrices whose $(i,j)$ entry is $a_{ij}(\pi_{m,x'}(y),t)$, respectively $\wt a_{ij}(\pi_{m,x'}(y),t)$.  We use $A^m(y,t)$ and $\wt A^m(y,t)$ to denote their respective inverses.

The main theorem of this section is the following.

\bet\label{GLM-main} 
Suppose \eqref{Ehyp} holds with $\beta>3-\alpha$.
Let $w'=G(t)^{1/2}(y-x')$. 
Then there exists a constant $c_{1}$ depending on $\al$, $\beta$, $\kappa_\beta$, $p$,
$\Lambda_0$, and $\Lambda_1$ but not $K$, such that for all $t>0$ and $x\in\R$:

\noindent(a) For all $1\le j\le K$, 
\begin{align*} \int_{\R^K} &|w'_j|^{2p}Q_K(w',\wt A(y,t))\,dw'\\
&\le c_{1}\Bigl[\int_{\R^j}|w'_j|^{2p}Q_j(w',\wt A^j(y,t))\,dw'+1\Bigr].
\end{align*}

\noindent(b) \[\int_{\R^K} Q_K(w',\wt A(y,t))\,dw'\le c_{1},\] and 
\[\int_{\R^K} Q_K(y-x',A(y,t))\,dy\le c_{1}.\]
\eet

\ber {\rm This is one of the more important theorems of the paper.  In the proof of (a) we
will define a geometrically decreasing sequence $K_0,...,K_N$ with $K_0=K$
and $K_N=j$ and let $C_m$ be the expression on the right-hand side of (a) but with $K_m$ in place of $K$ and $\wt A^{K_m}$ in place of $\wt A$. We will bound $C_{m}$ inductively in terms of $C_{m+1}$ by
using Lemma \ref{U1.5} and Proposition \ref{Q5A}. This will give (a) and reduce (b) to the
boundedness in the $K=1$ case, which 
is easy to check. }
\eer

\begin{proof}[Proof of Theorem \ref{GLM-main}]
All constants in this argument may depend on $\al, \beta, \kappa_\beta, \Lambda_0, 
\Lambda_1$, and $p$. 
Let $K_0, K_1, \ldots, K_N$ be a decreasing sequence of positive integers
such that $K_0=K$, $K_N=j$, and $\frac54 \leq K_m/K_{m+1}\leq 4$ for each
$0\leq m< N$.

Let
\bee\label{n63-E1A}
C_m =\int |w'_j|^{2p} Q_{K_m}(w',\wt A^{K_m}(y,t))\, dw'.
\eee
Our plan is to bound $C_m$ inductively over  $m$.
 Write
\begin{align}
\nn C_m&=\int_{S^c_{B_0,K_m}} |w'_j|^{2p} Q_{K_m}(w', \wt A^{K_m}(y,t))\, dw'\\
\nn&\qq + \int_{S_{B_0,K_m}} |w'_j|^{2p} Q_{K_m}(w', \wt A^{K_m}(y,t))\, dw'\\
\label{n63-E3}&=I_1+I_2.
\end{align}
Assume $m<N$. We can bound $I_1$ using Lemma \ref{U1.5} and conclude
\bee\label{n63-E2}
I_1\leq c_pK^p_m e^{-B_0K_m/16}\leq c'_pe^{-B_0K_m/17}.
\eee

Turning to $I_2$, we see that by our hypothesis on $a$,
we have
\begin{align*}
|a_{ij}^{K_m}(y,t)- & a_{ij}^{K_m}(\pi_{K_{m+1}}(y),t)|\\
&\leq \kappa_\beta \sum_{k=K_{m+1}+1}^{K_m} |w_k|^\al k^{-\beta}\\
&= \kappa_\beta \sum_{k=K_{m+1}+1}^{K_m} |w'_k|^\al g_{kk}(t)^{\al/2} k^{-\beta}\\
&\leq c_1( t^{\al/2}\wedge K_m^{-\alpha}) \norm{w'}^\al \Big[
\sum_{k=K_{m+1}+1}^{K_m} k^{-2\beta/(2-\al)}\Big]^{(2-\al)/2}.
\end{align*}
In the last line we use H\"older's inequality and the bound
\bee\label{gbound} g_{kk}(t)=\int_0^te^{-2\lam_ks}\,ds\le t\wedge(2\lam_k)^{-1}\le c_2 (t\wedge k^{-2}),
\eee
by \eqref{s7-E-1}.  We also used the geometric decay of the $\{K_m\}$.

If $w'\in S_{B_0,K_m}$ so that $\norm{w'}^{\al}\leq (B_0\Lambda_1 K_m)^{\al/2}$, some elementary arithmetic 
shows there is a constant $c_3$ so that
\begin{align}\label{n63-E2A}
\nonumber|a_{ij}^{K_m}(y,t)-a_{ij}^{K_m}(\pi_{K_{m+1}}(y),t)|
&\leq c_3(t^{\al/2}\wedge K_m^{-\alpha}) K_m^{\al/2} [K_m^{1-(2\beta/(2-\al))}]^{(2-\al)/2}\\
&\leq c_3(t^{\al/2}\wedge K_m^{-\alpha}) K_m^{1-\beta}.
\end{align}
Set $\delta=c_3K_m^{1-\beta-\alpha}$. 
We now apply Proposition \ref{Q5A} for $w'\in S_{B_0,K_m}$ with 
$a=a^{K_m}(y,t)$ and $b=a^{K_m}(\pi_{K_{m+1}}(y),t)$.
In view of \eqref{easyschur} and \eqref{n63-E2A}, we may take
$$\theta=\Lambda_0^{-1}K_m^2\delta\qq\mbox{and}\qq \phi=
\Lambda_0^{-2}\Lambda_1 B_0 K^2_m \delta,$$
so that
$$\theta\lor \phi\leq c_3 K_m^{3-\beta-\alpha}\le c_3.$$
Proposition \ref{Q5A} shows that for
$w'\in S_{B_0,K_m}$,
\bee\label{Qratiobnd}\Bigl|\frac{Q_{K_m}(w', \wt A^{K_m}(y,t))}{Q_{K_m}(w',\wt A^{K_m}(\pi_{K_{m+1}}(y),t))}-1\Bigr|
\leq c_4K_m^{3-\beta-\alpha}.
\eee
Therefore we have
$$I_2\leq (1+c_4  K_m^{3-\beta-\alpha})
\int |w'_j|^{2p} Q_{K_m}(w', \wt A^{K_m}(\pi_{K_{m+1}}(y),t))\, dw'.$$

Recall $m+1\leq N$ so that $j\leq K_{m+1}$. 
Integrate over  $w'_{K_m}$  using Lemma \ref{Q3}, then over $w'_{K_m+1}$
using Lemma \ref{Q3} again, and continue until we have integrated
over 
$w'_{ K_{m+1}+1}$
 to see that 
\bee\label{n63-E301}
\int |w'_j|^{2p} Q_{K_m} (w', \wt A^{K_m}(\pi_{K_{m+1}}(y),t))\, dw'=C_{m+1},
\eee
and hence
\bee\label{n63-E302}
I_2\leq (1+c_4K_m^{3-\beta-\alpha})C_{m+1}.
\eee
This and \eqref{n63-E2} together show that \eqref{n63-E3}
implies that for $0\le m<N$,
\bee\label{n63-E4}
C_m\leq c'_p e^{-B_0K_m/17}+
(1+c_4  K_m^{3-\beta-\alpha}) C_{m+1}.
\eee
This and a simple
induction imply
\begin{align}
C_0&\leq \exp\Big( c_4  \sum_{m=0}^{N-1} K_m^{3-\beta-\alpha}\Big)
C_{N}\label{n63-E5}\\
&\qq +c'_p \sum_{m=0}^{N-1} e^{-B_0K_m/17}
\exp\Big(\sum_{\ell=1}^{m-1} c_4K_\ell^{3-\beta-\alpha}\Big)\nn\\
&\leq c_5(p)[C_N+1],\nn
\end{align}
since $\beta>3-\alpha$.  Part (a) follows.

For (b), we may apply (a) with $p=0$ and $j=1$ to get
\begin{align}
\int Q_K(w', & \wt A(y,t))\, dw'\label{n63-E8}\\
&\leq c_6\Bigl[\int_{-\infty}^\infty Q_1(w',\wt A^1(y,t))\, dw+1\Bigr].
\nn
\end{align}
Recall from Lemma~4.5 that the scalar $\wt A^1(y,t)$ satisfies $(\Lambda_1)^{-1}
\leq |\wt A^1(y,t)|\leq (\Lambda_0)^{-1}$
and so the above integral is at most
$$\Big(\frac{\Lambda_1}{\Lambda_0}\Big)^{1/2} \int Q_1(w, \Lambda_1 t)\, dw
=(\Lambda_1/\Lambda_0)^{1/2}.$$
The first bound in (b) follows from this and \eqref{n63-E8}. Using the change of variables
$w'=G(t)^{1/2}w$, we see that the second integral in (b) equals
the first.
\end{proof}

\bep\label{GLM2} Under the hypotheses of Theorem~\ref{GLM-main},\\
$$\int Q_K(y-x', A^K(y,t))\, dy\to 1$$ as $t\to 0$,
uniformly in $K$ and $x$.
\eep

\begin{proof} We will use the notation of the proof of Theorem \ref{GLM-main} with
$j=1$, $p=0$, and $t<1$. Using the change of variables $w'=G(t)^{1/2}(y-x')$, it suffices
to prove
$$\int Q_K(w', \wt A^K(y,t))\, dw'$$
converges to 1 uniformly as $t\to 0$. 

We define a decreasing sequence
$K_0, \ldots, K_N$ as in the proof of Theorem \ref{GLM-main}
with $K_0=K$ and $K_N=1$, we let
$$C_m(t)=\int Q_{K_m}(w', \wt A^{K_m}(y,t))\, dw',$$
we let $R>0$ be a real number to be chosen later, and we write
\bee\label{GLM2-E1}
|C_0(t)-1|\leq |C_N(t)-1|+\sum_{m=0}^{N-1}
|C_m(t)-C_{m+1}(t)|.
\eee
We will bound each term on the right hand side of \eqref{GLM2-E1}
appropriately, and that will complete the proof. 

Using \eqref{n63-E301} and with $S_{B,K}$ defined by \eqref{GLM-E101}, we write
\begin{align*}
|C_m&(t)-C_{m+1}(t)|\\
& \leq \int _{S_{R,K_m}^c}
[Q_{K_m}(w', \wt A^{K_m}(y,t))
+ Q_{K_m}(w', \wt A^{K_m}(\pi_{K_{m+1}}(y),t))]
\, dw' \\
&\qq +\int_{S_{R,K_m}}
\Big| \frac{Q_{K_m}(w', \wt A^{K_m}(y,t))}{Q_{K_{m}}(w', \wt A^{K_{m}}(\pi_{K_{m+1}}(y),t))}
-1\Big| \\
&\qq \qq \times  Q_{K_{m+1}}(w', \wt A^{K_{m}}(\pi_{K_{m+1}}(y),t))\, dw'\\
&=J_1(t)+J_2(t).
\end{align*}
By Lemma \ref{U1.5}(a), 
$$J_1(t)\leq c_1 e^{-c_2RK_m}.$$
Choose $0<\eta<\beta-(3-\alpha)$ and note that \eqref{n63-E2A} implies
there exists $c_2=c_2(R)$ such that
\[|a_{ij}^{K_m}(y,t)-a_{ij}^{K_m}(\pi_{K_{m+1}}(y),t)|\le c_2t^\eta K_m^{1-\beta-\alpha+\eta}\equiv\delta.\]
Follow the argument in the proof of Theorem \ref{GLM-main} with this value of $\delta$ to see that
\begin{align*}
J_2(t)&\leq c_2 t^{\eta/2} K_m^{3-\beta-\alpha+\eta} \int Q_{K_{m+1}}(w', \wt A^{K_{m+1}}
(y,t))\, dw'\\
&=c_2 t^{\eta/2} K_{m}^{3-\beta-\alpha+\eta} C_{m+1}(t)\\
&\leq c_{3} t^{\eta/2} K_m^{3-\beta-\alpha+\eta}.
\end{align*}
We used the uniform boundedness of $C_{m+1}$ from Theorem~\ref{GLM-main} for the last inequality.

A very similar argument shows that
$$\Big|C_N(t)-\int Q_1(w', \wt A^1(x',t))\, dw'\Big|\leq c_4e^{-c_4R}+c_5t^{\al/2},$$
where $c_5$ depends on $R$.
For example, in bounding the analog of $J_2(t)$, we may now take $\delta=c_6R^\alpha t^{\alpha/2}$ by adjusting the argument leading up to \eqref{n63-E2A}. 
Now use that
$Q_1(w', \wt A(x',t))$ is the density of a normal random variable, so that
$\int Q_1(w', \wt A(x',t))\, dw'=1$.
Substituting in \eqref{GLM2-E1}, we obtain
\begin{align*}
|C_N(t)-1|&\leq c_4e^{-c_4R}+c_5t^{\al/2} + \sum_{m=0}^{N-1} [c_1 e^{-c_2RK_m}+c_{3}
t^{\eta/2} K_m^{3-\beta-\alpha+\eta}]\\
&\leq  c_7 e^{-c_7R} +c_8(t^{\al/2}+t^{\eta/2});
\end{align*}
$c_8$ depends on $R$ but $c_7$ does not.
For the second inequality recall that $3-\beta-\alpha+\eta<0$ and the $K_m$ were chosen in the proof of
Theorem \ref{GLM-main} so that $\frac54\leq K_m/K_{m+1}\leq 4$.
Given $\eps>0$, choose $R$ large so that $c_7e^{-c_7R}<\eps$ and then
take $t$ small enough so that $c_8 (t^{\al/2}+t^{\eta/2})<\eps$.
\end{proof}

\bec\label{s6-C3.5} Assume the hypotheses of Theorem \ref{GLM-main}.
For any $p\geq 0$ there exists $c_1=c_{1}(p)>0$
such that
$$\int \norm{w'}^{2p} Q_K(w', \wt A(y,t))\, dw'\leq c_{1}K^p$$
for all $t>0$.
\eec

\begin{proof} Bound the above integral by
$$\int_{S_{B_0}^c} \norm{w'}^{2p} Q_K(w', \wt A(y,t))\, dw'
+(B_0\Lambda_1 K)^p\int _{S_{B_0}}
Q_K(w', \wt A(y,t))\, dw'.$$
The first term is at most $c_pK^p e^{-B_0K/16}$ by
Lemma \ref{U1.5}. 
The integral in the second term is at most
 $c_{1}$ by Theorem~\ref{GLM-main} (b).
The result follows.
\end{proof}

\bel\label{s6-L6.5} If $r>0,\gamma>1$, then there exists $c_{1}=c_{1}(r,\gamma)$ such that  for all $N$,
$$\sum_{m=1}^N \frac{m^r}{1+|m-k|^\gamma}
\le c_{1} \Big[N^{(1+r-\gamma)^+}+ 1_{(\gamma=1+r)} \log N
+k^r\Big].$$
\eel

\begin{proof} The above sum is bounded by
$$c_2\Big[\sum_{m=1}^N \frac{(m-k)^r}{1+|m-k|^\gamma}+
k^r \sum_{m=1}^N \frac{1}{1+|m-k|^\gamma}\Big].$$
The first term is at most 
$c_3\sum_{n=1}^N n^{r-\gamma}$
and the second term is at most $c_4k^r$. The result
follows.
\end{proof}

For the remainder of this subsection, except for
Theorem \ref{s6-T6.6}, we take $p\geq 1/2$, $\al>1/2$,
$\gamma>3/2$, $\beta>(2-\al/2+p)\vee (3-\alpha)$, and assume
\eqref{gammacond} holds.
With a bit of additional work the condition on $\gamma$ may be weakened to $\gamma>1$ but in Section~\ref{S:secmainest} we will need stronger conditions on $\gamma$ so we made no attempt to optimize here.

For $p\geq 1/2$  and $f:\R^K\to \R$ define
$$\norm{f(w)}_{2p}=\Big[\int |f(w')|^{2p}
Q_j(w', \wt A(y,t))\, dw'\Big]^{1/2p},$$
the $L^{2p}$ norm of $f$.

We start with a rather crude bound.
We write $\wt A w'$ for $\wt A(y,t)w'$.

\bel\label{s6-L6.7}
There exists $c_{1}$ such that for all $1\le k\le j\le K$,
$$\norm{(\wt A w')_k}_{2p}\leq c_{1} j^{1/2}.$$
\eel

\begin{proof}  By \eqref{gammacond} and Lemma \ref{JAP5} we have
\begin{align*}
\norm{(\wt A w')_k}_{2p}
&\leq c_2\Big\Vert \sum_{m=1}^j \frac{|w_m'|}{1+|m-k|^\gamma}\Big\Vert_{2p}\\
&\leq c_3\sum_{m=1}^j \Big ( \frac{1}{1+|m-k|^\gamma}\Big) 
\norm{w_m'}_{2p}.
\end{align*}
We can use Corollary \ref{s6-C3.5} with $K=j$ to  bound
$\norm{w'_m}_{2p}$ by
$$\norm{(\norm{w'})}_{2p}\leq c_4j^{1/2} .$$
 The bound follows. \end{proof}

\bel\label{s6-L6.8}
Assume there exists $c_{1}>0$ such that
\bee\label{s6-E6.8.1}
\int |(\wt A(y,t)w')_k|^{2p} Q_j(w',\wt A(y,t))\, dw'
\leq c_{1} \eee
for all $j\geq k\geq ((j/2)\vee 2)$ and $t>0$. Then there is a constant
$c_{2}$, so that for all $1\le j\le K$ and all $t>0$,
\bee\label{s6-E2}\int |w'_j|^{2p}Q_K(w',\wt A(y,t))\, dw'\leq c_{2}.
\eee
\eel

\begin{proof} If $z=\wt A(y,t)w'$, then by Lemma \ref{JAP4}
\begin{align}
\norm{w'_j}_{2p}&=\Big \Vert \sum_{k=1}^j \wt a_{jk}z_k\Big\Vert_{2p}\label{s6-E3}\\
&\leq \sum_{k=1}^j \frac{\kappa_\gamma}{1+|k-j|^\gamma} \norm{z_k}_{2p}.\nn
\end{align}
Use Lemma \ref{s6-L6.7} to bound $\norm{z_k}_{2p}$ for
$k\leq (j/2)\vee 1$ and \eqref{s6-E6.8.1} to bound it for $k> (j/2)\vee 1$. This leads to
\begin{align*}
\norm{w'_j}_{2p}
&\leq c_3\Big[\sum_{k=1}^{\lfloor j/2 \rfloor\vee 1} j^{-\gamma} j^{1/2} +\sum_{k=\lceil j/2 \rceil\vee 2 }^j \Big( \frac{1}{1+|k-j|^\gamma}\Big)
\Big]\leq c_4,
\end{align*}
where $\gamma>3/2$ is used in the last line. This gives \eqref{s6-E2}.
\end{proof}

In order to apply Lemma \ref{s6-L6.8} we need
to establish \eqref{s6-E6.8.1}; to do so 
we argue in a way similar to that of Theorem \ref{GLM-main}. For $j\geq k$,
as in \eqref{s6-E6.8.1}, define $\ol \pi:\R^j\to \R^j$
by
$$\ol \pi(y_1, \ldots, y_j)=(y_1, \ldots, y_{k-1}, x'_k, y_{k+1}, \ldots, y_j)$$
and
$$b(y,t)=a(\ol \pi(y),t), \qquad B(y,t)=A(\ol \pi(y),t).$$
As usual 
$$\wt{b}(y,t)=G(t)^{1/2} b(y,t) G(t)^{1/2}$$
with inverse
$$\wt{B}(y,t)=g(t)^{1/2} B(y,t) g(t)^{1/2}.$$

\bel\label{s6-L6.9} There exists $c_{1}$ such that for all
$K\ge j\geq k\geq j/2>0$,
$$\int |(\wt{B}w')_k|^{2p} [Q_j(w', \wt A(y,t))-Q_j(w', \wt{B}(y,t))]\, dw'
\leq c_{1}.$$
\eel

\begin{proof} As usual, $w'=G(t)^{1/2} (y-x')$. If $j,k$ are as above, then by
\eqref{Ehyp} and \eqref{gbound}
\begin{align}
|a_{mn}(y,t)-b_{mn}(y,t)|&\leq \kappa_\beta |w_k|^\al k^{-\beta}\nn\\
&\leq c_2  \norm{w'}^\al k^{-\al-\beta}\label{s6-E6.5}
\end{align}
by \eqref{gbound} and $k\ge 2$.
So for $w'\in S_{B_0,j}$ we can use $k\geq j/2$ to conclude
$$|a_{mn}(y,t)-b_{mn}(y,t)|\leq c_3 k^{-\al/2-\beta},$$
and therefore using $k\geq j/2$ again,
$$\norm{a(y,t)-b(y,t)}_s\leq 2c_3 k^{1-\al/2-\beta}.$$
For $w'\in S_{B_0,j}$ we may therefore apply Proposition \ref{Q5A}
with
\bee\label{s6-E34}
\theta+ \phi\leq c_4 k^{2-\al/2-\beta}\leq c_4.
\eee
It follows from Proposition \ref{Q5A} and the first inequality in
\eqref{s6-E34} that
\bee\label{s6-E101}
\Big|\frac{Q_j(w', \wt A(y,t))}{Q_j(w', \wt{B}(y,t))}-1\Big|
\leq c_5k^{2-\al/2-\beta}\quad\hbox{ for $w'\in S_{B_0,j}$.}
\eee

By our off-diagonal bound \eqref{gammacond} and
Lemma \ref{JAP5} we have
\bee\label{s6-E102}
|\wt{B}_{km}|\leq c_6(1+|k-m|^\gamma)^{-1},
\eee
and so (the constants below may depend on $p$)
\begin{align}
|(\wt{B}w')_k|^{2p}&
\leq \Big|\sum_{m=1}^j \wt{B}^2_{km}\Big|^p \norm{w'}^{2p}\label{s6-E103}\\
&\leq c_7 \norm{w'}^{2p}.\nn
\end{align}

Use \eqref{s6-E101} and \eqref{s6-E103} to bound the required integral by
\begin{align*}
\int _{S_{B_0}^c} &|(\wt{B}w')_k |^{2p}  Q_j(w', \wt A(y,t))\, dw'\\
&\qq +\int_{S_{B_0}} |(\wt {B}w')_k|^{2p} Q_j(w', \wt{B}(y,t))\, dw'
\, c_5 
k^{2-\al/2-\beta}\\
&\leq c_7 \int_{S_{B_0}^c} \norm{w'}^{2p} Q_j(w', \wt A(y,t))\, dw'\\
&\qq +c_8  k^{2-\al/2+p-\beta} \int_{S_{B_0}} Q_j(w', \wt{B}(y,t))\, dw'.
\end{align*}
The first term is at most $c_p j^p e^{-B_0j/16}$ by 
Lemma \ref{U1.5}, and the last term is bounded by
$c_9 k^{2-\al/2+p-\beta},$
thanks to Theorem \ref{GLM-main}. Adding the above
bounds gives the required result because $\beta\geq 2-\al/2+p$.
\end{proof}

\bel\label{s6-L6.10} There exists a  constant $c_{1}$ such that for
all $j\geq k\geq (j/2)\vee2$,
$$\int |((\wt { A}-\wt{B})w')_k|^{2p} Q_j(w', \wt A(y,t))\, dw'
\leq c_{1}.$$
\eel

\begin{proof} We use
$$\norm{\wt A-\wt B}=\norm{\wt A({\wt b}-\wt a){\wt B}}
\leq \norm{\wt A}\, \norm{\wt b-\wt a}\, \norm{\wt B}.$$
Lemma \ref{tildeabnd} implies
$$\norm{ \wt B}\, \norm{\wt{A}}\leq \Lambda_0^{-2},$$
and Lemma \ref{Q4} and \eqref{s6-E6.5} show that
$$\norm{{\wt b}-\wt a}\leq \norm{ b-a}_s\leq
c_2\norm{w'}^\al k^{1-\beta-\al}.$$
These bounds give
\begin{align*}
\int | & ((\wt A(y,t)-{\wt B}(y,t))w')_k|^{2p} Q_j(w', \wt A(y,t))\, dw'\\
&\leq \int \norm{\wt A(y,t)-{\wt B}(y,t)}^{2p}
\norm{w'}^{2p} Q_j(w', \wt A(y,t))\, dw'\\
&\leq c_3 k^{2p(1-\beta-\al)} \int \norm{w'}^{2p(1+\al)}
Q_j(w', \wt A(y,t))\, dw'.
\end{align*}
By Corollary \ref{s6-C3.5} this is at most
$c_4 k^{p(3-2\beta-\al)},$
which  gives the required bound since $\beta> 3-\al\ge(3-\al)/2$.
\end{proof}

\bet\label{s6-T6.6}  Assume \eqref{gammacond} for some
$\gamma>3/2$ and \eqref{Ehyp} for some $\al>1/2$ and 
$$\beta>( 2-\al/2+p)\lor (\tfrac72-\al/2)\lor (3-\al).$$ Let $p\ge 0$.
Let $w'=G(t)^{1/2}(y-x')$. Then there is a $c_{1}=c_{1}(p)$ such that for 
all $i\leq j\leq K$, $t>0$, and $x\in \R^K$,
\bee\label{w'p} \int |w'_j|^{2p}Q_K(w',\wt A(y,t))\, dw'\leq c_{1} ,
\eee
and 
\bee\label{wp}\int |w_j|^{2p}Q_K(w,A(y,t))\, dw\leq c_{1}\frac{t^p}{(1+\lambda_j t)^p} 
\le c_{1} t^p.\eee
\eet

\begin{proof} Consider \eqref{w'p}. First assume $p\geq 1/2$.
As $\beta>3-\al$, Theorem \ref{GLM-main}(a) allows us to assume
$K=j$.
Lemma \ref{s6-L6.8} reduces the proof to 
 establishing \eqref{s6-E6.8.1} in Lemma \ref{s6-L6.8} for $j$ and $k$ as in that result,
so assume $j\geq k\geq (j/2)\vee 2$. Lemmas \ref{s6-L6.9}
and \ref{s6-L6.10} imply that
\begin{align}
\int & |(\wt A(y,t)w')_k|^{2p} Q_j(w', \wt A(y,t))\, dw'\label{s6-E104}\\
&\leq c_2\Big[\int |\wt A(y,t)-\wt{B}(y,t))w')_k|^{2p} 
Q_j(w', \wt A(y,t))\, dw'
\nn\\
&\qq + \int |(\wt{B}(y,t)w')_k|^{2p} Q_j(w', \wt A(y,t))\, dw'\Big]\nn\\
&\leq c_3\Big[1+
\int |(\wt{B}(y,t)w')_k|^{2p} Q_j(w', \wt{B}(y,t))\, dw'\Big]\nn\\
&\equiv c_4[1+I].\nn
\end{align}

To evaluate the integral $I$, note that
\begin{align*}
I&=\int |\wt B_{kk}(y,t)|^{2p}\,\cdot
\Big| w'_k+\sum_{m\ne k} \frac{\wt B_{km}(y,t)w'_m}{\wt B_{kk}(y,t)}\Big|^{2p}\\
&\qq\qq \times Q_j(w', \wt B(y,t))\, dw'.
\end{align*}
Changing the indices in Lemma \ref{Q3} with $\wt a$ and $\wt b$ playing the roles of $a$ and $b$, respectively, we see that
provided we hold the coordinates $\wh w=(w'_j)_{j\ne k}$ fixed, if $\wh y=(y_j)_{j\ne k}$ and $\wh B(\wh y,t)$ is the
inverse of $(\wt b_{mn}(y,t))_{m\ne k, n\ne k}$, then $Q_j(w', \wt B(y,t))/Q_{j-1}(\wh w,\wh B(\wh y,t))\, dw'$ 
as a function of $w'_k$ is the density of a normal random variable with mean
$$\mu=-\sum_{m\ne k} \frac{\wt B_{km}(y,t)}{\wt B_{kk}(y,t)}$$ and
variance $\sigma^2=\wt B_{kk}(y,t)^{-1}$. So if  we integrate over $w'_k$, Lemma \ref{Q3}
implies
\begin{align*}
I&=\int |\wt{B}_{kk}(y,t)|^p c_p Q_{j-1}(\wh w, \wh B(\wh y,t))\, d\wh w\\
&\leq c_p\int Q_{j-1}(\wh w, \wh B(\wh y,t))\, d\wh w.
\end{align*}
Finally we use
Theorem 6.3(b) to bound the above integral by
$c'_p$. Put this bound into \eqref{s6-E104} to complete
the proof of \eqref{w'p} when $p\ge 1/2$.

For $p<1/2$, we write
$$ \int |w'_j|^{2p}Q_K(w',\wt A(y,t))\, dw'\leq 
 \int (1+|w'_j|)    Q_K(w',\wt A(y,t))\, dw'$$
and apply the above and Theorem \ref{GLM-main}(b).

The change of variables $w'=G(t)^{1/2}w$ shows that
$$\int |w_j|^{2p} Q_j(w, A(y,t))\, dw
=g_{jj}(t)^p \int |w'_j|^{2p} Q_j(w', \wt A(y,t))\, dw'.$$
Now use \eqref{ETE1}   to see that
$$g_{jj}(t)\leq \frac{2t}{1+\lambda_j t}.$$
This and \eqref{w'p} now give \eqref{wp}.

\end{proof}

\section{A second derivative estimate}\label{S:secder}

We assume $0\le \lambda_1\le \lam_2\le ...\le \lam_K$ satisfies \eqref{s7-E-1} for all $i\le K$.  
Our goal in this section is to bound the second derivatives
$$D_{jk} Q_K(y-x',A(y,t))=\frac{\del^2}{\del x_j \del x_k} Q_K(y-x', A(y,t))$$
uniformly in $K$.
Here $a(y,t)$ and $A(y,t)=a(y,t)^{-1}$ are as in Section \ref{S:GL}, and we assume \eqref{Ehyp} for appropriate $\beta$
and \eqref{gammacond} for $\gamma>3/2$ throughout. The precise conditions on $\beta$ will be specified in each of the results below. The notations $A^m$, $\wt A^m$, $\wt A$ from Section \ref{S:GL} are also used.

A routine calculation shows that for $j,k\leq K$,
\begin{align}\label{s7-E-1A}
D_{jk} Q_K(y-x', A(y,t))&=e^{-(\lam_j+\lam_k)t} S_{j,k}(w,A(y,t))\\
\nonumber&\qq\times Q_K(w, A(y,t)),
\end{align}
where $w=y-x'$ and for a $K\times K$ matrix $A$,
\begin{align*}
S_{j,k}=S_{j,k}(w,A)&=\Big(\sum_{n=1}^K A_{jn}w_n\Big)
\Big(\sum_{n=1}^K A_{kn} w_n\Big)-A_{jk}\\
&=(Aw)_j(Aw)_k-A_{jk}.
\end{align*}
We use the same notation if $A$ is an $m\times m$ matrix for $m\leq K$, but then our
sums are up to $m$ instead of $K$.

We will need a bound on the $L^2$ norm of a sum of second derivatives. 
The usual change of variables $w'=G(t)^{1/2} (y-x')$ will reduce this to bounds on
$$I^K_{jk\ell}=\int _{\R^K} S_{j,j+\ell} S_{k,k+\ell}(w', \wt A(y,t))
\, Q_K(w', \wt A(y,t))\, dw'.$$
These bounds will be derived by induction
as in Theorem \ref{GLM-main} and so we introduce for $m\leq K$,
$$I^m_{jk\ell}=\int_{\R^m} S_{j,j+\ell}S_{k,k+\ell}(w', \wt A^m(y,t))
Q_m(w', \wt A^m(y,t))\, dw'.$$

As the argument is more involved than the one in the proof of Theorem \ref{GLM-main}, to simplify things we will do our induction from $m$ to $m-1$ rather than using
geometric blocks of variables. This leads to a slightly stronger condition
on $\beta$ in Proposition \ref{s7-P6} below than would otherwise be needed.

If $A$ is an $m\times m$ matrix, we set $A_{ij}=0$ if
$i$ or $j$ is greater than $m$. This means, for example, that 
$S_{j,k}(w,A)=0$ if $j\lor k>m$. In what follows $x$ is always fixed, all
bounds are uniform in $x$, and when integrating over $w'_j$, 
we will be integrating over $y_j=y_j(w'_j)$ as well.

Since $w'=G(t)^{1/2}w$ we have from \eqref{n63-E2A}
\bee\label{s7-E0}
|w_n|=g_{nn}(t)^{1/2}|w'_n|\leq 
\begin{cases}c_1(\sqrt t \wedge (n^{-1}))|w'_n|&\hbox{ if }n\ge 2\\
c_1\sqrt t |w'_n|&\hbox{ if }n=1.
\end{cases}
\eee

\bel\label{s7-L1} Assume $\beta>\frac{5}{2}$. There exists $c_{1}$ such that for all $m,j,k>0$ and
$\ell\geq 0$ satisfying $(j\lor k)+\ell\leq m\leq K$ and $m\ge 2$,
\begin{align*}
\int |S_{j,j+\ell}S_{k,k+\ell}(w', \wt A^m(y,t))
&-S_{j,j+\ell}S_{k,k+\ell}(w', \wt A^m(\pi_{m-1}(y),t))|\\
&\qq \times Q_m(w, \wt A^m(y,t))\, dw'\\
&\leq c_{1} m^{5/2-\beta-\al}.
\end{align*}
\eel

\begin{proof} Let $j,k,\ell$ and $m$ be as above.  The pointwise bound on $\wt A^m$ in Lemma \ref{JAP5} implies
\begin{align}
\label{s7-E1}|S_{k,k+\ell}&(w', \wt A^m(y,t))|\\
&\le c_2\Big[\sum_{n=1}^m \sum_{\nu=1}^m (1+|n-k|^\gamma)^{-1}
(1+|\nu-k-\ell|^\gamma)^{-1}|w'_n|\, |w'_{\nu}|\nn\\
&\qq +(1+\ell^\gamma)^{-1}\Big],\nn
\end{align}
and so 
\bee\label{s7-E1.5}
|S_{k,k+\ell}(w', \wt A^m(y,t))|
\leq c_{3}(\norm{w'}^2+1).
\eee
The triangle inequality gives
\begin{align}
|S_{j,j+\ell}&(w', \wt A^m(y,t))-S_{j,j+\ell}(w', \wt A^m(\pi_{m-1}(y),t))|\
\label{s7-E2}\\
&\leq |((\wt A^m(y,t)-\wt A^m(\pi_{m-1}(y),t))w')_j|\,
|(\wt A^m(y,t)w')_{j+\ell}|\nn\\
&\qq +|(\wt A^m(\pi_{m-1}(y),t)w')_j|\,
|((\wt A^m(y,t)-\wt A^m(\pi_{m-1}(y),t))w')_{j+\ell}|\nn\\
&\qq +|\wt A^m_{j,j+\ell}(y,t)-\wt A^m_{j,j+\ell}(\pi_{m-1}(y),t)|.\nn
\end{align}
By \eqref{normbnd} in Lemma \ref{Q4}, for $i\le m$,
\begin{align}
|(\wt A^m(y,t)-&\wt A^m(\pi_{m-1}(y),t)w')_i|\label{s7-E3}\\
&
\nn\leq \norm{\wt A^m(y,t)-\wt A^m(\pi_{m-1}(y),t)}\, \norm{w'}\\
&\leq \Lambda_0^{-2} \norm{\wt a^m(y,t)-\wt a^m(\pi_{m-1}(y),t)}_s\, \norm{w'}\nn\\
&\leq \Lambda_0^{-2} \sum_{j=1}^m \kappa_\beta |w_m|^\al m^{-\beta}
\norm{w'}\nn\\
&\leq c_4  \norm{w'} \, |w'_m|^\al m^{1-\beta-\al},\nn
\end{align}
where \eqref{s7-E0} and $m\ge 2$ are used in the last line.

Lemma \ref{JAP5} implies that for $i\leq m$,
\bee\label{s7-E4}
|(\wt A^m(y,t)w')_i|\leq c_{5} \sum_{\nu=1}^m (1+|\nu-i|^\gamma)^{-1}
|w'_{\nu}|,
\eee
and \eqref{LAE3} together with the calculation in \eqref{s7-E3} implies
\begin{align}
|\wt A^m(y,t)_{j,j+\ell}-\wt A^m(\pi_{m-1}(y),t)_{j,+\ell}|
&\leq \Lambda_0^{-2} \norm{a^m(y,t)- a^m(\pi_{m-1}(y),t)}_s\label{s7-E5}\\
&\leq c_6 |w'_m|^\al m^{1-\beta-\al},\nn
\end{align}
as in \eqref{s7-E3} above. 

Now use \eqref{s7-E3}, \eqref{s7-E4}, and \eqref{s7-E5} in
\eqref{s7-E2} and then appeal to \eqref{s7-E1.5} to conclude that
\begin{align}
\int &|S_{j,j+\ell}(w', \wt A^m(y,t))-S_{j,j+\ell}(w', \wt A^m(\pi_{m-1}(y),t))|\,
|S_{k,k+\ell}(w', \wt A^m(y,t))|\nn\\
&\qq \times Q_m(w', \wt A^m(y,t))\, dw' \label{s7-E6}\\
&\leq c_7m^{1-\beta-\al}\Bigl\{ \int (
\norm{w'}^2+1)\, |w_m'|^\al Q_m(w',\wt A^m(y,t))\,dw'\nn\\
&\qq\qq+
\sum_{\nu=1}^m \Bigl((1+|\nu-j|^\gamma)^{-1}+(1+|\nu-j-\ell|^\gamma)^{-1}\Bigr)\nn\\
&\qq\qq\times\int |w'_{\nu}||w'_m|^\al[\norm{w'}^3+\norm{w'}]\Bigr\}Q_m(w',\wt A^m(y,t))\,dw'\nn
\end{align}
There are several integrals to bound but the one giving the largest contribution and requiring
the strongest condition on $\beta$ will be 
\[I=\int|w'_\nu|\,|w'_m|^\al\,\norm{w'}^3Q_m(w',\wt A^m(y,t))\,dw'.\]
Apply H\"older's inequality for triples with $p=\frac{1+\al}{1-\varepsilon}$, $q=\frac{1+\al}{\al(1-\vep)}$ and $r=\vep^{-1}$  to conclude
\begin{align*}
I&\le \Bigl[\int|w'_\nu|^pQ_m(w',\wt A^m(y,t))\,dw'\Bigr]^{1/p}\Bigl[\int |w'_m|^{\al q}Q_m(w',\wt A^m(y,t))\,dw'\Bigr]^{1/q}\\
&\qq\qq\times\Bigl[\int \norm{w'}^{3r}Q_m(w',\wt A^m(y,t))\,dw'\Bigr]^{1/r}\\
&\le c_8 m^{3/2}.
\end{align*}
Here we used Corollary~\ref{s6-C3.5}, Theorem~\ref{s6-T6.6} and the fact that $\beta>5/2$ 
means the hypotheses of this last result are satisfied for $\vep$ small enough.  
The other integrals on the right-hand side of \eqref{s7-E6} lead to smaller bounds
and so the left-hand side of \eqref{s7-E6} is at most
$c_9 m^{5/2-\beta-\al}.$
A similar bound applies with the roles of $j$ and $k$ reversed, and so the
required result is proved.
\end{proof}

\bel\label{s7-L2} Assume $\beta>2-(\al/2)$. There exists $c_{1}$ such that
for all $j,k,\ell,m$ as in Lemma \ref{s7-L1} and satisfying $2\leq m$,
\begin{align*}
\int |S_{j,j+\ell}&S_{k,k+\ell}(w',\wt A^m(\pi_{m-1}(y),t))|\\
&\qq \times|Q_m(w', \wt A^m(y,t))-Q_m(w',\wt A^m(\pi_{m-1}(y),t))|\, dw'\\
&\leq c_{1}m^{2-(\al/2)-\beta}.
\end{align*}
\eel

\begin{proof} Recall that $B_0$ is as in Lemma \ref{U1.5}.
Use \eqref{s7-E1.5} on $S^c_{B_0,M}$ and \eqref{s7-E1} on
$S_{B_0,m}$ to bound the above integrand by
\begin{align*}
c_2\Big[&\int_{S^c_{B_0,m}}(\norm{w'}^4+1)[Q_m(w',\wt A^m(y,t))+Q_m(w', \wt A^m(\pi_{m-1}(y),t))]\, dw'\\
&+\int_{S_{B_0,m}} \Big[\Big(\sum_{n=1}^m \sum_{\nu=1}^m (1+|n-k|^\gamma)^{-1} 
(1+|\nu-k-\ell|^\gamma)^{-1} |w'_n|\, |w'_{\nu}|\Big)+1\Big]\\
&\qq\qq \times \Big| \frac{Q_m(w',\wt A^m(y,t))}{Q_m(w', 
\wt A^m(\pi_{m-1}(y),t))}-1\Big|\, Q_m(w',\wt A^m(\pi_{m-1}(y),t))\Big]
\, dw'\Big]\\
&=c_2(I_1(t)+I_2(t)).
\end{align*}

By Lemma \ref{U1.5},
$$I_1(t)\leq c_3 m^2 e^{-B_0m/16}\leq c_4 e^{-B_0m/17}.$$
We bound $I_2(t)$ as in the proof of Theorem \ref{GLM-main} but with
$m$ in place of $K_m$. This requires some minor changes. Now for $w'\in S_{B_0,m}$ the $\delta$ 
coming from 
\eqref{n63-E2A} is less than or equal to  
$$c_5|w_m'|^\al m^{-\beta-\al}\leq c_6  m^{-\al/2-\beta},$$ and
$$\phi\lor \theta \leq c_7  m^{2-\al/2-\beta}\leq c_7.$$
So for
$w'\in S_{B_0.m}$, applying Proposition \ref{Q5A} as before, we get
$$\Big|\frac{Q_m(w',\wt A^m(y,t))}{Q_m(w', \wt A^m(\pi_{m-1}(y),t))}-1\Big|\leq c_8 m^{2-\al/2-\beta},$$
and therefore
\begin{align*}
I_2(t)&\leq c_8m^{2-\al/2-\beta}
\int \Big[\Big(\sum_{n=1}^m \sum_{\nu=1}^m (1+|n-k|^\gamma)^{-1}
(1+|\nu-k-\ell|^\gamma)^{-1}\\
&\qq \qq\qq\qq \times|w'_n|\, |w'_{\nu}|\Big)+1\Big]
 Q_m(w',\wt A^m(\pi_{m-1}(y),t))\, dw'\\
&\leq c_9 m^{2-\al/2-\beta},
\end{align*}
where  Theorem \ref{s6-T6.6} and Cauchy-Schwarz are used in the last line.
The lower bound on $\beta$ shows the hypotheses of Theorem \ref{s6-T6.6} are satisfied.
Combining  the bounds on $I_1(t)$ and $I_2(t)$ completes the proof.
\end{proof}

Note that if $Z$ is a standard normal random variable, then
 $\E [(Z^2-1)^2]=2$. 

\bel\label{s7-L3} If $j,k,\ell,m$ are as in Lemma \ref{s7-L1}
and  for $w'\in \R^m$,
$$r_{m-1}w'=(w_1', \ldots, w'_{m-1}),$$
then
\begin{align}
\int &S_{j,j+\ell} S_{k,k+\ell}(w',\wt A^m(\pi_{m-1}(y),t))
\frac{Q_m(w', \wt A^m(\pi_{m-1}(y),t))}{Q_{m-1}(r_{m-1}w',
\wt A^{m-1}(y,t))}\, dw'_m\nn\\
&=\Big\{S_{j,j+\ell}S_{k,k+\ell}(w', \wt A^{m-1}(y,t))1_{((j\lor k)+\ell\leq m-1)}\Big\}\label{s7-E7.5}\\
&\qq+\Big\{[\wt A^m_{jm}(\pi_{m-1}(y),t)(\wt A^{m-1}(y,t) r_{m-1}w')_{j+\ell}\nn\\
&\qq\qq\qq +\wt A^m_{j+\ell,m}(\pi_{m-1}(y),t)(\wt A^{m-1}(y,t)r_{m-1}w')_j]\nn\\
& \qq\qq \times [\wt A^m_{km}(\pi_{m-1}(y),t)(\wt A^{m-1}(y,t)r_{m-1}w')_{k+\ell}\nn\\
&\qq\qq \qq + \wt A^m_{k+\ell,m}(\pi_{m-1}(y),t)(\wt A^{m-1}(y,t) r_{m-1}w')_k]\nn\\
&\qq \qq\times \wt A^m_{mm}(\pi_{m-1}(y),t)^{-1}\Big\}\nn\\
&\qq +\Big\{2(\wt A^m_{jm}\wt A^m_{j+\ell,m} \wt A^m_{km}\wt A^m_{k+\ell,m})
(\pi_{m-1}(y),t) \wt A^m_{mm}(\pi_{m-1}(y),t)^{-2}\Big\}\nn\\
&=V^1(j,k,\ell,m)+V^2(j,k,\ell,m)+V^3(j,k,\ell,m).\nn
\end{align}
\eel

\begin{proof} We apply Lemma \ref{Q3} with $m$ in place of $m+1$ 
and $\wt a^m(\pi_{m-1}(y),t)$ playing the role of $a(y)$. Then under
\bee\label{s7-defG}
G_m(y,t)=\frac{Q_m(w', \wt A^m(\pi_{m-1}(y),t))}{Q_{m-1} (r_{m-1}w',\wt A^{m-1}(y,t))},
\eee
$w'_m$ has a normal distribution with mean
$$\mu=-\sum_{i=1}^{m-1} \frac{\wt A^m_{mi}(\pi_{m-1}(y),t)w'_i}{\wt A^m_{mm}(\pi_{m-1}(y),t)}$$
and variance $\sigma^2=\wt A^m_{mm}(\pi_{m-1}(y),w')^{-1}$.
Set $\wh w'_m=w_m-\mu$,
\begin{align*}
R_j^m &=\sum_{i=1}^{m-1} \wt A^m_{ji}(\pi_{m-1}(y),t) w'_i, \qq j\leq m,\\
R_j^{m-1}&=\sum_{i=1}^{m-1} \wt A^{m-1}_{ji}(y,t)w_i', \q\mbox{for } 
j\leq m-1,\qq 
R_m^{m-1}=0,\\
\mbox{and}\qq C_j&=\wt A^m_{mj}(\pi_{m-1}(y),t), \qq j\leq m.
\end{align*}
Lemma \ref{LAL100} with $a=\wt a^m(\pi_{m-1}(y),t)$ and $m$ in place of $m-1$
gives
$$\wt A^m_{ji}(\pi_{m-1}(y),t)=\wt A^{m-1}_{ji}(y,t)+C_jC_i\sigma^2, \qq j,i\leq m,$$
where we recall that by convention
 $\wt A^{m-1}_{ji}(y,t)=0$ if $i$ or $j$ is greater than
$m-1$. Therefore
$$R_j^m=R^{m-1}_j-C_j\mu, \qq j\leq m,$$
and so for $j,k,\ell,m$ as in the lemma,
\begin{align*}
S_{j,j+\ell}&S_{k,k+\ell}(w',\wt A^m(\pi_{m-1}(y),t))\\
&=[(R_j^m+C_jw'_m)(R^m_{j+\ell} +C_{j+\ell}w'_m)-\wt A^m_{j,j+\ell}(\pi_{m-1}(y),t))]\\
&\qq \times
[(R_k^m+C_kw'_m)(R^m_{k+\ell} +C_{k+\ell}w'_m)-\wt A^m_{k,k+\ell}(\pi_{m-1}(y),t))]\\
&=[(R^{m-1}_j+C_j \wh w'_m)(R^{m-1}_{j+\ell}+C_{j+\ell}\wh w'_m)
-\wt A^{m-1}_{j,j+\ell}(y,t)-C_jC_{j+\ell} \sigma^2]\\
&\qq\times 
[(R^{m-1}_k+C_k \wh w'_m)(R^{m-1}_{k+\ell}+C_{k+\ell}\wh w'_m)
-\wt A^{m-1}_{k,k+\ell}(y,t)-C_kC_{k+\ell} \sigma^2].
\end{align*}
Rearranging terms, we see that the above equals
\begin{align}
[R^{m-1}_j& R^{m-1}_{j+\ell}-\wt A^{m-1}_{j,j+\ell}(y,t)\nn\\
&\qq +\wh w'_m(C_jR^{m-1}_{j+\ell} +C_{j+\ell}R^{m-1}_j)
+(|\wh w'_m|^2-\sigma^2)C_jC_{j+\ell}]\label{s7-E8}\\
\times&
[R^{m-1}_k R^{m-1}_{k+\ell}-\wt A^{m-1}_{k,k+\ell}(y,t)
+\wh w'_m(C_kR^{m-1}_{k+\ell} +C_{k+\ell}R^{m-1}_k)\nn\\
&\qq\q
+(|\wh w'_m|^2-\sigma^2)C_kC_{k+\ell}]\nn\\
=(&R^{m-1}_j R^{m-1}_{j+\ell}-\wt A^{m-1}_{j,j+\ell}(y,t))
(R_k^{m-1} R^{m-1}_{k+\ell} -\wt A^{m-1}_{k,k+\ell}(y,t))\nn\\
&\qq + |\wh w'_m|^2 (C_j R^{m-1}_{j+\ell} +C_{j+\ell} R^{m-1}_j)
(C_k R^{m-1}_{k+\ell} +C_{k+\ell} R^{m-1}_k)\nn\\
&\qq +(|\wh w'_m|^2-\sigma^2)^2 C_jC_{j+\ell} C_k C_{k+\ell}
+\mbox{ off-diagonal terms}.\nn
\end{align}
When we multiply each off-diagonal term by 
$G_m(y,t)$
and integrate over $w'_m$, we get zero. This is because the conditional
normal distribution of $w'_m$ under $G_m(y,t)$ implies
that each of
\begin{align*}
&\int \wh w'_m G_m(y,t) \, dw'_m,\\
&\int (|\wh w_m'|^2-\sigma^2) G_m(y,t)\, dw'_m, \mbox{ and }\\
&\int (\wh w_m') (|\wh w_m'|^2-\sigma^2) G_m(y,t)\, dw'_m
\end{align*}
equals zero.

Now integrate the remaining terms on the  right hand side of \eqref{s7-E8} with respect
to $G_m(y,t)\, dw'_m$, 
noting that $R_i^{m-1}$, $C_i$, and $\wt A^{m-1}_{ij}$ do not depend
on $w'_m$. Use the fact that
$$\int |\wh w'_m|^2 G_m(y,t)\, dw'_m=\sigma^2=\wt A^m(\pi_{m-1}(y),t)^{-1}$$
and
$$\int (|\wh w_m'|^2-\sigma^2)^2 G_m(y,m)\, dw_m'=2\sigma^4
=2\wt A^m(\pi_{m-1}(y),t)^{-2} $$
to obtain the desired expression. In particular note that
\begin{align*}
(R_j^{m-1}&R_{j+\ell}^{m-1}-\wt A^{m-1}_{j,j+\ell}(y,t))
(R_k^{m-1}R_{k+\ell}^{m-1}-\wt A^{m-1}_{k,k+\ell}(y,t))\\
&=S_{j,j+\ell}S_{k,k+\ell}(r_{m-1}w',\wt A^{m-1}(y,t))1_{((j\lor k)+\ell\leq m-1)}.
\end{align*}
\end{proof}

We treat $V^2$ and $V^3$ in \eqref{s7-E7.5} as error terms and
so introduce
\begin{align*}
E^1(j,k,\ell,m)&=\int_{\R^{m-1}} |V^2(j,k,\ell,m)|\, dw',\\
E^2(j,k,\ell,m)&=\int_{\R^{m-1}} |V^3(j,k,\ell,m)|\, dw',
\end{align*}
and
$$E(j,k,l,m)=E^1(j,k,\ell,m)+E^2(j,k,\ell,m).$$

We are ready for our inductive bounds on the integral $I^m_{jk\ell}$, defined at the beginning
of this section. 

\bep\label{s7-P4} Assume $\beta>\frac{7}{2}-\al$. There exists $c_{1}$ such that
for all integers $j,k,\ell$ such that $1\le j\le k\leq k+\ell\le K$,
\begin{align*}
I^K_{jk\ell}&\leq c_{1}(k+\ell)^{(7/2)-\al-\beta}
+\sum_{m=(k+\ell)\vee 2}^K E(j,k,\ell,m).
\end{align*}
\eep

\begin{proof} If $K\ge m\ge 2\vee (k+\ell)$, we can combine Lemmas \ref{s7-L1},
\ref{s7-L2} and \ref{s7-L3} to see that
\begin{align*}
I^m_{jk\ell}&\leq I^{m-1}_{jk\ell} 1_{(k+\ell\leq m-1)}+c_{2} m^{5/2-\beta-\al}+ c_{3} m^{2-\al/2-\beta}
+E(j,k,\ell,m).
\end{align*}
Therefore by induction
\begin{align}
\label{s7-E9}I^K_{jk\ell}&\leq I_{jk\ell}^{1\vee(k+\ell-1)} 1_{(k+\ell\leq 1\vee(k+\ell-1))}
+c_4 \sum_{m=2\vee(k+\ell)}^K 
m^{(5/2)-\beta-\al}\\
\nn&\qq +\sum_{m=2\vee(k+\ell)}^K E(i,j,k,\ell).
\end{align}

The first term in the above is $I^1_{110}1_{(k+\ell=1)}$.  
For $m=1$, $\wt A^1(y,t)$ is a scalar and an argument similar to that in (b) of Theorem~\ref{GLM-main} shows that 
\begin{align}\label{s7-E11}
I^1_{110}=&S_{1,1}(w',\wt A^1(y,t))^2\,Q_1(w',\wt A^1(y,t))dw'\\
\nn\le& c_5\int(1+\norm{w'}^4)Q_1(w',\wt A^1(y,t))dw'\qq\hbox{(by \eqref{s7-E1.5})}\\
\nn\le & c_6.
\end{align}

Use \eqref{s7-E11} to bound the first term in
\eqref{s7-E9} and then bound the second terms in the obvious manner to complete the proof.
\end{proof}

To use the above bound we of course will have to control the  $E(j,k,\ell,m)$'s.
If $\zeta>0$, set 
\bee\label{Jdefine}
J=J_\zeta(t)=\lceil (\zeta \log(t^{-1}+1))/t)^{1/2}\rceil.
\eee

\bel\label{s7-L5} Assume $\beta>3-(\al/2)$.
There exists a $c_{1}$ such that for all $0\le \ell\leq K$,
$$\sum_{1\leq j\leq k\leq J_\zeta(t)} \,\sum_{m=2\vee(k+\ell)}^K E(j,k,\ell,m)
\leq c_{1} J_\zeta(t).$$
\eel

\begin{proof} We consider $E^1(j,k,\ell,m)$. There is a product giving
rise to four terms, all of which are handled in a similar way. We consider
only
\begin{align*}
&E^1_1(j,k,\ell,m)\\
&\ =\int_{\R^{m-1}} |\wt A_{j+\ell,m}(\pi_{m-1}(y),t) (\wt A^{m-1}(y,t)w')_j
\wt A_{k+\ell,m} (\pi_{m-1}(y),t)\\
&\qq \times (\wt A^{m-1}(y,t)w')_k|\wt A^{m}_{mm}(\pi_{m-1}(y),t)^{-1}
 Q_{m-1}(w', \wt A^{m-1}(y,t))\, dw',
\end{align*}
as this is the worst term. Use the upper bound
on $\wt A^m_{ij}$ and the lower bound on $\wt A^m_{ii}$ from Lemma \ref{JAP5}
to see that
\begin{align*}
E^1_1(j,k,\ell,m)&\le c_2(1+|m-j-\ell|^\gamma)^{-1}(1+|m-k-\ell|^\gamma)^{-1}\\
&\qq \times \sum_{n=1}^{m-1} \sum_{\nu=1}^{m-1} (1+|n-j|^\gamma)^{-1}
(1+|\nu-k|^\gamma)^{-1}\\
&\qq \times \int |w_\nu|\, |w'_{\nu}| Q(w', \wt A^{m-1}(y,t))\, dw'.
\end{align*}
An application of Cauchy-Schwarz and Theorem \ref{s6-T6.6} shows that  for our value of $\beta$
the last integral is bounded by $c_3$.  This leads to
$$E^1_1(j,k,\ell,m)\leq c_4 (1+|m-j-\ell|^\gamma)^{-1}
(1+|m-k-\ell|^\gamma)^{-1}.$$
Now sum over $j,m$, and $k$ in that order to see that
{\allowdisplaybreaks
\begin{align*}
\sum_{1\leq j\leq k\leq J}& \sum_{m=2\vee(k+\ell)}^K E_1^1(j,k,\ell,m)\\
&\leq \sum_{k=1}^J \sum_{m=k+\ell}^K \sum_{j=1}^k
(1+|m-j-\ell|^\gamma)^{-1}(1+|m-k-\ell|^\gamma)^{-1} c_4\\
&\leq \sum_{k=1}^J \sum_{m=k+\ell}^K (1+|m-k-\ell|^{\gamma})^{-1} c_5\\
&\leq c_6 J .
\end{align*}
}
The other terms making up $E^1(j,k,\ell,m)$ are bounded in
a similar manner.

Consider now $E^2(j,k,\ell,m)$. Again the upper and lower
bounds in Lemma \ref{JAP5} and Theorem~\ref{GLM-main}(b) imply that for
$j\leq k\leq k+\ell\leq m$,
\begin{align*}
E^2(j,k,\ell,m)&\leq c_7 (1+|m-j|^\gamma)^{-1} (1+|m-k|^\gamma)^{-1}
(1+|m-j-\ell|^\gamma)^{-1}\\
&\qq \times (1+|m-k-\ell|^\gamma)^{-1}\\
&\leq c_7
(1+|m-j-\ell|^\gamma)^{-1}(1+|m-k-\ell|^\gamma)^{-1}.
\end{align*}
Again sum over $j$ then $m$ and then $k$ to see
$$\sum_{1\leq j\leq k\leq J} \sum_{m=2\vee(k+\ell)}^K E^2(j,k,\ell,m)\leq c_8J.$$

Combining the above bounds gives the required result.
\end{proof}

\bep\label{s7-P6} Assume $\beta>\frac{9}{2}-\al$. There exists $c_{1}$
so that for any $0\leq \ell\leq J$,
\begin{align*}
\int \Big(\sum_{j=1}^J e^{-\lam_j t-\lam_{j+\ell} t}& S_{j,j+\ell}(y-x', A(y,t))\Big)^2
Q_K(y-x', A(y,t))\, dy\\&\leq c_{1}J t^{-2}.
\end{align*} 
\eep

\begin{proof} As usual we set $w=g(t)^{1/2} w'$, which leads to
\begin{align*}
S_{j,j+\ell}(w, A(y,t))&=S_{j,j+\ell}(g(t)^{1/2} w',A(y,t))\\
&=G_{jj}(t)^{1/2} S_{j,j+\ell}(w', \wt A(y,t))G_{j+\ell,j+\ell}(t)^{1/2}.
\end{align*}
Let
$H_i(t)=e^{-\lam_i t} G_{ii}(t)^{1/2},$ so that
\bee\label{s7-E12}
0\leq H_i(t)=\Big(\int_0^te^{2\lam_i(t-s)}\,ds\Big)^{-1/2}\leq t^{-1/2}.
\eee

The integral we have to bound now becomes
\begin{align*}
\int\Big(\sum_{j=1}^J &H_j(t) S_{j,j+\ell}(w', \wt A(y,t))H_{j+\ell}(t)
\Big)^2\\
&\qq \times Q_K(w', \wt A(y,t))\, dw'\\
&=\sum_{j,k=1}^J H_j(t) H_k(t) H_{j+\ell}(t) H_{k+\ell}(t) I^K_{jk\ell}.
\end{align*}
Now use the upper bound on $H_i$, Lemma~\ref{s7-L5}, Proposition \ref{s7-P4} for $j\leq k$, and 
 symmetry in $(j,k)$ to bound the above by
\begin{align}
\label{s7-E13}
c_2t^{-2}&\Big\{
\sum_{j\le k\le J}\Bigl[(k+\ell)^{(7/2)-\beta-\al}
+ \sum_{m=2\vee(k+\ell)}^K E(j,k,\ell,m)\Bigr]\Big\}\nn\\
&\leq c_3t^{-2} J[\ell^{(9/2)-\beta-\al}+1]\nn
\end{align}
where Lemma \ref{s7-L5} and the condition on $\beta$ are used in the
last line.
\end{proof}

We need a separate (and much simpler) bound to handle the
absolute values of $D_{jk}Q_K(y-x',A(y,t))$ for $j\lor k\geq J_\zeta(t)$.

\bel\label{s7-L7} Assume $\beta>3-\frac{\al}{2}$. 
There exists  $c_{1}$ such that for all $i,j,k\leq K$ and $p\geq 0$,
$$\int|w'_i|^{2p} |S_{j,k}(w', \wt A(y,t))| Q_K(w', \wt A(y,t))\, dw'
\leq c_{1}.$$
\eel

\begin{proof} By  (\ref{s7-E1}) the above integral is at most
\begin{align*}
c_2\int\Bigl(\sum_{n=1}^m\sum_{\nu=1}^m (1+|n-j|^\gamma)^{-1}& (1+|\nu-k|^\gamma)^{-1} |w'_n|\, |w'_{\nu}|+1\Bigr)\\
&\qq\times
|w'_i|^{2p}Q_K(w', \wt A(y,t))\, dw'.
\end{align*}
Now apply Theorem \ref{s6-T6.6} and Cauchy-Schwarz 
to obtain the required bound. 
\end{proof}

The proof of the following is left to the reader.

\bel\label{s7-L8} There exists a  constant $c_{1}$ so that for all $\theta>0$, $r\geq 1$, 
$$\sum_{|j|+|k|\geq r} e^{-\theta j^2}e^{-\theta k^2}\leq \frac{c_{1}}{\theta } e^{-\theta r^2/4}.$$
\eel

\bep\label{s7-P99} Assume $\beta>3-\frac{\al}{2}$.  There exists $c_1$
such that for all $i,j,k$ and $p$ 
$$\int_{\R^K}|w_i|^{2p} |D_{jk}Q_K(y-x', A(y,t))|\, dy\leq c_{1}t^{-1+p}.$$
\eep

\begin{proof} As in the proof of Proposition \ref{s7-P6}, if
$H_i(t)=e^{-\lam_i t} G_{ii}(t)^{1/2}$, then the substitution $w=g(t)^{1/2}w'$ leads to
\begin{align*}
\int |w_i|^{2p}|D_{jk}&Q_K(y-x', A(y,t))|\, dy\\
&=\int |w_i|^{2p}e^{-(\lam_j+\lam_k)t}|S_{j,k}(w, A(y,t))| Q_K(w, A(y,t))\, dw\\
&\le t^pH_j(t)H_k(t)\int |w'_i|^{2p}|S_{j,k}(w', \wt A(y,t))|Q_K(w', \wt A(y,t))\, dw'\\
&\le c_{2}t^p  H_j(t) H_k(t),
\end{align*}
the last by Lemma \ref{s7-L7}.

A bit of calculus shows that
$$H_j(t)=\Big(\int_0^te^{2\lam_j(t-s)}\,ds\Big)^{-1/2}
\le e^{-\lam_j t/2} t^{-1/2}.$$
\end{proof}

\bep\label{s7-P9} Assume $\beta>3-\frac{\al}{2}$.  There are constants $\zeta_0$ and $c_{1}$
such that if $\zeta\geq \zeta_0$ and $J=J_\zeta(t)$, then
\begin{align*}
\sum_{j=1}^K \sum_{k=1}^K 1_{(j\lor k>J)}&\int_{\R^K} |D_{jk}Q_K(y-x', A(y,t))|\, dy\leq c_{1}(t+1)^{-2}.
\end{align*}
\eep

\begin{proof}
Using Proposition \ref{s7-P99},
the sum is at most
\begin{align*}
c_{2}\sum_{j=1}^K \sum_{k=1}^K 1_{(j\lor k>J)} e^{-(\lam_j+\lam_k)t/2} t^{-1}
&\leq c_{2}\sum_{j=1}^K \sum_{k=1}^K  1_{(j\lor k>J)}
e^{-c_3(j^2+k^2)t} t^{-1}\\
&\leq c_4 e^{-c_4J^2t} t^{-2}.
\end{align*}
Lemma \ref{s7-L8} is used in the last line, and
(\ref{s7-E-1}) and $j\vee k>J\geq 1$ are used in the next to the last line. The above
bound is at most
$$c_5(t^{-1}+1)^{-c_4\zeta} t^{-2}.$$
Now take $\zeta_0 =2/c_4$ to complete the proof.
\end{proof}

\section{Main estimate}\label{S:secmainest}

We assume now that $a$ satisfying (\ref{posdef}) is also of Toeplitz form.
For a point  $v$ in $\ell^2$ define $v'_k=e^{-\lam_kt}v_k$ and, abusing our earlier notation slightly, define $\pi_{k}=\pi_k:\ell^2\to\ell^2$ by 
\[\pi_{k}(x)=(x_i,\dots,x_k,0,0,\dots).\]
For $1\le i,j\le K$ we let 
\[a_{|i-j|}^K(x)\equiv a_{ij}^K(x)=a_{ij}(\pi_K(x)),\q a_{ij}^K(x,t)=a_{ij}^K(x)\int_0^te^{-(\lam_i+\lam_j)s}ds,\, x\in\ell^2,\]
and let $A^K(x,t)$ be the inverse of $a^K(x,t)$. We will apply the results of Sections 6 and 7 to these $K\times K$ matrices. We will sometimes write
$\ol x_K$ for $(x_1, \ldots, x_K)$, and when convenient will  identify 
$\pi_K(x)$ with $\ol x_K$.  It will be convenient now to work with the notation
\bee\label{NKdefine}
N_K(t,x,y)=Q_K(\pi_K(y-x'),A^K(y,t)),
\eee
so that
\bee\label{projN}
N_K(t,x,y)=N_K(t,\pi_K(x),\pi_K(y)),\ x,y\in\ell^2.
\eee
As before $D_{ij}N_K(t,x,y)$ denotes second order partial derivatives in the $x$ variable.

Our goal in this section
is to prove the following:

\bet\label{L1bound} Assume $(a_{ij}(y))$ satisfies \eqref{Ehyp} and \eqref{gammacond} for all $i,j,k\in\bN$, for some $\alpha\in(\frac{1}{2},1]$, $\beta>\frac{9}{2}-\al$, and $\gamma>\frac{2\alpha}{2\alpha-1}$.  Then there is a $c_{1}>0$ 
and $\eta_{1}=\eta_{1}(\alpha,\gamma)>0$ so that for all $x\in\ell^2$, $K\in\bN$, and $t>0$,
\begin{align}\label{ME1}
\int_{\bR^K} \Big|\sum_{i,j=1}^\infty [&a^K_{ij}(x)-a^K_{ij}(y)]D_{ij}N_K(t,x,y)\Bigr|\,d\ol y_K\\
\nonumber&\leq c_{1}t^{-1+\eta_{1}}(1+\norm{x}_\infty^\alpha).
\end{align}
\eet
\begin{proof} Note first that by \eqref{projN} $D_{ij}N_K=0$ if $i\vee j>K$ and so by the symmetry of $a(x)$ and the Toeplitz form of $a$,  the integral we need to bound is 
\begin{align*}
I\equiv&\int_{\R^K}\Bigl|\sum_{i=1}^K\sum_{j=1}^K(a^K_{ij}(x)-a^K_{ij}(y))D_{ij}N_K(t,x,y)\Bigr|\,d\ol y_K\\
\le &2\int_{\R^K}\Bigl|\sum_{\ell=1}^{K-1}\sum_{j=1}^K(a^K_\ell(x)-a^K_\ell(y))D_{j+\ell,j} N_K(t,x,y)\Bigr|\,d\ol y_K\\
&+\int_{\R^K}|a_0^K(x)-a_0^K(y)|\Bigl|\sum_{j=1}^KD_{jj} N_K(t,x,y)\Bigr|\,d\ol y_K.
\end{align*}
Now let $J=J_\zeta(t)$ where $\zeta$ is as in Proposition~\ref{s7-P9}.  If $j>J$ or $\ell\ge J$ then clearly $i=j+\ell>J$, so that 
\begin{align}
\label{Idec}I \le 2&\int\sum_{\ell=0}^{J-1}|a^K_\ell(x)-a^K_\ell(y)|\Bigl|\sum_{j=1}^J D_{j+\ell,j}N_K(t,x,y)\Bigr|\,d\ol y_K\\
\nonumber &+\sum_{i=1}^K\sum_{j=1}^K1_{(i\vee j\ge J)}\int|(a_{ij}^K(x)-a^K_{ij}(y))D_{ij}N_K(t,x,y)|\,d\ol y_K\\
\nonumber =2&I_1+I_2.
\end{align}
Note that
\bee\label{E8.45}
|a_{ij}^K(z)|=|\langle a^K(z)e_i,e_j\rangle |
\le |\langle a^K(z)e_i,e_i\rangle|^{1/2} 
 |\langle a^K(z)e_j,e_j\rangle|^{1/2}\le \Lambda_1
\eee
by Cauchy-Schwarz.
Then
Proposition~\ref{s7-P9} implies that 
\bee\label{I2bound}
I_2\le 2\Lambda_1c_{2}(t+1)^{-2}.
\eee

Recalling that $x'_k=e^{-\lam_kt}x_k$, we can write
\begin{align}
\nonumber I_1\le&\sum_{\ell=0}^{J-1}\int|a_\ell^K(x')-a_\ell^K(y)|\Bigl|\sum_{j=1}^JD_{j,j+\ell}N_K(t,x,y)\Bigr|\,d\ol y_K\\
&+\sum_{\ell=0}^{J-1}|a_\ell^K(x)-a^K_\ell(x')|\int\Bigl|\sum_{j=1}^J D_{j,j+\ell}N_K(t,x,y)\Bigr|\,d\ol y_K\\
\label{Isum}\equiv&I_{1,1}+I_{1,2}.
\end{align}

Let
\bee\label{dabdefine}
d_{\al,\beta}(x,y)=\sum_{n=1}^K|x_n-y_n|^\alpha n^{-\beta}.
\eee
By \eqref{Ehyp} and \eqref{gammacond},
\bee\label{aincr1}
|a^K_\ell(x')-a^K_\ell(y)|\le c_3\min((1+\ell^\gamma)^{-1},d_{\al,\beta}(x',y)).
\eee
Therefore by \eqref{s7-E-1A}
{\allowdisplaybreaks
\begin{align}\nn
I_{1,1}=&\sum_{\ell=0}^{J-1}\int|a^K_\ell(x')-a^K_\ell(y)|\Bigl|\sum_{j=1}^J\exp(-(\lam_j+\lam_{j+l})t)\\
\label{I11bound}&\qq\qq\times S_{j,j+\ell}(\pi_K(y-x'),A^K(y,t))\Bigr|N_K(t,x,y)\,d\ol y_K\\
\nn
\le&c_4\sum_{\ell=0}^{J-1}\Bigl[\int\Bigl((1+\ell^\gamma)^{-2}\wedge d_{\al,\beta}(x',y)^2\Bigr)N_K(t,x,y)\,d\ol y_K\Bigr]^{1/2}\\
\nn&\qq\qq\times\Bigl[\int\Bigr(\sum_{j=1}^J\exp(-(\lam_j+\lam_{j+\ell})t)S_{j,j+\ell}(\pi_K(y-x'),A^K(y,t))\Bigr)^2\\
\nn&\qq\qq\times N_K(t,x,y)\,d\ol y_K\Bigr]^{1/2}\\
\nn\le&c_5\Bigl(\sum_{\ell=0}^{J-1}\Bigl((1+\ell^\gamma)^{-1}\wedge\Bigl[\int\Bigl(\sum_{n=1}^K|x'_n-y_n|^{2\alpha}n^{-\beta}\Bigr)N_K(t,x,y)\,d\ol y_K\Bigr]^{1/2}\Bigr)\Bigr)\\
\nn&\qq\qq\qq\times \sqrt J t^{-1}.
\end{align}
}
In the last line we used Proposition~\ref{s7-P6} on the second factor and the Cauchy-Schwarz inequality on the sum in the first factor and then Theorem~\ref{GLM-main}(b) to bound the total mass in this factor. Next use Theorem~\ref{s6-T6.6} with $p=\alpha$ to conclude that
\[\int|x_n'-y_n|^{2\alpha}N_K(t,x,y)\,d\ol y_K\le c_{6 }t^\alpha.\]
It now follows from \eqref{I11bound} and the choice of $J$ that $I_{1,1}$ is at most
\begin{align}\label{I11bound1}
\nn c_7&\Bigl\{\sum_{\ell=0}^{J-1}((1+\ell^\gamma)^{-1}\wedge(t^{\alpha/2}))\Bigr\}\Bigl(\log\Bigl(\frac{1}{t}+1\Bigr)\Bigr)^{1/4}t^{-5/4}\\
&\le c_8 \Bigl\{\sum_{\ell=1}^J(\ell^{-\gamma}\wedge t^{\alpha/2})\Bigr\}\Bigl(\log\Bigl(\frac{1}{t}+1\Bigr)\Bigr)^{1/4}t^{-5/4}.
\end{align}
By splitting the above sum up at $\ell=\lfloor t^{-\alpha/2\gamma}\rfloor$ we see that 
\bee\label{ellsum}
\sum_{\ell=1}^J(\ell^{-\gamma}\wedge t^{\alpha/2})\le c_9\Bigl(t^{\alpha/2}\Bigr)^{(\gamma-1)/\gamma}.
\eee
Using this in \eqref{I11bound1}, we may bound $I_{1,1}$ by
\bee\label{I11bound2}
c_{10}\Bigl(\log\Bigl(\frac{1}{t}+1\Bigr)\Bigr)^{1/4}t^{(\alpha(\gamma-1)/2\gamma)-5/4}\le c_{11}t^{-1+\eta},
\eee
for some $\eta=\eta(\alpha,\gamma)>0$ because $\gamma>\frac{2\alpha}{2\alpha-1}$. 

Turning to $I_{1,2}$, note that
\begin{align}\nn d_{\al,\beta}(x',x)&=\sum_{n=1}^K |x_n|^\alpha|1-e^{-\lam_n t}|^\alpha n^{-\beta}\\
\nn&\le \norm{x}_\infty^\alpha t^\alpha\sum_{n=1}^\infty \lam_n^\alpha n^{-\beta}\\
\label{I11bound3}&\le c_{12}\norm{x}_\infty^\alpha t^\alpha,
\end{align}
where \eqref{s7-E-1} and $\beta-2\alpha>1$ are used in the last line.  Therefore \eqref{aincr1} now gives
\bee\label{s8-E15A}
|a_\ell^K(x')-a^K_\ell(x)|\le c_{13}\min((1+\ell^\gamma)^{-1},\norm{x}_\infty^\alpha t^\alpha).
\eee
As in \eqref{I11bound} we now get (again using Proposition~\ref{s7-P6})
\begin{align}\label{I12bound}
\nn I_{1,2}&\le\sum_{\ell=0}^{J-1}c_{13}\min((1+\ell^\gamma)^{-1},\norm{x}_\infty^\alpha t^\alpha)\\
\nn&\ \ \times\Bigl[\int\Bigl(\sum_{j=1}^J e^{-(\lam_j+\lam_{j+\ell})t} S_{j,j+\ell}(\pi_K(y-x'),A^K(y,t))\Bigr)^2N_K(t,x,y)\,d\ol y_K\Bigr]^{1/2}\\
&\le c_{14}\sqrt Jt^{-1}\sum_{\ell=0}^{J-1}\min((1+\ell^\gamma)^{-1},\norm{x}_\infty^\alpha t^\alpha).
\end{align}
Now use \eqref{ellsum} with $\norm{x}_\infty^\alpha t^\alpha$ in place of $t^{\alpha/2}$ to conclude that
\begin{align}\label{I12Abound}
I_{1,2}\le c_{15}&\Bigl(\log\Bigl(\frac{1}{t}+1\Bigr)\Bigr)^{1/4}t^{-5/4}(\norm{x}_\infty^\alpha t^\alpha)^{(\gamma-1)\gamma}\le c_{16}(\norm{x}_\infty^\alpha+1)t^{-1+\eta}
\end{align}
for some $\eta=\eta(\alpha,\gamma)>0$ because $\gamma>\frac{2\alpha}{2\alpha-1}>\frac{4\alpha}{4\alpha-1}$.  

Finally use the above  bound on $I_{1,2}$ and the bound on $I_{1,1}$ in \eqref{I11bound2} to bound $I_1$ by the right-hand side of \eqref{ME1}.  Combining this with the bound on $I_2$ in \eqref{I2bound} completes the proof.\end{proof}

For $R>0$ let $p_R:\bR\to\bR$ be given by $p_R(x)=(x\wedge R)\vee(-R)$ and define a truncation operator $\tau_R:\ell^2\to \ell^2$ by $(\tau_Rx)_n=p_R(x_n)$.  
Define $a^R$ by 
\bee \label{aRdef}a^R(x)=a(\tau_Rx).\eee
Clearly $a^R(x)=a(x)$ whenever $\norm{x}_\infty\equiv\sup_n|x_n|\le R$.
We write $a^{K,R}$ for the $K\times K$ matrix $(a^R)^K$.

\bel\label{pRbnd} For any $\lam\ge 0$ and $t,R>0$, $\sup_{x\in\bR}|p_R(x)-p_R(xe^{-\lam t})|\le R\lam t$. 
\eel
\begin{proof} Assume without loss of generality that $x>0$ and set $x'=e^{-\lam t}x$.  If $x'\ge R$, $p_R(x)=p_R(x')=R$, and if $x\le R$, then 
\[|p_R(x)-p_R(x')|=|x-x'|=(1-e^{-\lam t})x\le \lam tR.\] 
Finally if $x'<R<x$, then 
\[|p_R(x)-p_R(x')|=R-x'=R-e^{-\lam t}x\le R(1-e^{-\lam t})\le \lam tR.\]
\end{proof}

\bel\label{aRprop} If $a$ satisfies \eqref{posdef}, \eqref{Ehyp}, and \eqref{gammacond} and is of Toeplitz form,  then for any $R>0$, $a^R$ satisfies the same conditions with the same constants.
\eel
\begin{proof} This is elementary and so we only consider \eqref{Ehyp}.   For this note that
\begin{align*}
|a_{ij}^R(y+he_k)-a_{ij}^R(y)|&\le \kappa_\beta|p_R(x_k+h)-p_R(x_k)|^\al k^{-\beta}\\
&\le \kappa_\beta|h|^\al k^{-\beta},
\end{align*}
as required. \end{proof}

\bec\label{L1Rbound} Assume the hypotheses of Theorem~\ref{L1bound}.  Then for all $x\in\ell^2$, $K\in\bN$ and $R,t>0$, 
\begin{align}\label{MER1}
\int_{\bR^K} \Big|\sum_{i,j=1}^\infty [&a^{K,R}_{ij}(x)-a^{K,R}_{ij}(y)]D_{ij}N_K(t,x,y)\Bigr|\,d\ol y_K\\
\nonumber&\leq c_{1}t^{-1+\eta_{1}}(1+R^\alpha).
\end{align}
\eec
\begin{proof} We use the notation in the proof of Theorem~\ref{L1bound}.  By Lemma~\ref{aRprop} and the proof of Theorem~\ref{L1bound} it suffices to show 
that we have 
\bee\label{I12R}
I_{1,2}\le c_{2}(R^\alpha+1)t^{-1+\eta}
\eee
instead of \eqref{I12Abound}.          
We have by Lemma~\ref{pRbnd} 
\begin{align}\nn
d_{\al,\beta}(\tau_Rx',\tau_Rx)&=\sum_{n=1}^K|p_R(x_n)-p_R(e^{-\lam_n t}x_n)|^\alpha n^{-\beta}\\
\nn&\le \sum_{n=1}^K(R\lam_n t)^\al n^{-\beta}\\
\nn&\le (Rt)^\al c_3\sum_{n=1}^Kn^{2\al-\beta}\\
\label{DpRbound}&\le c_4R^\al t^\al.
\end{align}
The fact that $\beta-2\al>1$ is used in the last line. Now use \eqref{DpRbound} in place of \eqref{I11bound3} and argue exactly as in the proof of \eqref{I12Abound} to derive \eqref{I12R} and so complete the proof.
\end{proof}

\section{Uniqueness}\label{S:Uniq}

In this section we prove Theorem \ref{mainSDE}.
Recall the definitions of $\sT^2_k$ and $\sT^{2,C}_k$ 
and the definition of the martingale problem for the operator $\sL$
from Section \ref{S:RO}.
Throughout this section we assume the hypotheses of Theorem \ref{mainSDE}
are in force.

\bel\label{acont} There exists $c_{1}$ so that for all $x,y\in \ell^2$,
\[\norm{a(x)-a(y)}\le \norm{a(x)-a(y)}_s\le c_{1}|x-y|^{\al/2}.\]
\eel
\begin{proof} We need only consider the second inequality by \eqref{Schur}.  Our hypotheses \eqref{Ehyp} and \eqref{gammacond} imply 
\begin{align*}|a_{ij}(x)-a_{ij}(y)|&\le \min\Bigl(\frac{2\kappa_\gamma}{1+|i-j|^\gamma},\kappa_\beta\sum_k|x_k-y_k|^\al k^{-\beta}\Bigr)\\
&\le c_2(1+|i-j|^{-\gamma/2})\Bigl(\sum_k|x_k-y_k|^\al k^{-\beta}\Bigr)^{1/2}.
\end{align*}
The second inequality follows from $\min(r,s)\le r^{1/2}s^{1/2}$ if $r,s\ge 0$.
We have $\gamma>2$ and $2\beta>2-\al$ by \eqref{Ehyp}, and so
\begin{align*}
\sum_j|a_{ij}(x)-a_{ij}(y)|&\le c_3\Bigl(\sum_k |x_k-y_k|^\al k^{-\beta}\Bigr)^{1/2}\\
&\le c_3\Bigl(\sum_k|x_k-y_k|^2\Bigr)^{\al/4}\Bigl(\sum_k k^{-2\beta/(2-\al)}\Bigr)^{(2-\al)/4}\\
&\le c_4\norm{x-y}^{\al/2}.
\end{align*}
\end{proof}

\bep\label{exists} For each $v\in \ell^2$ there is a solution to the martingale problem for $\sL$ starting at $v$. 
\eep
\begin{proof} This is well known and follows, for example from the continuity of $a$ 
given by Lemma~\ref{acont} and
 Theorem 4.2 of \cite{ABGP}.\end{proof}
 
We turn to uniqueness.
Let $\sL^R(x)$ be defined in terms of $a^R$ analogously to how
$\sL$ is defined in terms of $a$.

\bel\label{localuni} For any $R>0$ and $v\in \ell^2$ there is a unique solution to the martingale problem for $\sL_R$ starting at $v$.  \eel
\begin{proof} By Lemma \ref{aRprop} and Proposition~\ref{exists} we only need show uniqueness.  

We fix $R>0$ and for $K\in \bN$ define
$$\sM_K^x f(z)=\sum_{i,j\leq K} a^R_{ij}(x) D_{ij}f(z)-\sum_{ j\leq K}
 \lam_j z_j D_j f (z).$$
Note that  if $f\in \sT_k^2$ and $K\ge k$, then
\bee\label {MLid}  \sL_R f(x)=\sM_K^xf(x).\eee
Let  
$$\gamma_K(dy)= m(dy_1) \cdots m(dy_K)\delta_{0}(dy_{K+1})\delta_{0}(dy_{K+2})
\cdots,$$
where $m$ is Lebesgue measure on $\R$ and $\delta_z$ is point mass at
z.  Define 
\bee\label{E9.05}
\norm{f}_{C_0}=\sup_z |f(z)|.
\eee 

Suppose $\P_1, \P_2$ are two solutions to the martingale problem for
$\sL_R$ started at some fixed point $v$.  For $\theta>0$ and $f$ bounded and
measurable on $\ell^2$, let
$$S_\theta^i f= \E_i \int_0^\infty e^{-\theta t} f(X_t)\, dt, \qq i=1,2,$$
and $S_\Delta f=S^1_\theta f-S^2_\theta f$. 
Set
$$\Gamma=\sup_{\norm{f}_{C_0}\leq 1} |S_\Delta f|.$$
Note \bee\label{s9-E9.2A}
\Gamma<\infty
\eee
by the definition of $S_\theta^if$.

If $f\in \sT^2$, we have
$$ f(X_t)-f(X_0) =M^f(t)+ \int_0^t \sL_R f(X_s)\, ds$$
where $M^f$ is a martingale under each $\P_i$. 
Taking expectations, multiplying both sides by $\theta e^{-\theta t}$, and
integrating over $t$ from
0 to $\infty$, we see that
$$f(v)=S_\theta^i (\theta f -\sL_R f).$$
Now take differences in the above to get
\bee\label{UE1}
S_\Delta(\theta f-\sL_R f)=0.
\eee

Next let $g\in \sT^{2,C}_k$ and for $K\ge k$ set
$$f_{\eps K}(x)=\int e^{\theta \eps}\int_\eps^\infty e^{-\theta t}
N_K(t,x,y) g(y)\, dt\, \gamma_K(dy).$$
Recall that $N_K$ is defined in \eqref{NKdefine}.
Since $N_K(t,x,y)$ is smooth in $x$, bounded uniformly for $t\ge \vep$ and $N_K(t,x,y)$ depends on $x$ only
through $\pi_K(x)$, we see that $f_{\eps K}\in \sT^2_K$.

If we write 
\bee\label{s9-E9.35}
W_{\eps K}(x,y)=e^{\theta \eps}\int_\eps^\infty
e^{-\theta t} N_K(t,x,y)\, dt,
\eee
then $$f_{\eps K}(x)=\int W_{\eps K}(x,y) g(y)\, \gamma_K(dy).$$
Holding $y$ fixed and viewing $N_K(t,x,y)$ and $W_{\eps K}(x,y)$ as
functions of $x$, we see by Kolmogorov's backward equation for the Ornstein-Uhlenbeck process with diffusion matrix $(a_{ij}(y))_{i,j\le K}$ that 
$$\sM_K^{\pi_K(y)} N_K(t,x,y)=\frac{\del}{\del t} N_K(t,x,y).$$
Alternatively, one can explicitly calculate the derivatives.
Using dominated converge to differentiate under the integral in \eqref{s9-E9.35} gives 
\bee\label{UE2}
(\theta-\sM_K^{\pi_K(y)}) W_{\eps K}(x,y)=N_K(\eps,x,y).
\eee

By \eqref{MLid} for all $x$ and $K\ge k$
\begin{align}\label{Idecomp}
(\theta-\sL_R) f_{\eps K}(x)&=(\theta-\sM_K^x) f_{\eps K}(x)\\
\nn&=\int (\theta-\sM_K^{\pi_K(y)})W_{\eps K}(x,y)  g(y)\, \gamma_K(dy)\\
\nn&\qq -\int (\sM_K^{\pi_K(x)}-\sM_K^{\pi_K(y)})W_{\eps K}(x,y)  g(y)\, \gamma_K(dy)\\
\nn&\qq -\int (\sM_K^x-\sM_K^{\pi_K(x)})W_{\eps K}(x,y)  g(y)\, \gamma_K(dy)\\
\nn&=g(x)+\Big[\int N_K(\eps,x,y) g(y)\, \gamma_K(dy)-g(x)\Big] \\
\nn&\qq -\int (\sM_K^{\pi_K(x)}-\sM_K^{\pi_K(y)})W_{\eps K}(x,y)  g(y)\, \gamma_K(dy)\\
\nn&\qq -\int (\sM_K^x-\sM_K^{\pi_K(x)})W_{\eps K}(x,y)  g(y)\, \gamma_K(dy)\\
\nn&=g(x)+I_1(\eps,K,x)+I_2(\eps,K,x)+I_3(\eps,K,x).
\end{align}
We used \eqref{UE2} in the third equality.

For $x\in \ell^2$ fixed we first claim that
\bee\label{I1lim} I_1(\eps,K,x)\to 0 
\eee
 boundedly 
 and uniformly in $K\ge k$ as $\eps\to 0$.
By virtue of Proposition \ref{GLM2},
it suffices to show
$$\int N_K(\eps,x,y)[g(y)-g(x)]\, \gamma_K(dy)\to 0$$
boundedly and pointwise as $\eps\to 0$, uniformly in $K\ge k$. The boundedness
is immediate from Theorem~\ref{GLM-main}. 
Since $g\in \sT^2_k$, given $\eta$ there exists $\delta$ such that
$|g(y)-g(x)|\leq \eta$ if $|\pi_k(y-x)|\leq \delta$, and using Theorem 
\ref{GLM-main}, it suffices to show
$$\int_{\{y: \sum_{i=1}^k |y_i-x_i|^2\geq \delta^2\}}N_K(\vep,x,y)\,\gamma_K(dy)\to 0$$
pointwise as $\eps\to 0$, uniformly in $K$.
Since $e^{-\lam_i \eps}x_i\to x_i$ for $i\le k$ as $\eps\to 0$, it suffices to
show (recall $\ol y_K=(y_1,\dots,y_K)$)
\bee\label{E25B1}
\int_{\{y: \sum_{i=1}^k |y_i-x'_i|^2\geq \delta^2/2\}} N_K(\eps,x,y)\, d\ol y_K\to 0
\eee
 as $\vep\to0$ uniformly in $K\ge k$.
By Theorem~\ref{s6-T6.6} the above integral is at most
\[\int 
\sum_{i=1}^k \frac{|y-x'_i|^2}{\delta^2/2}N_K(\vep,x,y)\,d\ol y_K\le \frac{c_{1 }k\vep}{\delta^2/2}\]
and \eqref{I1lim} is established. 

Next we claim that 
 for each $\vep>0$  
\bee\label{I3lim} 
\lim_{K\to\infty}\sup_x|I_3(\vep,K,x)|=0.
\eee
Since $t\ge \vep$ in the integral defining $W_{\vep K}(x,y)$ we can 
use dominated convergence to differentiate through the integral and conclude that
\begin{align}\label{I3bnd1}
|I_3&(\vep,K,x)|\\
\nn&\le \int_\vep^\infty e^{-\theta(t-\vep)}\sum_{i,j\le K}|a^R_{ij}(x)-a^R_{ij}(\pi_K(x))|\,\norm{g}_{C_0}\\
\nn&\qq\qq\times\int_{\bR^K}e^{-(\lam_i+\lam_j)t}|S_{i,j}(w,A^K(y,t))|Q(w,A^K(y,t))dw\,dt.
\end{align}
As in the proof of Proposition~\ref{s7-P6}, the substitution $w'=G(t)^{1/2}w$ shows that the integral over $\bR^K$ in \eqref{I3bnd1} equals
\bee\label{I3bnd2} \int_{\R^K} H_i(t)H_j(t)\,|S_{i,j}(w',\wt A^K(y,t)|Q_K(w'\wt A^K(y,t))\,dw'\le c_2t^{-1},
\eee
where \eqref{s7-E12} and Lemma~\ref{s7-L7} are used in the above.  By \eqref{Ehyp} we have 
\begin{align} \label{I3bnd3}\sum_{i,j\le K}|a^R_{ij}(x)-a^R_{ij}(\pi_K(x))|&\le \sum_{i,j\le K}\sum_{\ell>K}\kappa_\beta|p_R(x_\ell)|^\al\ell^{-\beta}\\
\nn&\le \kappa_\beta K^2R^{\al}\sum_{\ell>K}\ell^{-\beta}\\
\nn&\le c_3R^\al K^{3-\beta}.
\end{align}
Use \eqref{I3bnd2} and \eqref{I3bnd3} in \eqref{I3bnd1} to get 
\begin{align*}
|I_3(\vep,K,x)|&\le \int_\vep^\infty e^{-\theta(t-\vep)}c_4t^{-1}R^\al K^{3-\beta}\,dt\\
&\le c_4\theta^{-1}\vep^{-1}R^\al K^{3-\beta},
\end{align*}
which proves \eqref{I3lim} by our hypothesis on $\beta$. 

Finally for $I_2$, we use Corollary~\ref{L1Rbound} and multiply both sides of \eqref{MER1} by $e^{-\theta (t-\vep)}$, and then integrate over $t$ from $\vep$
to $\infty$ to obtain by Fubini
\begin{align}\label{UE1.5}
|&I_2(\vep,K,x)|\\
\nn&= \Big|\int_{y\in \R^K} (\sM_K^{\pi_K(x)}-\sM_K^{\pi_K(y)})\Bigl[e^{\theta \eps}
\int_\eps^\infty e^{-\theta t}N_K(t,\cdot,y)g(y)\, dt\, \Bigr](x)\gamma_K(dy)\Big|\\
\nn&\leq \tfrac14 \norm{g}_{C_0},
\end{align}
for all $\eps\in (0,1)$ and $K\ge k$, provided we choose $\theta>\theta_0\ge 1$,
where $\theta_0$ depends on $R$ and the $c_1$ and $\eta_1$ of Theorem \ref{L1bound}.
 This implies that for $\theta>\theta_0$,
\bee\label{I2lim}
\sup_{\vep\in(0,1),K\ge k}|S_\Delta(I_2(\vep,K,\cdot))|\le \tfrac12\Gamma\norm{g}_{C_0}.
\eee

Using \eqref{UE1} and \eqref{Idecomp} for $K\ge k$, we have
$$|S_\Delta g|\leq 
|S_\Delta(I_1(\eps,K,\cdot))|
+|S_\Delta(I_2(\eps,K,\cdot))|
+|S_\Delta(I_3(\eps,K,\cdot))| .$$
Now let $K\to \infty$ and use \eqref{I3lim} and \eqref{I2lim} to conclude that
\begin{align*}
|S_\Delta g|&\leq \limsup_{K\to \infty} |S_\Delta (I_1(\eps,K,\cdot))|
+\limsup_{K\to \infty} |S_\Delta (I_2(\eps,k,\cdot))|\\
&\leq \limsup_{K\to \infty} |S_\Delta(I_1(\eps,K,\cdot))|+
\tfrac12 \Gamma \norm{g}_{C_0}.
\end{align*}
Then letting $\eps\to 0$ and using \eqref{I1lim}, we obtain
$$|S_\Delta g|
\leq \tfrac12 \Gamma \norm{g}_{C_0},$$
provided $g\in \sT_k^{2,C}$.
By a monotone class argument and the fact that $S_\Delta$ is
the difference of two  finite measures, we have the above inequality for
$g\in \sT$. 
The $\sigma$-field we are using is generated by the
cylindrical sets, so another application of the monotone class theorem leads to
$$|S_\Delta g|\leq \tfrac12 \Gamma \norm{g}_{C_0}$$
for all bounded $g$ which are measurable with respect to $ \sigma(\cup_j \sT_j)$. Taking the supremum over
all such $g$ bounded by 1, we obtain
$$\Gamma\leq \tfrac12 \Gamma.$$
Since $\Gamma<\infty$ by \eqref{s9-E9.2A}, then $\Gamma=0$ for every $\theta>\theta_0$. 

This proves that $S_\theta^1f=S_\theta^2 f$ for every bounded and
continuous $f$. By the uniqueness of the Laplace transform, this
shows that the one-dimensional distributions of $X_t$ are the same
under $\P_1$ and $\P_2$.
We now proceed as in 
\cite[Chapter 6]{SV} or \cite[Chapter 5]{Bas97} to obtain uniqueness
of the martingale problem for $\sL_R$. 
\end{proof}

We now complete the proof of the main result for infinite-dimensional stochastic differential
equations from the introduction.

\begin{proof}[Proof of Theorem \ref{mainSDE}] 
We have existence holding by Proposition~\ref{exists}.  Uniqueness follows from 
 Lemma~\ref{localuni} by a standard localization argument; see
\cite[Section 6]{BP-inf}.
\end{proof}

To derive Corollary \ref{weakuniqueness} from Theorem \ref{mainSDE}
is completely standard and is left to the reader.

\section{SPDEs}\label{sec:spde}

Before proving our uniqueness result for our SPDE, we first need
need a variant of Theorem \ref{mainSDE} for our application to 
SPDEs. Let $\lam_0=0$ and now let
$$ \sL' f(x)=\sum_{i,j=0}^\infty a_{ij}(x)D_{ij}f(x)-\sum_
{i=0}^\infty \lam_ix_iD_if(x)
$$
for $f\in \sT$.
In this case $\ell^2=\ell^2(\Z_+)$. 

\bet\label{SDEth2} 
 Suppose  $\al, \beta, \gamma $,  and the $\lam_i$  are as in Theorem 
\ref{mainSDE} and in addition
$\beta>\gamma/(\gamma-2)$. Suppose $a$ satisfies \eqref{posdef} and
$a$ can be written as $a_{ij}=a^{(1)}_{ij}+a^{(2)}_{ij}$, where the $a^{(1)}_{ij}$ satisfy 
\eqref{gammacond} and \eqref{Ehyp} and   is of Toeplitz form,
and  $a^{(2)}_{ij}$ satisfies  \eqref{Ehyp} 
 and there exists a constant $\kappa'_\gamma$ such that
\bee\label{gammaprimecond}
|a^{(2)}_{ij}(x)|\leq \frac{\kappa'_\gamma}{1+(i+j)^\gamma}
\eee for all $x\in \ell^2$ and $i,j\geq 0$. Then if
$v\in \ell^2$, there exists a solution
to the martingale problem for $ \sL'$ starting at $v$ and the solution
is unique in law.
\eet

\begin{proof}
First, all the arguments of the previous sections are still valid
when we let our indices run over $\{0,1,2,\ldots\}$ instead of
$\{1,2,\ldots\}$  provided 
\begin{description}
\item{(1)} we replace expressions
like $2\lam_i/(1-e^{-2\lam_it})$ by $1/t$ when $\lam_i=0$, which happens
only when $i=0$, and
\item{(2)} we replace expressions like $n^{-\beta}$ by $(1+n)^{-\beta}$.
\end{description}

Existence follows from Theorem 4.2 of \cite{ABGP} as in the proof of Theorem~\ref{mainSDE}.

Define $N_K$ in terms of $a$ and its inverse $A$ as in \eqref{NKdefine}.
We prove the following 
analog of \eqref{MER1} exactly
as in the proof of Corollary~\ref{L1Rbound} and Theorem~\ref{L1bound}:
\begin{align}
\int_{\bR^K} \Big|\sum_{i,j=1}^\infty [&(a^{(1)})^{K,R}_{ij}(x)-(a^{(1)})^{K,R}_{ij}(y)]D_{ij}N_K(t,x,y)\Bigr|\,dy\label{s10-E1}\\
&\leq c_{1}t^{-1+\eta_{1}}(1+R^\alpha).\nn
\end{align}
Here note that the proof uses the bounds on $N_K$ and $D_{ij}N_K$ from Sections~\ref{S:GL} and \ref{S:secder} and the regularity properties of $a^{(1)}$ (which are the same as those of $a$ in the proof of Theorem~\ref{mainSDE}) separately.  
 If we prove the analog
of \eqref{s10-E1} with $a^{(1)}$ replaced by $a^{(2)}$, we can then proceed exactly
as in Section  \ref{S:Uniq} to obtain our theorem.  That is, it
suffices to fix $K$ and $R$ and to show that for some $c_1,\eta_1>0$, 
\bee\label{a2goal}\int_{\bR^K} \Big|\sum_{i,j=1}^J [(a^{(2)})^{K,R}_{ij}(x)-(a^{(2)})^{K,R}_{ij}(y)]D_{ij}N_K(t,x,y)\Bigr|\,dy\le c_1 t^{-1+\eta_1}.\eee

Very similarly to the derivation of \eqref{s8-E15A} (see also that of \eqref{DpRbound}),
we have
$$|(a^{(2)})^{K,R}_{ij}(x)-(a^{(2)})^{K,R}_{ij}(x')|\leq c_1\min ((1+i+j)^{-\gamma},
R^\al t^\al).$$
Since $\al\in (1/2,1]$ and $\gamma>2\al/(2\al-1)$, then $\gamma>2$. We can
choose $\eta_2\in (0,1)$ such that $\gamma(1-\eta_2)>2$, and then 
$$|(a^{(2)})^{K,R}_{ij}(x)-(a^{(2)})^{K,R}_{ij}(x')|\leq c_1 (1+i+j)^{-\gamma(1-\eta_2)}
R^{\al \eta_2} t^{\al \eta_2}.$$
Using this and Proposition \ref{s7-P99} with $p=0$  and observing that
$(a^{(2)})^{K,R}$ satisfies all the hypotheses in Section~\ref{S:secder}, we conclude that
\begin{align}
\sum_{i,j=0}^J \int& |(a^{(2)})^{K,R}_{ij}(x)-(a^{(2)})^{K,R}_{ij}(x')|\,
|D_{ij}N_K(t,x,y)|\, dy\label{s10-S1}\\
&\leq  c_2\sum_{i,j=0}^J (1+i+j)^{-\gamma(1-\eta_2)} t^{\al \eta_2-1}.\nn
\end{align}

The condition $\beta>\gamma/(\gamma-2)$, allows us to find $\eta_3$ such that
$\gamma(1-\eta_3)>2$ and $\beta\eta_3>1$.
Fix $i$ and $j$ for the moment and let $d_{\al,\beta}(x,y)$ be defined as in 
\eqref{dabdefine}.
We write
\begin{align*}
\int d_{\al,\beta}(x',y)^{\eta_3} &|D_{ij}N_K(t,x,y)|\, dy\\&
\le \int \sum_{n=0}^\infty |x'_n-y_n|^{\al \eta_3} (n+1)^{-\beta \eta_3}
|D_{ij}N_K(t,x,y)|\, dy\\
&\le \sum_{n=0}^\infty (n+1)^{-\beta \eta_3} t^{\al \eta_3/2-1}\\
&\leq c_3t^{\al \eta_3/2-1},
\end{align*}
using Proposition \ref{s7-P99}.
Since
$$|(a^{(2)})^{K,R}_{ij}(x')-(a^{(2)})^{K,R}_{ij}(y)|
\le c_4\min((1+i+j)^{-\gamma}, d_{\al,\beta}(x',y)),$$
then
$$|(a^{(2)})^{K,R}_{ij}(x')-(a^{(2)})^{K,R}_{ij}(y)|\leq
c_4(1+i+j)^{-\gamma(1-\eta_3)}d_{\al,\beta}(x',y)^{\eta_3}.$$
Consequently
\begin{align}
\sum_{i,j=0}^J \int& |(a^{(2)})^{K,R}_{ij}(x')-(a^{(2)})^{K,R}_{ij}(y)|\,
|D_{ij}N_K(t,x,y)|\, dy\label{s10-S2}\\
&\leq  c_5\sum_{i,j=0}^J (1+i+j)^{-\gamma(1-\eta_3)}
\sup_{i,j} \int d_{\al,\beta}(x',y)^{\eta_3} |D_{ij}N_K(t,x,y)|\, dy\nn\\
&\le c_5t^{\al \eta_3/2-1}.\nn
\end{align}
Combining with \eqref{s10-S1} gives \eqref{a2goal}, as required.
\end{proof}

Before proving Theorem \ref{mainSPDE}, we need the following lemma.
Recall that  $e_n(x)=\sqrt 2 \cos n\pi x$ for $n\geq 1$ and $e_0\equiv 1$.

\bel\label{Fourierdecay}
Suppose $f\in C_{per}^\zeta$ and $\norm{f}_{C^\zeta}\leq 1$. There exists
a constant $c_1$ depending only on $\zeta$ such that
$$|\angel{f, e_n}|\leq \frac{c_1}{1+n^\zeta}\qq\hbox{for all }n\in\Z_+.$$
\eel

\begin{proof} Let $T$ be the circle of circumference 2 obtained by
identifying $\pm1$ in $[-1,1]$.  Since we can extend the domain of $f$ to $T$ so that $f$ is 
$C^\zeta$ on $T$ and $\cos y= \tfrac12 (e^{iy}+e^{-iy})$, it suffices
to show that the Fourier coefficients of a $C^\zeta$ function on $T$
decay at the rate $|n|^{-\zeta}$. If $\zeta=k+\delta$ for $k\in\Z_+$ and $\delta\in[0,1)$, \cite[II.2.5]{Zygmund}
says that the $n^{th}$ Fourier coefficients of $f^{(k)}$ is $c_2|n|^{k}$ times the
$n^{th}$ Fourier coefficient of $f$. Writing $\wh g$ for the Fourier coefficients of $g$, we then have $|\wh f(n)|\leq c_3 |n|^{-k} |\wh f^{(k)}(n)|$.
By \cite[II.4.1]{Zygmund}, 
$$|\wh f^{(k)}(n)|\leq c_4 |n|^{-\delta}.$$
Combining proves the lemma.
\end{proof}

We now prove Theorem \ref{mainSPDE}.

\begin{proof}[Proof of Theorem \ref{mainSPDE}] 
Our first job will be to use the given $A:C[0,1]\to C[0,1]$ to build a corresponding mapping
$a:\ell^2\to \sL_+(\ell^2,\ell^2)$, where $\sL_+(\ell^2,\ell^2)$ is the space of self-adjoint bounded positive definite mappings on $\ell^2$, so that $a$ satisfies the hypotheses of Theorem~\ref{SDEth2}. 

We first argue that $A$ has a unique continuous extension to a map $\ol A:L^2[0,1]\to L^2[0,1]$.  Let $\sS$ be the space of finite linear combinations of the $\{e_k\}$.  If $u=\sum_{i=0}^N x_ie_i$, $v=\sum_{i=0}^Ny_ie_i\in \sS$, then by \eqref{Fholder} and H\"older's inequality we have
\begin{align*}
\norm{A(u)-A(v)}_2&\le \kappa_1\sum_{i=0}^N|x_i-y_i|^\alpha(i+1)^{-\beta}\\
&\le \kappa_1\norm{u-v}_2^\alpha\left(\sum_{i=0}^N(i+1)^{-2\beta/(2-\alpha)}\right)^{(2-\alpha)/2}\\
&\le c_1\norm{u-v}_2^\alpha,
\end{align*}
 because $\beta>\frac{9}{2}-\alpha>(2-\alpha)/2$.  
Using \eqref{L2cont}, we have
\bee\label{AL2bdd}
\norm{A(u)-A(v)}_2\leq c_1\norm{u-v}_2^\al
\eee
for $u,v\in C[0,1]$. Therefore $A$, whose domain is $C[0,1]$, is a bounded
operator with respect to the $L^2$ norm. Thus there is a unique extension
of $A$ to all of $L^2$. By continuity, it
is clear that the extension satisfies \eqref{Fholder}, \eqref{Abnd} (for almost every $x$ with respect to Lebesgue measure),
and \eqref{Fdecay}.

If $x=\{x_j\}\in\ell^2$, let $u(x)=\sum_{j=0}^\infty x_je_j\in L^2$ and define a symmetric operator on $\ell^2$ by
\[a_{jk}(x)=\int_0^1A(u(x))(y)^2e_j(y)e_k(y)\,dy.\]
If $z\in \ell^2$, then
{\allowdisplaybreaks
\begin{align*}
\sum_{i,j} z_i a_{ij}(x)z_j&=\sum_{i,j}\int_0^1 z_i e_i(y) A(u)(y)^2 z_j e_j(y)\, dy\\
&=\int_0^1 \Big( \sum_i z_i e_i(y)\Big)^2 A(u)(y)^2\, dy\\
&\geq\kappa_2^2 \int_0^1 \Big( \sum_i z_i e_i(y)\Big)^2\, dy\\
&=\kappa_2^2\sum_{i=0}^\infty z_i^2,
\end{align*}}
using the lower bound in \eqref{Abnd} and the fact that the $e_i$ are an orthonormal basis. The upper
bound is done in the very same fashion, and thus \eqref{posdef} holds.

Using the identity 
$$\cos A\cos B=\tfrac12[\cos(A-B)+\cos(A+B)],$$
we see that if $i,j\ge 1$, 
\begin{align*}
a_{ij}(x)&=\int_0^1 A(u)(y)^2 e_i(y)e_j(y)\, dy=2\int_0^1 A(u)(y)^2 \cos(i \pi y)
\cos (j\pi y)\, dy\\
&=\int_0^1 A(u)(y)^2 \cos( (i-j)\pi y)\, dy+\int_0^1 A(u)(y)\cos((i+j)\pi y)\, dy\\
&=a^{(1)}_{ij}(x)+a^{(2)}_{ij}(x).
\end{align*}
If $i$ or $j$ is $0$, there is a trivial adjustment of a multiplicative constant. 
Note both $a^{(1)}$ and $a^{(2)}$ are symmetric because cosine is an even function,
and that $a^{(1)}$ is of Toeplitz form. Also \eqref{Fdecay} now shows that $a^{(1)}$ satisfies \eqref{gammacond} and $a^{(2)}$ satisfies \eqref{gammaprimecond}.

Finally we check \eqref{Ehyp}. We have
\begin{align*}
 |a^{(1)}_{ij}(x+he_k)-a^{(1)}_{ij}(x)|^2&\le
 |\angel{A(u+he_k)^2-A(u)^2,e_{i-j}}|^2\\
&\le \norm{A(u+he_k)^2-A(u)^2}_2^2\\
&\le 4\kappa_2^{-2}\norm{A(u+he_k)-A(u)}_2^2\\
&\le 4\kappa_2^{-2}\kappa_1^2|h|^{2\alpha}(k+1)^{-2\beta}
\end{align*}
by \eqref{Abnd} and \eqref{Fholder}.
This establishes \eqref{Ehyp} for $a^{(1)}$ and virtually the same argument gives it for $a^{(2)}$.  Hence $a$ satisfies the hypotheses of Theorem~\ref{SDEth2}.

Turning next to uniqueness in law, let $u$ satisfy \eqref{SPDEeq} with $u_0\in C[0,1]$ and define $X_n(t)=\angel{u(\cdot,t), e_n}$.  The continuity of $t\to u(t,\cdot)$ in $C[0,1]$ shows that $t\to X_t\equiv\{X_n(t)\}$ is a continuous $\ell^2$-valued process. 
Applying \eqref{s10-E11} with $\vp=e_k$, we see that
$$X_k(t)=X_k(0)-\int_0^t \frac{k^2\pi^2}{2} X_k(s)\, ds+M_k(t),$$
where $M_k(t)$ is a martingale such that
\bee\label{sdeder}
\angel{M_j,M_k}_t=\int_0^t \angel{A(u_s)e_j,A(u_s)e_k}\,ds
=\int_0^t a_{jk}(X(s))\, ds.\eee
Thus we see that $\{X_k\}$ satisfies \eqref{system} with $\lambda_i=i^2\pi^2/2$.

 Since $u_t$ is the $L^2$
limit of the sums $\sum_{k=0}^n X_k(t) e_k(x)$ and $u_t$ is continuous in $x$,
then $u_t$ is easily seen to be a Borel measurable function of  $X(t)$. Thus to prove uniqueness
in law of $u$, it suffices to prove uniqueness in law of $X$.
It is routine to show  the equivalence of uniqueness in 
law of \eqref{system} to 
uniqueness of the martingale problem for $\sL$.
Since the $a_{ij}$ satisfy the hypotheses of Theorem \ref{SDEth2}, we have
uniqueness of the martingale problem for $\sL$.

Finally, the proof of Theorem \ref{mainSPDE} will be complete
once we establish the  existence of solutions to \eqref{SPDEeq}.
The proof of this is standard, but we include it in Appendix \ref{A-existence}
 for the sake of  completeness.
\end{proof}

\begin{proposition}\label{prop:Wasscond} Let $\alpha,\beta,\gamma>0$.\hfil \break
\noindent(a) If 
\begin{equation}\label{Asmooth} A:C[0,1]\to C^\gamma_{per}\hbox{ and }\sup_{u\in C[0,1]}\norm{A(u)}_{C^\gamma}\le \kappa'_3,
\end{equation}
then \eqref{Fdecay} holds for some $\kappa_3$ depending on $\kappa'_3$ and $\gamma$.

\noindent (b) If 
\bee\label{Wasserstein}
\norm{A(u)-A(v)}_2\leq \kappa'_1 \sup_{\vp\in C^{\beta/\alpha}_{per}, \norm{\vp}_{C^{\beta/\alpha}}\leq 1}
|\angel{u-v,\vp}|^\al
\eee for all $u,v$ continuous on $[0,1]$, then \eqref{L2cont} holds and \eqref{Fholder} holds for some $\kappa_1$, depending on $\kappa'_1$, $\alpha$ and $\beta$.
\end{proposition}

\begin{proof} (a) It follows easily from Leibniz's formula that
\[\norm{A^2(u)}_{C^\gamma}\le c_\gamma\norm{A(u)}^2_{C^\gamma}.\]
It is also clear that $A(u)\in C^\gamma_{per}$ implies that the same is true of $A(u)^2$.  The result now follows from Lemma~\ref{Fourierdecay}.

\noindent(b) Cauchy-Schwarz shows the left-hand side of \eqref{Wasserstein} is bounded above by $\kappa'_1\norm{u-v}_2^\alpha$ and so \eqref{L2cont} follows.  By \eqref{Wasserstein} and Lemma~\ref{Fourierdecay} we have
\begin{align*}
\norm{A(u+he_k)-A(u)}_2&\le \kappa'_1\sup_{\vp\in C^{\beta/\alpha}_{per},\norm{\vp}_{C^{\beta/\alpha}}\le 1}|h|^\alpha|\langle e_k,\vp\rangle|^\alpha\\
&\le \kappa'_1|h|^\alpha\left(\frac{c_1(\beta/\alpha)}{1+k^{\beta/\alpha}}\right)^\alpha\\
&\le \kappa'_1c_2(\alpha,\beta)|h|^\alpha(1+k)^{-\beta}.
\end{align*}
\end{proof}

\begin{proof}[Proof of Theorem \ref{mainSPDE2}]
This is an immediate consequence of Theorem \ref{mainSPDE} and
Proposition \ref{prop:Wasscond}.
\end{proof}

\begin{proof}[Proof of Corollary \ref{Examples}]
By our assumptions on $f$, $A(u)(x)$ is bounded above and below
by positive constants, is in $C^\gamma_{per}$, and is bounded
in $C^\gamma$ norm uniformly in $u$. By our assumptions on $f$,
\begin{align*}
|A(u)(x)-A(v)(x)|&\leq c_1\sum_{j=1}^n |\angel{u-v,\vp_j}|^\al\\
&\leq c_2 \sup_{\vp\in C_{per}^{\ol \beta}, \norm{\vp}_{C^{\ol \beta}}\le 1}
|\angel{u-v, \vp}|^\al.
\end{align*}
Squaring and integrating over $[0,1]$ shows that $A$ satisfies
\eqref{Wasserstein1} and we can then apply Theorem \ref{mainSPDE2}.
\end{proof}

\begin{proof}[Proof of Corollary \ref{convoleg}] 
We verify the hypotheses of Theorem~\ref{mainSPDE}. Use \eqref{convA} to define $A(u)(x)$ for all $x$ in the line, not just $[0,1]$. It is clear
 that for any $u\in C[0,1]$, $A(u)$ is then an even $C^\infty$ function on $\R$ with period two, and so in particular
\[A:C[0,1]\to C^\infty_{per}\equiv\cap_kC^k_{per}.\]
Moreover the $k$th derivative of $A(u)(x)$ is bounded uniformly in $x$ and $u$.
If we choose $\gamma$ and $\beta$ large enough so that the conditions of Theorem~\ref{mainSPDE} are satisfied, we see from the above and Proposition~\ref{prop:Wasscond}(a) that \eqref{Fdecay} holds.

Turning to the boundedness condition \eqref{Abnd}, we have 
\[A(u)(x)\ge a\int\psi(x-y)\,dy=a\norm{\psi}_1>0,\]
and the corresponding upper bound is similar.

For \eqref{Fholder}, note that by the H\"older continuity of $f$, 
\begin{align}
\nn\sup_{x\in[0,1]}&|A(u+he_k)(x)-A(u)(x)|\\
\nn&\le \sup_{x\in[0,1]}\Bigl|\int\psi(x-y)[f(\phi_1*(\ol{u+he_k})(y),\dots,\phi_n*(\ol{u+he_k})(y))\\
\nn&\qq\qq-f(\phi_1*\ol u(y),\dots,\phi_n*\ol u(y))]\,dy\Bigr|\\
\label{Abound1}&\le \norm{\psi}_1c_f\sup_{y\in\R,j\le n}|h|^\alpha |\phi_j*e_k(y)|^\alpha.
\end{align}
In the last inequality we use the linearity of $u\to\ol u$ and $\ol e_k=e_k$.  Since $\phi_j$ is smooth with compact support, its Fourier transform 
decays faster than any power, and so 
\bee\label{sclass}
\left|\int\phi_j(w)e^{-iw2\pi x}\,dw\right|\le c_{\beta/\alpha,j}(1+|2\pi x|)^{-\beta/\alpha}\qq\hbox{for all }x.
\eee
Now for $k\ge 0$,
\begin{align*}
|\phi_j*e_k(y)|&\le \sqrt{2}\left|\int\phi_j(y-z)\cos(2\pi kz)\,dz\right|\\
&\le \sqrt{2}\left|\int\phi_j(y-z)e^{i2\pi kz}\,dz\right|\\
&=\sqrt{2}\left|\int\phi_j(w)e^{-i2\pi kw}\,dw\,e^{i2\pi ky}\right|\\
&\le \sqrt{2}c_{\beta/\alpha,j}(1+k)^{-\beta/\alpha},
\end{align*}
by \eqref{sclass}.
Use this in \eqref{Abound1} to obtain \eqref{Fholder}.  Finally, the proof of \eqref{L2cont} is easy and should be clear from \eqref{Abound1}.   The result now follows from Theorem~\ref{mainSPDE}.\end{proof}

\appendix

\section{Proofs of linear algebra results}\label{sec:linalgpfs}

We give the proofs of some of the linear algebra results of 
Section \ref{sec:linalg}.

\begin{proof}[Proof of Lemma \ref{tildeabnd}] 
Our definitions imply
\begin{align*}
\langle \wt a(t)x,x\rangle&=\langle G(t)^{1/2}\int_0^tE(s)aE(s)\,ds\,G(t)^{1/2}x,x\rangle\\
&=\int_0^t\langle aE(s)G(t)^{1/2}x,E(s)G(t)^{1/2}x\rangle\,ds\\
&\ge \Lambda_0\int_0^t\langle E(s)G(t)^{1/2}x,E(s)G(t)^{1/2}x\rangle\,ds,
\end{align*}
by the hypotheses on $a$.  The right side is 
\[\Lambda_0\int_0^t\sum_ie^{-2\lam_is}\frac{2\lam_i}{1-e^{-2\lam_it}}|x_i|^2\,ds=\Lambda_0\Vert x\Vert^2.\]
The upper bound  is similar. The bounds on $a(t)$ are a reformulation of Lemma~\ref{Q1} and the analogous upper bound.
\end{proof}

\begin{proof}[Proof of Lemma \ref{Q4}]
 The first inequality in \eqref{tildeab} follows from \eqref{Schur}.  The second inequality holds since\begin{align*}
 \norm{\wt a(t)&-\wt b(t)}_s\\
 &=\norm{G(t)^{1/2}(a(t)-b(t))G(t)^{1/2}}_s\\
&=\sup_i\sum_jG_{ii}(t)^{1/2}\Bigl(\frac{1-e^{(\lam_i+\lam_j)t}}{\lam_i+\lam_j}\Bigr)G_{jj}(t)^{1/2}|a_{ij}-b_{ij}|\\
&\le \sup_i\sum_j|a_{ij}-b_{ij}|=\norm{a-b}_s,
\end{align*}
where Lemma~\ref{jaff2L1} is used in the last line and symmetry is used in the
next to last line.

Turning to \eqref{normbnd}, we have
\begin{align} \nonumber\norm{\wt A(t)-\wt B(t)}&=\norm{\wt A(t)(\wt b(t)-\wt a(t))\wt B(t)}\\
\label{specbnd1}&\le \norm{\wt A(t)}\norm{\wt B(t)}\norm{\wt b(t)-\wt a(t)}.
\end{align}
The lower bound on $\wt a(t)$ (and hence $\wt b(t)$) in Lemma~\ref{tildeabnd}  implies that
\[\norm{\wt A(t)}\norm{\wt B(t)}\le \Lambda_0^{-2}.\]
Use this and \eqref{tildeab} in \eqref{specbnd1} to derive \eqref{normbnd}. \eqref{LAE3} is then immediate.
\end{proof}

\begin{proof}[Proof of Lemma \ref{Q5}]
We write
\begin{align}
\frac{\det \wt b(t)}{\det \wt a(t)}&=\det (\wt b(t)\wt A(t))=\det(I+(\wt b(t)\wt A(t)-I))\label{LAE7}\\
&=\det(I+(\wt b(t)-\wt a(t))\wt A(t)).\nn
\end{align}
Clearly
\bee\label{LAE7.5}
\norm{I+(\wt b(t)-\wt a(t))\wt A(t)}\leq \norm{I}+\norm{\wt b(t)-\wt a(t)}\, \norm{\wt A(t)}.
\eee
Use the lower bound on $\wt a(t)$ in Lemma~\ref{tildeabnd} to see that $\norm{\wt A(t)}\le \Lambda_0^{-1}$, and then use \eqref{tildeab} in the above to conclude that
\begin{align*}
\norm{I+(\wt b(t)-\wt a(t))\wt A(t)}&\leq 1+ \Lambda_0^{-1} \Vert a-b\Vert_s.
\end{align*}

 Hence from \eqref{LAE7} and \eqref{detbd}
we have the bound
\begin{align*}
\Big|\frac{\det \wt b(t)}{\det \wt a(t)}\Big| &\leq \norm{I+(\wt b(t)-\wt a(t))\wt A(t)}^m\\
&\leq \Bigl(1+\Lambda_0^{-1}\norm{a-b}_s\Bigr)^m\\
&\leq e^{\Lambda_0^{-1} m\norm{a-b}_s}.
\end{align*}
Observe that $\wt a(t)$ and $\wt b(t)$ are positive definite, so
$\det \wt a(t)$ and $\det \wt b(t)$ are positive real numbers.
We now use the inequality $e^x\leq 1+xe^x$ for $x>0$ to
obtain 
\[\frac{\det \wt b(t)}{\det \wt a(t)}\le 1+\theta e^\theta.\]

Reversing the roles of $a$ and $b$, 
$$\frac{\det \wt a(t)}{\det \wt b(t)} \leq 1+\theta e^\theta,$$
and so,
$$\frac{\det \wt b(t)}{\det \wt a(t)}\geq \frac{1}{1+\theta e^\theta}
\geq 1-\theta e^\theta.$$ \end{proof}

\begin{proof}[Proof of Proposition \ref{Q5A}]
Using the inequality 
\bee\label{FC1}
|e^x-1|\leq |x|e^{(x^+)},
\eee
we have from Lemma~\ref{Q4},
$$\Big| e^{-\langle w,(\wt A(t)-\wt B(t))w\rangle/2}-1\Big|\leq \phi e^{\phi}.$$
Using the inequalities 
$$|1-\sqrt x|\leq |1-x|, \qq x\geq 0,$$
and
$$|xy-1|\leq |x|\, |y-1|+|x-1|, \qq x,y\geq 0,$$
the proposition now follows by Lemmas \ref{Q4} and \ref{Q5} with \hfil\break
$c_1=e^M(1+Me^M)^{1/2}$.
\end{proof}

\begin{proof}[Proof of Lemma \ref{LAL100}]
Let $\delta_{ij}$ be 1 if $i=j$ and 0 otherwise. If $i,j\leq m$,
then
\begin{align*}
\sum_{k=1}^m b_{ik}B_{kj}&=\sum_{k=1}^m a_{ik}A_{kj}-\sum_{k=1}^m a_{ik}
\frac{A_{k,m+1}A_{j,m+1}}{A_{m+1,m+1}}\\
&=\sum_{k=1}^{m+1} a_{ik}A_{kj}-a_{i,m+1}A_{m+1,j}
-\sum_{k=1}^{m+1} a_{ik}\frac{A_{k,m+1}A_{j,m+1}}{A_{m+1,m+1}}\\
&\qq +a_{i,m+1}\frac{A_{m+1,m+1}A_{j,m+1}}{A_{m+1.m+1}}\\
&=\delta_{ij}-\frac{\delta_{i,m+1}A_{j,m+1}}{A_{m+1,m+1}}=\delta_{ij}.
\end{align*}
The last equality holds because $i\leq m$.
\end{proof}

\section{Proof of existence}\label{A-existence}

We give here the proof of existence to a solution to
\eqref{SPDEeq}.

\begin{proof}
Let $X^n(t)=\angel{u_t, e_n}$. 
  By Theorem~\ref{SDEth2}  there is a unique continuous $\ell^2$-valued 
solution $X$ to \eqref{system} with $\lambda_n=n^2\pi^2/2$, where $a$ is constructed from $A$ as above.  If
\bee\label{udefn} u(s,x)=\sum_{n=0}^\infty X^n(s)e_n(x),
\eee
then the continuity of $X(t)$ in $\ell^2$ shows that the above series converges in $L^2[0,1]$ for all $s\ge 0$ a.s. and $s\to u(s,\cdot)$ is a continuous $L^2$-valued stochastic process.  It follows from \eqref{system} that
\bee\label{Xndecomp} X^n(t)=\langle u_0,e_n\rangle +M_n(t)-\lambda_n\int_0^tX^n(s)\,ds,
\eee
where each $M_n$ is a continuous square integrable martingale such that
\bee\label{Mnsqfn}
\langle M_m,M_n\rangle_t=\int_0^ta_{mn}(X_s)\,ds=\int_0^t\int_0^1A(u_s)(y)^2e_m(y)e_n(y)\,dy\,ds.
\eee

We next verify that $u$ satisfies \eqref{s10-E11}.  Let $\phi\in C^2[0,1]$ satisfy $\phi'(0)=\phi'(1)=0$.  Note that 
\bee\label{uNconv}u^N(s,x)\equiv\sum_{n=0}^N X^n(s)e_n(x)\to u(s,x)\hbox{ in }L^2[0,1] \eee
 as $N\to\infty$ for all $s\ge 0$ a.s.
By \eqref{Xndecomp} we have 
\begin{align}
\langle u^N_t,\phi\rangle &=\sum_{n=0}^N\langle u_0,e_n\rangle\langle\phi,e_n\rangle+\sum_{n=0}^N M_n(t)\langle\phi,e_n\rangle -\int_0^t\sum_{n=1}^N\lambda_nX^n(s)\langle e_n,\phi\rangle\, ds\nn\\
&=I^N_1(\phi)+M^N_t(\phi)+V_t^N(\phi).
\label{NSPDE} 
\end{align}
Parseval's equality shows that 
\bee\label{ICconv}
\lim_{N\to\infty}I_1^N(\phi)=\langle u_0,\phi\rangle.
\eee
Integrating by parts twice in $\langle \phi,e_n\rangle$, and using the boundary conditions of $\phi$, we find that 
\[
V_t^N(\phi)=\int_0^t\langle u_s^N,\phi''/2\rangle\,ds.\]
Now $\sup_{s\le t}\norm{u^N_s}_2\le \sup_{s\le t}\norm{u_s}_2<\infty$ for all $t>0$ and so by dominated convergence we see from the above and \eqref{uNconv} that 
\bee\label{VNconv}
\lim_{N\to\infty}V_t^N(\phi)=\int_0^t\langle u_s,\phi''/2\rangle\,ds\ \hbox{ for all }t\ge 0\qq \mbox{\rm a.s.}
\eee

If $N_2>N_1$, then by \eqref{Mnsqfn} and \eqref{Abnd} we have 
\begin{align*}
\langle (M^{N_2}-M^{N_1})(\phi)\rangle_t&=\int_0^t\int_0^1 A(u_s)(y)^2\Bigl(\sum_{n=N_1+1}^{N_2}\langle e_n,\phi\rangle e_n(y)\Bigr)^2\,dy\,ds\\
&\le \kappa_2^{-2}\int_0^t\int_0^1\Bigl(\sum_{n=N_1+1}^{N_2}\langle e_n,\phi\rangle e_n(y)\Bigr)^2\,dy\,ds\\
&= \kappa_2^{-2}t\sum_{n=N_1+1}^{N_2}\langle e_n,\phi \rangle ^2\to 0\hbox{ as }N_1,N_2\to\infty.
\end{align*}
It follows that there is a continuous $L^2$ martingale $M_t(\phi)$ such that for any $T>0$, 
\[\sup_{t\le T}|M^N_t(\phi)-M_t(\phi)|\to 0 \hbox{ in }L^2,\]
and
\begin{align*} \langle M(\phi)\rangle _t&=L^1-\lim_{N\to\infty} \langle M^N(\phi)\rangle_t\\
&=L^1-\lim_{N\to\infty} \int_0^t\int_0^1A(u_s)(y)^2\Bigl(\sum_0^N\langle e_n,\phi\rangle e_n(y)\Bigr)^2\,dy\,ds\\
&=\int_0^t\int_0^1 A(u_s)(y)^2\phi(y)^2\,dy\,ds.
\end{align*}
Since $A$ is bounded, $M$ is an orthogonal martingale measure in the sense of Chapter 2 of \cite{walsh} and so is a 
continuous orthogonal martingale measure in the sense of
Chapter 2 of \cite{walsh}.  This (see especially Theorem 2.5 and
Proposition 2.10 of \cite{walsh}) and the fact that A is bounded
below
 means one can define a white noise $\dot W$ on $[0,1]\times [0,\infty)$ on the same probability space,
so that 
\[M_t(\phi)=\int_0^t\int_0^1 A(u_s)(y)\phi(y)\, dW_{s,y}\ \ \hbox{ for all }t\ge 0\ \mbox{\rm a.s.}\ \ \hbox{ for all }\phi\in L^2[0,1].\]
Therefore we may take limits in \eqref{NSPDE} and use the above, together with \eqref{ICconv} and \eqref{VNconv}, to conclude that $u$ satisfies \eqref{s10-E11}.

It remains to show that there is a jointly continuous version of $u(t,x)$.  Note first that 
\bee\label{expform}
X^n(t)=e^{-\lambda_n t}\langle u_0,e_n\rangle+\int_0^t e^{-\lambda_n(t-s)}\,dM_n(s),
\eee
and so
\begin{align}\label{uNexpform}
u^N(t,x)&=\sum_{n=0}^N e^{-\lambda_n t}\langle u_0,e_n\rangle e_n(x)+\sum_{n=0}^N\int_0^t e^{-\lambda_n(t-s)}\,dM_n(s)\, e_n(x)\\
\nn&\equiv \hat u^N(t,x)+\tilde u^N(t,x).
\end{align}
Let $p(t,x,y)$ denote the fundamental solution of $\frac{\partial p}{\partial t}=\frac{1}{2}\frac{\partial^2}{\partial x^2}p(t,x,y)$ 
on $[0,1]$ with Neumann boundary conditions, and let $P_t$ 
be the corresponding semigroup.  
By Mercer's theorem (see, e.g., Theorem 11 of Chapter 30 of \cite{Lax}),
\[p(t,x,y)=\sum_{n=0}^\infty e^{-\lambda_n t}e_n(x)e_n(y),\]
where the series converges uniformly on $t\ge \eps$, $x,y\in[0,1]$ for every $\eps>0$. It follows that 
\[\hat u^N_t(x,y)\to P_tu_0(x)\hbox{ for all }t>0,\ x\in[0,1].\]
 
An easy $L^2(\P)$ convergence argument using square functions shows there is a jointly measurable random field $\{\tilde u(t,x):t\ge 0, x\in[0,1]\}$ so that $\tilde u(0,\cdot)\equiv0$ and 
\[\tilde u^N(t,x)\to \tilde u(t,x)\hbox{ in }L^2(\P) \hbox{ uniformly in }(t,x),\]
and so for some subsequence 
\begin{equation}\label{pwconv}\tilde u^{N_k}(t,x)\to \tilde u(t,x) \ \mbox{\rm a.s.}\hbox{ for each }(t,x).
\end{equation}
So let $N=N_k\to\infty$ in \eqref{uNexpform} to conclude
\[\lim_k u^{N_k}(t,x)=P_tu_0(x)+\tilde u(t,x)\ \ \mbox{\rm a.s.}\quad\hbox{for all }t>0,\ x\in[0,1].
\]
It now follows easily from \eqref{uNconv} that
\bee\label{uGff}
u(t,x)=P_tu_0(x)+\tilde u(t,x)\qq\hbox{ a.a. }x,\ \P-\mbox{\rm a.s.} \hbox{ for all }t\ge0,
\eee
where the equality holds trivially for all $x$ if $t=0$.  

Clearly $P_tu_0(x)$ is jointly continuous by the continuity of $u_0$, and so we next show there is a continuous version of $\tilde u(t,x)$. 
Let $0\le s<t$, choose reals $x<y$ and fix $q\ge 1$.  Our constants $c_i$ below may depend on $q$ but not $s,t,x,y$.  By Burkholder's inequality and \eqref{Mnsqfn} we have
\begin{align*}
\E&(|\tilde u^N(t,x)-\tilde u^N(s,y)|^q)\\
&\le c_1\Bigl(\E(|\tilde u^N(t,x)-\tilde u^N(t,y)|^q)+\E\Bigl(\Bigl|\sum_{n=0}^N\int_s^te^{-\lambda_n(t-v)}dM_n(v)e_n(y)\Bigr|^q\Bigr)\\
&\quad +\E\Bigl(\Bigl|\sum_{n=0}^N\int_0^s[e^{-\lambda_n(t-v)}-e^{-\lambda_n(s-v)}]dM_n(v)e_n(y)\Bigr|^q\Bigr)\\
&\le c_2\Bigl\{\E\Bigl(\Bigl[\int_0^t\int_0^1A(u_v)(z)^2\Bigl[\sum_{n=1}^Ne^{-\lambda_n(t-v)}e_n(z)(e_n(x)-e_n(y))\Bigr]^2\,dz\,dv\Bigr]^{\frac{q}{2}}\Bigr)\\
&\ \ +\E\Bigl(\int_s^t\int_0^1A(u_v)(z)^2\Bigl[\sum_{n=0}^Ne^{-\lambda_n(t-v)}e_n(z)e_n(y)\Bigr]^2\,dz\,dv\Bigr]^{\frac{q}{2}}\Bigr)\\
&\ \ +\E\Bigl(\Bigl[\int_0^s\int_0^1A(u_v)(z)^2\\
&\phantom{\ \ +\E\Bigl(\Bigl[\int_0^s\int_0^1}\times\Bigl[\sum_{n=1}^N(e^{-\lambda_n(t-u)}-e^{-\lambda_n(s-u)})e_n(z)e_n(y)\Bigr]^2\,dz\,dv\Bigr]^{\frac{q}{2}}\Bigr)\Bigr\}.
\end{align*}
Next use the uniform boundedness of $A(u_v)(z)$ (by \eqref{Abnd}) and the fact that $\{e_n\}$ is an orthonormal system in $L^2([0,1])$ to bound the above by
\begin{align}
\nonumber &c_3\Bigl\{\Bigl(\int_0^t\sum_{n=1}^Ne^{-2\lambda_n(t-v)}(e_n(x)-e_n(y))^2\,dv\Bigr)^{\frac{q}{2}}+\Bigl(\int_s^t\sum_{n=0}^Ne^{-2\lambda_n(t-v)}e_n(y)^2\,dv\Bigr)^{\frac{q}{2}}\\
\nonumber&\quad+\Bigl(\int_0^s\sum_{n=1}^N(e^{-\lambda_n(t-v)}-e^{-\lambda_n(s-v)})^2e_n(y)^2\,dv\Bigr)^{\frac{q}{2}}\Bigr\}\\
\nonumber&\le c_4\Bigl\{\Bigl(\sum_{n=1}^N(2\lambda_n)^{-1}(e_n(x)-e_n(y))^2\Bigr)^{\frac{q}{2}}+\Bigl(\sum_{n=0}^N(t-s)\wedge\frac{1}{2\lambda_n}\Bigr)^{\frac{q}{2}}\\
\nonumber&\quad+\Bigl(\sum_{n=1}^N(1-e^{-\lambda_n(t-s)})^2\lambda_n^{-1}\Bigr)^{\frac{q}{2}}\Bigl\}\\
\label{burk1}&\equiv c_4\Bigl\{T_1+T_2+T_3\Bigr\}.
\end{align}

Let $\delta\in(0,1/2)$. For $T_1$, use the fact that 
\[|e_n(x)-e_n(y)|\le 8[n|x-y|\wedge 1]\]
to see that 
\begin{align}
\label{burk2}T_1&\le c_5\Bigl[\sum_{n=1}^Nn^{-2}[(n|x-y|)^2\wedge 1]\Bigr]^{q/2}\\
\nonumber&\le c_5\Bigl[\sum_{n=1}^Nn^{-2}n^{1-\delta}|x-y|^{1-\delta}\Bigr]^{q/2}\le c_6(\delta)|x-y|^{(1-\delta)q/2}.
\end{align}

Elementary reasoning gives
\begin{align}
\nonumber T_2&\le c_4\Bigl[ |t-s|+\sum_{n=1}^N(t-s)^{\frac{1}{2}-\delta}(1/(2\lambda_n))^{\frac{1}{2}+\delta}\Bigr]^{\frac{q}{2}}\\
\label{burk3} &\le c_7(\delta)[|t-s|^{\frac{q}{2}}+|t-s|^{(\frac{1}{2}-\delta){\frac{q}{2}}}\Bigr],
\end{align}
and
\begin{align}
\nonumber T_3&\le \Bigl[\sum_{n=1}^N[(\lambda_n|t-s|)\wedge 1]^2\lambda_n^{-1}\Bigr]^{\frac{q}{2}}\\
\label{burk4}&\le \Bigl[\sum_{n=1}^N\lambda_n^{\frac{1}{2}-\delta-1}|t-s|^{\frac{1}{2}-\delta}\Bigr]^{\frac{q}{2}}\le c_8|t-s|^{(\frac{1}{2}-\delta)\frac{q}{2}}.
\end{align}

By using \eqref{burk2}-\eqref{burk4} in \eqref{burk1} we may conclude that for all $T>0$ there is a $c(T,q,\delta)$ so that for $0\le s\le t\le T$ and $x,y\in\R$, 
\[\E(|\tilde u^N(t,x)-\tilde u^N(s,y)|^q)\le c(T,q)[|x-y|^{(1-\delta){\frac{q}{2}}}+|t-s|^{(\frac{1}{2}-\delta){\frac{q}{2}}}].\]
By Fatou's Lemma and \eqref{pwconv} the same upper bound is valid for \hfil \break
$\E(|\tilde u(t,x)-\tilde u(s,y)|^q)$.  
Kolmogorov's continuity criterion (see, for example, Theorem (2.1) in Chapter I of \cite{RY}) shows there is a jointly continuous version of $\tilde u$ on $\R_+\times\R$.

We have shown that there is a jointly continuous process $v(t,x)$ such that 
\[u(t,x)=v(t,x)\hbox{ a.a.  }x \hbox{ for all }t\ge 0 ,\hbox{ and }v(0,\cdot)=u_0(\cdot),\ \P-\mbox{\rm a.s.}\]
Here the continuity in $t$ in $L^2$ of both sides allows us to find a
null set independent of $t$.
As $A$ has been continuously extended to a map from $L^2$ to $L^2$, we have 
$A(u_s)=A(v_s)$ in $L^2[0,1]$  for all $s\ge 0$ a.s. and so the white noise integral in \eqref{s10-E11} remains unchanged if $u$ is replaced by $v$.  It now follows easily that \eqref{s10-E11} remains valid
with $v$ in place of $u$. Therefore $v$ is the required continuous $C[0,1]$-valued solution of \eqref{SPDEeq}. \end{proof}

\newpage

\ni {\bf Richard F. Bass}\\
Department of Mathematics\\
University of Connecticut \\
Storrs, CT 06269-3009, USA\\
{\tt r.bass@uconn.edu}
\ms

\ni {\bf Edwin A. Perkins}\\
Department of Mathematics\\
University of British Columbia\\
Vancouver, B.C. V6T 1Z2, Canada\\
{\tt perkins@math.ubc.ca}
\ms


\begin{thebibliography}{99}

\bibitem{ABGP}
S.R. Athreya, R.F.  Bass, M. Gordina, and E.A. Perkins,
Infinite dimensional stochastic differential equations of 
Ornstein-Uhlenbeck type. 
Stochastic Process. Appl. 116 (2006) 381--406.

\bibitem{Ba}
R.F. Bass,  Uniqueness in law for pure jump processes, Prob.
Th. Rel. Fields 79 (1988) 271--287.

\bibitem{Bas97}
R.F. Bass, Diffusions and Elliptic Operators, Berlin, Springer, 1997.

\bibitem{BP-inf} R.F. Bass and E.A. Perkins,
Countable systems of degenerate stochastic differential equations 
with applications to super-Markov chains. 
Electron. J. Probab. 9 (2004) 634--673.

\bibitem{BP-Bismut}
R.F. Bass and E.A. Perkins,
A new technique for proving uniqueness for martingale problems, In: From Probability to Geometry (I): Volume in Honor of the 60th
Birthday of Jean-Michel Bismut, 47-53.
Soci\'et\'e Math\'ematique de France, Paris, 2009.

\bibitem{CD}
P. Cannarsa and G. Da Prato,
Infinite-dimensional elliptic equations with H\"older-continuous coefficients,
Adv. Diff. Eq. 1 (1996), 425-452.

\bibitem{DZ}
G. Da Prato and J. Zabczyk,
Stochastic Equations in Infinite Dimensions, Cambridge, Cambridge University Press, 1992.

\bibitem{horn-johnson}
R.A. Horn and C.R. Johnson,
Matrix Analysis, Cambridge, Cambridge University Press, 1985.

\bibitem{jaffard}
S. Jaffard, Propri\'et\'es des matrices ``bien localis\'ees'' pr\`es de 
leur diagonale et quelques applications. 
 Ann. Inst. H. Poincar\'e Anal. Non Lin\'eaire 7 (1990), no. 5, 461--476. 

\bibitem{KX}
G. Kallianpur and J. Xiong, Stochastic Differential Equations in Infinite Dimensional
Spaces, Lect. Notes IMS V. 26, IMS, 1995.

\bibitem{Lax}
P.D. Lax, Functional Analysis, New York, Wiley, 2002.

\bibitem{men}
S. Menozzi, Parametrix techniques and martingale problems for some degenerate Kolmogorov equations, Elect. Comm. Probab. 16 (2011), 234--250.

\bibitem{MMP}
C. Mueller, L. Mytnik, and E. Perkins, 
Nonuniqueness for a parabolic SPDE with $\frac{3}{4}-\vep$-H\"older diffusion coefficients,
Preprint.

\bibitem{MP}
L. Mytnik and E. Perkins, Pathwise uniqueness for stochastic heat equations with Holder continuous coefficients: the white noise case, Prob. Th. Rel. Fields 149 (2011), 1--96.

\bibitem{RY}
D. Revuz and M. Yor, Continuous Martingales and Brownian Motion, Springer-Verlag, Berlin, 1991.

\bibitem{RW}
L.C.G. Rogers and D. Williams, Diffusions, Markov Processes and Martingales, V.2 It\^o Calculus,
2nd edition,
Cambridge, Cambridge University Press, 1994.

\bibitem{SV}
D.W. Stroock and S.R.S. Varadhan, Multidimensional Diffusion Processes,
Berlin, Springer, 1977.

\bibitem{walsh}
J. Walsh, 
An Introduction to Stochastic Partial Differential
Equations, 
 in {\em Ecole d'Et\'e de Probabilit\'es de Saint Flour 1984}, Lect.
Notes. in Math. 1180, Springer, Berlin, 1986.

\bibitem{YW}
T. Yamada and S. Watanabe,
On the uniqueness of solutions of
stochastic differential equations, 
 J. Math. Kyoto U. 11  (1971) 155--167.
 
\bibitem{Zamb}
L. Zambotti, An analytic approach to existence and uniqueness for martingale problems
in infinite dimensions.  Probab. Theory Rel. Fields 118 (2000), 147--168.

\bibitem{Zygmund}
A. Zygmund, Trigonometric Series, Vol. I, 3rd ed., Cambridge, Cambridge University
Press, 2002.

\end{thebibliography}
\end{document}